\documentclass[leqno,english]{amsart}

\usepackage{hyperref}   
\usepackage[square,numbers]{natbib}     
\usepackage{amsthm}     
\usepackage{amsmath}    
\usepackage{amssymb}    
\usepackage{amsxtra}    
\usepackage{eucal}      
\usepackage[bbgreekl]{mathbbol}   
\usepackage[all]{xy}    
\usepackage{graphics}   

\usepackage{tikz}
\usetikzlibrary{matrix}
\usetikzlibrary{shapes}
\usetikzlibrary{arrows}
\usetikzlibrary{calc,3d}
\usetikzlibrary{decorations,decorations.pathmorphing}
\usetikzlibrary{through}
\tikzset{ext/.style={circle, draw,inner sep=1pt},int/.style={circle,draw,fill,inner sep=1pt},nil/.style={inner sep=1pt}}
\tikzset{exte/.style={circle, draw,inner sep=3pt},inte/.style={circle,draw,fill,inner sep=3pt}}
\tikzset{diagram/.style={matrix of math nodes, row sep=3em, column sep=2.5em, text height=1.5ex, text depth=0.25ex}}
\tikzset{diagram2/.style={matrix of math nodes, row sep=0.5em, column sep=0.5em, text height=1.5ex, text depth=0.25ex}}
\usepackage{tikz-cd}

\theoremstyle{plain}

\swapnumbers                        
\newtheorem{prop}[subsubsection]{Proposition}

\newtheorem{thm}[subsubsection]{Theorem}
\newtheorem{cor}[subsubsection]{Corollary}

\numberwithin{equation}{subsubsection}

\title[Little discs operads, graph complexes (\dots)]{Little discs operads, graph complexes and Grothendieck--Teichm\"uller groups}

\date{November 29, 2018}

\author{Benoit Fresse}
\address{Univ. Lille, CNRS, UMR 8524 - Laboratoire Paul Painlev\'e, F-59000 Lille, France}
\email{Benoit.Fresse@univ-lille.fr}

\thanks{This research is supported in part by Labex ``CEMPI'' ANR-11-LABX-0007-01.}


\DeclareMathOperator{\kk}{\mathbb{k}}   
\DeclareMathOperator{\NN}{\mathbb{N}}   
\DeclareMathOperator{\FF}{\mathbb{F}}
\DeclareMathOperator{\ZZ}{\mathbb{Z}}   
\DeclareMathOperator{\QQ}{\mathbb{Q}}   
\DeclareMathOperator{\RR}{\mathbb{R}}   
\DeclareMathOperator{\CC}{\mathbb{C}}
\DeclareMathOperator{\DD}{\mathbb{D}}   
\DeclareMathOperator{\Sphere}{\mathbb{S}}   
\DeclareMathOperator{\hofib}{hofib}     

\DeclareMathOperator{\CCat}{\mathcal{C}}            
\DeclareMathOperator{\Mod}{\mathcal{M}\mathit{od}}  
\DeclareMathOperator{\BiMod}{\mathcal{B}\mathit{i}\mathcal{M}\mathit{od}}  
\DeclareMathOperator{\Top}{\mathcal{T}\mathit{op}}  
\DeclareMathOperator{\sSet}{\mathit{s}\mathcal{S}\mathit{et}}   
\DeclareMathOperator{\ComCat}{\mathcal{C}\mathit{om}}
\DeclareMathOperator{\Cat}{\mathcal{C}\mathit{at}}
\DeclareMathOperator{\Grd}{\mathcal{G}\mathit{rd}}
\DeclareMathOperator{\MT}{\mathit{MT}}
\DeclareMathOperator{\GC}{\mathit{GC}}      
\DeclareMathOperator{\HGC}{\mathit{HGC}}    

\DeclareMathOperator{\Hopf}{\mathcal{H}\mathit{opf}}    
\DeclareMathOperator{\Op}{\mathcal{O}\mathit{p}}    

\DeclareMathOperator{\dg}{\mathit{dg}}      
\DeclareMathOperator{\gr}{\mathit{gr}}      

\DeclareMathOperator{\Mor}{\mathtt{Mor}}    
\DeclareMathOperator{\Aut}{\mathtt{Aut}}    
\DeclareMathOperator{\Hom}{\mathtt{Hom}}    
\DeclareMathOperator{\BiDer}{\mathtt{BiDer}}    
\DeclareMathOperator{\Map}{\mathtt{Map}}    
\DeclareMathOperator{\Emb}{\mathtt{Emb}}    
\DeclareMathOperator{\Imm}{\mathtt{Imm}}    
\DeclareMathOperator{\GT}{\mathit{GT}}      
\DeclareMathOperator{\GRT}{\mathit{GRT}}
\DeclareMathOperator{\grt}{\mathfrak{grt}}
\DeclareMathOperator{\Gal}{\mathit{Gal}}

\DeclareMathOperator{\id}{\mathit{id}}      
\DeclareMathOperator{\unit}{\mathbb{1}}     
\DeclareMathOperator{\Ob}{\mathtt{Ob}}

\DeclareMathOperator{\FreeOp}{\mathbb{F}}   
\DeclareMathOperator{\Sym}{\mathbb{S}}      
\DeclareMathOperator{\LLie}{\mathbb{L}}     
\DeclareMathOperator{\UFree}{\mathbb{U}}
\DeclareMathOperator{\GLike}{\mathbb{G}}

\DeclareMathOperator{\ecell}{\mathbb{e}}        

\DeclareMathOperator{\SO}{\mathtt{SO}}
\DeclareMathOperator{\Sing}{\mathtt{Sing}}      
\DeclareMathOperator{\DGB}{\mathtt{B}}          
\DeclareMathOperator{\DGC}{\mathtt{C}}          
\DeclareMathOperator{\DGF}{\mathtt{F}}          
\DeclareMathOperator{\DGG}{\mathtt{G}}
\DeclareMathOperator{\DGH}{\mathtt{H}}          
\DeclareMathOperator{\DGHH}{\mathtt{HH}}        
\DeclareMathOperator{\DGR}{\mathtt{R}}
\DeclareMathOperator{\DGW}{\mathtt{W}}          
\DeclareMathOperator{\DGOmega}{\mathtt{\Omega}}          
\DeclareMathOperator{\DGGamma}{\mathtt{\Gamma}}
\DeclareMathOperator{\DGMC}{\mathtt{MC}}

\DeclareMathAlphabet{\mathsfit}{OT1}{cmss}{m}{sl}   
\DeclareMathOperator{\COp}{\mathsfit{C}}
\DeclareMathOperator{\DOp}{\mathsfit{D}}        
\DeclareMathOperator{\EOp}{\mathsfit{E}}        
\DeclareMathOperator{\FMOp}{\mathsfit{FM}}      
\DeclareMathOperator{\FOp}{\mathsfit{F}}        
\DeclareMathOperator{\KOp}{\mathsfit{K}}
\DeclareMathOperator{\MOp}{\mathsfit{M}}
\DeclareMathOperator{\NOp}{\mathsfit{N}}
\DeclareMathOperator{\POp}{\mathsfit{P}}        
\DeclareMathOperator{\QOp}{\mathsfit{Q}}        
\DeclareMathOperator{\ROp}{\mathsfit{R}}        
\DeclareMathOperator{\SOp}{\mathsfit{S}}
\DeclareMathOperator{\LambdaOp}{\mathsfit{\Lambda}}
\DeclareMathOperator{\PiOp}{\mathsfit{\Pi}}
\DeclareMathOperator{\OmegaOp}{\mathsfit{\Omega}}
\DeclareMathOperator{\CstOp}{\mathsfit{Cst}}

\DeclareMathOperator{\AsOp}{\mathsfit{As}}      
\DeclareMathOperator{\PoisOp}{\mathsfit{Pois}}      
\DeclareMathOperator{\GraphOp}{\mathsfit{Graphs}}   
\DeclareMathOperator{\ICG}{\mathsfit{ICGraphs}}
\DeclareMathOperator{\CoB}{\mathsfit{CoB}}
\DeclareMathOperator{\PaB}{\mathsfit{PaB}}
\DeclareMathOperator{\CD}{\mathsfit{CD}}
\DeclareMathOperator{\PaCD}{\mathsfit{PaCD}}
\DeclareMathOperator{\palg}{\mathfrak{p}}
\DeclareMathOperator{\galg}{\mathfrak{g}}

\begin{document}

\begin{abstract}
This paper is a survey on the homotopy theory of $E_n$-operads written for the new handbook of homotopy theory.
\textcolor[rgb]{1.00,0.00,0.00}{First draft.}
\end{abstract}

\maketitle


\setcounter{section}{0}

\section*{Introduction}

The operads of little discs (and the equivalent operads of little cubes)
were introduced in topology, in the works of Boardman--Vogt~\cite{BoardmanVogtAnnouncement,BoardmanVogt} and May~\cite{May},
for the recognition of iterated loop spaces.
We refer to the paper~\cite{CarlssonMilgram}, in the first handbook of algebraic topology, for an account of this theory.

The purpose of this article is to survey new applications of the little discs operads
which were motivated by the works of Kontsevich~\cite{KontsevichFormalityConj,KontsevichMotives}
and Tamarkin~\cite{Tamarkin}
on the deformation-quantization of Poisson manifolds
in mathematical physics on the one hand,
and by the Goodwillie--Weiss embedding calculus in topology~\cite{GoodwillieWeissEmbeddings,WeissEmbeddings}
on the other hand.
Let us give a brief overview of the connection between the little discs operads and these subjects
before explaining the objectives of this paper with more details.

We will recall the precise definition of the little discs operads later on.
For the moment, we just need to know that an operad (in topological spaces, in a category of modules, {\dots})
is a structure which governs a category of algebras. The operad of little $n$-discs, denoted by $\DOp_n$,
where $n = 1,2,\dots$ is a dimension parameter, is a particular instance of an operad defined in the category of topological spaces.

\medskip
The deformation--quantizations of a Poisson manifold $M$ are determined by solutions of a Maurer--Cartan equation
in a differential graded Lie algebra of polydifferential operators
on $M$.
The existence of deformation--quantizations can be deduced from the observation that this differential graded Lie algebra
of polydifferential operators is quasi-isomorphic to a differential graded Lie algebra
of polyvector fields,
which governs deformations of Poisson structures.
In the algebraic setting and for an affine space, this statement is equivalent to the claim that the Hochschild cochain complex
of a polynomial algebra $L = \DGC^*_{\DGHH}(\kk[x_1,\dots,x_n])$,
where $\kk$ is a field of characteristic zero, is quasi-isomorphic to its cohomology (is formal) as a differential graded Lie algebra.
Indeed, the Hochschild cochain complex can be regarded as an algebraic analogue of the Lie algebra of polydifferential operators of manifolds,
while the Hochschild-Kostant-Rosenberg theorem asserts that the Hochschild cohomology of a polynomial
algebra $\DGHH^*(\kk[x_1,\dots,x_n]) = \DGH^*(\DGC^*_{\DGHH}(\kk[x_1,\dots,x_n]))$
is identified with an algebra of vector fields (in the algebraic sense) $T = \kk[x_1,\dots,x_n,\xi_1,\dots,\xi_n]$,
where we set $\xi_i = \partial/\partial x_i$.
The counterpart of the Lie bracket of polydifferential operations for the Hochschild cochain complex is given by a Lie algebra structure
which was defined by Gerstenhaber in his work in deformation theory (see~\cite{Gerstenhaber}).
In~\cite{KontsevichFormalityConj}, Kontsevich gives a direct proof of the formality of the Hochschild cochain complex $L = \DGC^*_{\DGHH}(\kk[x_1,\dots,x_n])$
by making explicit an $L_{\infty}$-morphism that realizes the cohomology isomorphism
of the Hochschild-Kostant-Rosenberg theorem at the cochain level
(an $L_{\infty}$-morphism is a homotopy version of the notion of a morphism of Lie algebras).

In~\cite{KontsevichMotives,Tamarkin}, Kontsevich and Tamarkin give new proofs of the formality of the Hochschild cochain complex.
The idea is that the Lie algebra structure of the Hochschild cochain complex can be integrated
in the structure of a Poisson algebra up to homotopy and therefore, one can compare homotopy Poisson algebra structures
rather than just Lie algebra structures.
Indeed, one can prove that the Hochschild cohomology of a polynomial algebra $\DGHH^*(\kk[x_1,\dots,x_n])$
is intrinsically formal as a homotopy Poisson algebra:
a natural cohomology theory, which governs the obstructions to constructing formality quasi-isomorphisms,
vanishes when we take the full strongly homotopy Poisson algebra structure
of the Hochschild cohomology into account (and not only the Lie algebra structure).
By using this method, we can actually establish the general result that any homotopy Poisson algebra whose cohomology is isomorphic to the Poisson algebra
of polyvectorfields $T = \kk[x_1,\dots,x_n,\xi_1,\dots,\xi_n]$ (like the Hochschild cochain complex $L = \DGC^*_{\DGHH}(\kk[x_1,\dots,x_n])$)
is quasi-isomorphic to this object $T = \kk[x_1,\dots,x_n,\xi_1,\dots,\xi_n]$
as a homotopy Poisson algebra. (In fact, we also need to assume that our objects are equipped with a extra action of the group
of affine transformations of the affine space $\kk^n$
to get this rigidy result,
but this assumption is also fulfilled in the case of the Hochschild cochain complex.)
In what follows, we refer to this second approach of the proof of the existence of deformation-quantizations as the Kontsevich--Tamarkin approach
in order to distinguish this second generation of proofs from the Kontsevich initial construction.

In order to go further into our explanations, we should specify that the Lie bracket
of the Hochschild cochain complex (and the Lie bracket of polydifferential operators equivalently)
decreases cohomological degrees by one. We explicitly have $[-,-]: \DGC_{\DGHH}^i(A)\otimes\DGC_{\DGHH}^j(A)\rightarrow\DGC_{\DGHH}^{i+j-1}(A)$,
for all degrees $i,j\in\NN$.
Thus, in our study, we actually consider a form of graded Poisson structure where the Poisson bracket
has cohomological degree $-1$ (equivalently, has homological degree $1$).
In what follows, we generally use the expression $2$-Poisson algebra to refer to such a graded Poisson structure,
and more generally we call $n$-Poisson algebra a graded Poisson algebra
equipped with a bracket of cohomological degree $1-n$ (equivalently, of homological degree $n-1$). The name Gerstenhaber algebra is also used
in the literature for this category of graded Poisson algebras
when $n=2$.
The category of $n$-Poisson algebras is associated to an operad in graded modules $\PoisOp_n$
which can actually be identified with the singular homology of the topological operad
of little $n$-discs $\DGH_*(\DOp_n) = \DGH_*(\DOp_n,\kk)$.
The application of the little discs operad in deformation-quantization
comes from this relationship:
\begin{equation*}
\PoisOp_n\cong\DGH_*(\DOp_n).
\end{equation*}

The proof that the Hochschild cochain complex is equipped with the structure
of a Poisson algebra up to homotopy can actually be divided in two statements which are interesting for their own.
To formulate the ideas of these constructions, we have to consider the class of $E_n$-operads, which consists of the operads $\EOp_n$
that are (weakly) homotopy equivalent to the little $n$-discs operad
$\EOp_n\sim\DOp_n$
rather than the singled out object $\DOp_n$.
To be more precise, the notion of an $E_n$-operad makes sense in many contexts.
For the applications to the deformation-quantization problem,
we consider the class of $E_n$-operads in differential graded modules which consists of the operads
that are quasi-isomorphic to the operad of singular chains
on the little $n$-discs operad $\DGC_*(\DOp_n) = \DGC_*(\DOp_n,\kk)$,
where we take the ground ring of our module category $\kk$
as coefficients.

First, one can prove that the Hochschild cochain complex naturally inherits an action of an $E_2$-operad.
This statement is known as the Deligne conjecture in the literature and holds without any assumption
on the ground ring (we refer to~\cite{KontsevichSoibelman,McClureSmithFirst}
for the first proofs of this result).
Then one can prove that the operad of singular chains on the little $2$-discs operad $\DGC_*(\DOp_2) = \DGC_*(\DOp_2,\kk)$
is quasi-isomorphic to the operad of $2$-Poisson algebras $\PoisOp_2$
as an operad in differential graded modules
when we pass to rational coefficients $\kk = \QQ$.
This result, which can be interpreted as a formality result for the singular chains on the little $2$-discs operad,
implies that the category of algebras over any model of $E_2$-operad
is equivalent to the category of $2$-Poisson algebras
up to homotopy. We just apply this equivalence and the result of the Deligne conjecture
in order to get that the Hochschild cochain complex can be identified with an object
of the category of homotopy $2$-Poisson algebras,
so that we can proceed with the Kontsevich--Tamarkin approach
of the construction of deformation-quantizations.
In passing, we get that we can parameterize deformation-quantization functors
by homotopy classes of formality quasi-isomorphisms
for the operad of little $2$-discs.

We will see that we can associate a formality quasi-isomorphism for the little $2$-discs operad to any Drinfeld associator,
a notion introduced by Drinfeld to define universal deformation functors in quantum group theory (see~\cite{Drinfeld}).
The set of Drinfeld's associators inherits an action of the (rational) Grothendieck--Teichm\"uller group $\GT(\QQ)$.
This object $\GT(\QQ)$ was also defined by Drinfeld in the of quantum group theory,
but the idea of the definition of the Grothendieck--Teichm\"uller group
goes back to the Grothendieck program
in Galois theory (see~\cite{Grothendieck}).
In any case, the correspondence between deformation-quantization functors, formality quasi-isomorphisms for the little $2$-discs operad
and Drinfeld's associators, which results from the Kontsevich--Tamarkin approach,
hints that the set of deformation-quantization functors of the affine space
inherits an action of the Grothendieck--Teichm\"uller group.
This action has been made explicit by Thomas Willwacher and his collaborators
by using a graph complex description
of the Grothendieck--Teichm\"uller Lie algebra (see~\cite{WillwacherGraphs}).
This correspondence between the Grothendieck--Teichm\"uller Lie algebra
and graph complexes is one of the main outcomes
of the work that we intend to explain
in this paper.

To complete this overview of the applications of little $n$-discs operads to the deformation-quantization problem,
let us mention that the formality of the operad of little $2$-discs
can be extended to the higher dimensional little discs operads $\DOp_n$, $n = 2,3,\dots$,
and that higher dimensional generalizations of the deformation-quantization problem,
which involve structures governed by any class
of $n$-Poisson algebras,
have been studied by Calaque--Pantev--To\"en--Vaqui\'e--Vezzosi in the realm of derived algebraic geometry (see~\cite{CalaqueAl}).
Regarding the formality, we explicitly have a zigzag of quasi-isomorphisms
in the category of operads in differential graded modules
\begin{equation*}
\DGC_*(\DOp_n)\xleftarrow{\sim}\cdot\xrightarrow{\sim}\DGH_*(\DOp_n),
\end{equation*}
for all $n = 2,3,\dots$,
where we consider the homology of the operad little $n$-discs $\DGH_*(\DOp_n)$ as an operad in differential graded modules
equipped with a trivial differential. Recall also that we have an identity $\PoisOp_n = \DGH_*(\DOp_n)$
between this homology operad $\DGH_*(\DOp_n)$ and the operad of $n$-Poisson algebras $\PoisOp_n$.
We go back to this subject in the next part of this introduction.
In short, we are going to see that we have a refinement of the above formality result in the setting
of the Sullivan rational homotopy theory and we will explain that this enhanced formality result enables us to use
the $n$-Poisson operad $\PoisOp_n$
in order to construct a model for a rationalization of the little $n$-discs operad
in the category of spaces $\DOp_n^{\QQ}$.
We actually use this rational model to compute the rational homotopy type of embedding spaces.

The connection of the operads of little discs with the embedding calculus comes from certain descriptions
of the Goodwillie--Weiss towers, which are towers of ``polynomial'' approximations
of the embedding spaces $\Emb(M,N)$
associated to manifolds (see~\cite{GoodwillieWeissEmbeddings,WeissSurvey}). We just provide an overview of this connection
and we refer to Arone--Ching's paper, in this handbook volume, for a more comprehensive introduction
to the embedding calculus.

To make our account more precise, we focus on the particular case of Euclidean spaces $M = \RR^m$ and $N = \RR^n$,
and we consider the spaces of embeddings with compact support $\Emb_c(\RR^m,\RR^n)$.
The elements of this space are embeddings $f: \RR^m\hookrightarrow\RR^n$
which agree with a fixed affine map outside a compact domain.
We have an obvious map $\Emb_c(\RR^m,\RR^n)\rightarrow\Imm_c(\RR^m,\RR^n)$ where $\Imm_c(\RR^m,\RR^n)$
is a space of immersions with compact support (the obvious immersion counterpart
of the space of embeddings with compact support).
We use the notation $\overline{\Emb}_c(\RR^m,\RR^n)$
for the homotopy fiber of this map:
\begin{equation*}
\overline{\Emb}_c(\RR^m,\RR^n) := \hofib(\Emb_c(\RR^m,\RR^n)\rightarrow\Imm_c(\RR^m,\RR^n)).
\end{equation*}
We then have a weak-equivalence
$T_k\overline{\Emb}_c(\RR^m,\RR^n)\sim\DGOmega^{m+1}\Map_{\Op^{\leq k}}^h(\DOp_m,\DOp_n)$,
where $T_k\overline{\Emb}_c(\RR^m,\RR^n)$ is the $k$th level of the Goodwillie--Weiss tower for the embedding space $\overline{\Emb}_c(\RR^m,\RR^n)$,
whereas $\Map_{\Op^{\leq k}}^h(-,-)$ is a derived mapping space functor on a category of truncated operads $\Op^{\leq k}$
and $\DGOmega^{m+1}(-)$ denotes the $m+1$-fold iterated loop space functor. When $n-m\geq 3$, we can use convergence statements
on the Goodwillie--Weiss tower (see~\cite{GoodwillieKlein,GoodwillieWeissEmbeddings})
to deduce from this result
that we have a weak-equivalence
\begin{equation*}
\overline{\Emb}_c(\RR^m,\RR^n)\sim\DGOmega^{m+1}\Map_{\Op}^h(\DOp_m,\DOp_n),
\end{equation*}
where we consider a derived mapping space on the full category of operads
in topological spaces $\Map_{\Op}^h(-,-)$.

These results represent the abutment of a series of research initiated by Sinha~\cite{SinhaKnot} in the case $m=1$,
and pursued by Arone--Turchin~\cite{AroneTurchin}
for general $m\geq 1$.
To be more precise, the latter article gives a description of the spaces $T_k\overline{\Emb}_c(\RR^m,\RR^n)$
in terms of mapping spaces of (truncated) operadic bimodules over the operad $\DOp_m$
instead of mapping spaces of operads.
The model of the spaces $T_k\overline{\Emb}_c(\RR^1,\RR^n)$ given in Sinha's work
is also equivalent to a description in terms of mapping spaces
of operadic bimodules over the operad $\DOp_1$.
But subsequent works on the subject have given the expression of the spaces $T_k\overline{\Emb}_c(\RR^m,\RR^n)$
in terms of $m+1$-fold iterated loop spaces of operadic mapping spaces
stated in the above formula.
These finer results have been obtained by Boavida--Weiss~\cite{BoavidaWeiss}, for all $m\geq 1$, by an improvement of the methods
used in the study of the Goodwillie--Weiss calculus of embedding spaces,
while other authors have obtained general results on mapping spaces of (truncated) operadic bimodules
which permit to recover this homotopy identity between the spaces $T_k\overline{\Emb}_c(\RR^m,\RR^n)$
and iterated loop spaces of operadic mapping spaces
from the form of the results
obtained by Sinha and Arone--Turchin in their works (see the articles of Dwyer--Hess~\cite{DwyerHess} and Turchin \cite{TurchinChords}
for the case $m=1$, and the article of Ducoulombier--Turchin~\cite{DucoulombierTurchin}
for the case of general $m\geq 1$).

The formality of the little discs operads over the rationals can be used to determine the rational homotopy type
of the operadic derived mapping spaces $\Map_{\Op}^h(\DOp_m,\DOp_n)$
which occur in this description of the embedding spaces.
For this purpose, we use that we have a rational homotopy equivalence
\begin{equation*}
\Map_{\Op}^h(\DOp_m,\DOp_n)\sim_{\QQ}\Map_{\Op}^h(\DOp_m,\DOp_n^{\QQ})
\end{equation*}
as soon as $n-m\geq 3$, where $\DOp_n^{\QQ}$ denotes a rationalization of the topological operad of little $n$-discs $\DOp_n$
which is given by an operadic extension of the Sullivan rational homotopy theory of spaces.
In fact, we use an improved version of the formality which implies that this rational operad $\DOp_n^{\QQ}$
has a model $\langle\DGH^*(\DOp_n)\rangle$
which is determined by the rational cohomology of the operad of little $n$-discs $\DGH^*(\DOp_n) = \DGH^*(\DOp_n,\QQ)$.
Thus, we eventually have an effective approach to compute the rational homotopy of the operadic mapping spaces $\Map_{\Op}^h(\DOp_m,\DOp_n)$
and hence, the rational homotopy of the embedding spaces $\overline{Emb}{}^c(\RR^m,\RR^n)$
by the Goodwillie--Weiss theory of embedding calculus.

Besides the homotopy of the mapping spaces $\Map_{\Op}^h(\DOp_m,\DOp_n^{\QQ})$, we can compute the rational homotopy type of the spaces
of homotopy automorphisms $\Aut_{\Op}^h(\DOp_n^{\QQ})$ in the category of operads.
We will actually see that $\Aut_{\Op}^h(\DOp_2^{\QQ})$ is weakly-equivalent to a semi-direct product $\GT(\QQ)\ltimes\SO(2)^{\QQ}$,
where $\GT(\QQ)$ is the rational Grothendieck--Teichm\"uller group,
and this statement gives a theoretical explanation for the occurrence of the rational Grothendieck--Teichm\"uller group
in deformation-quantization. Let us mention that we have an analogue of this result
in the realm of profinite homotopy theory. We explain both statements
in this paper.

In fact, the main objective of this survey is to explain the result of the computations of mapping spaces
and of homotopy automorphism spaces of operads.
We organize our account as follows.

We recall the basic definition of the operads of little discs in the first section of the paper.
We also state the fundamental results on the connection between the operadic mapping spaces
and the Goodwillie--Weiss tower of embedding spaces
with more details in this first section.
We devote the second section of the paper to the particular case $n=2$
of the homotopy theory of $E_n$-operads.
We will explain that the little $2$-discs operad has a model given by an operad shaped on braid groupoids
and we use this model to get the connection between the Grothendieck--Teichm\"uller group
and our homotopy automorphism space $\Aut_{\Op}^h(\DOp_2^{\QQ})$.

We explain the general formality result on $E_n$-operads in the third section and we tackle the applications
to the computation of the rational homotopy of mapping spaces
of $E_n$-operads in the fourth section.
We use graph complexes in the proof of the formality of $E_n$-operads.
We therefore retrieve graph complexes, the graph complexes alluded to in the title of this paper, in our expression
of the rational homotopy of the mapping spaces of $E_n$-operads. The ultimate goal of this survey is precisely to explain
this graph complex description of the rational homotopy type of the mapping spaces of $E_n$-operads.

\medskip
In general, in this paper, we use the term `differential graded module' and the language of differential graded algebra,
rather than the language of chain complexes.
In fact, we only use the expression `(co)chain complex' for specific constructions of differential graded modules,
like the singular complex of a topological space, the Hochschild cochain complex, {\dots}
For short, we also use the prefix `dg' for any category of structured objects (like dg-algebras, dg-operads, {\dots})
that we form within a base category of differential graded modules.

In what follows, we generally define a differential graded module (thus, a dg-module for short) as the structure,
equivalent to a (possibly unbounded) chain complex, which consists of a module $M$
equipped with a lower $\ZZ$-graded decomposition $M = \bigoplus_{n\in\ZZ} M_n$
and with a differential $\delta: M\rightarrow M$
such that $\delta(M_*)\subset M_{*-1}$. In some cases, we may deal with upper graded dg-modules, which are equivalent to cochain complexes.
In this case, we can use the standard correspondence $M_* = M^{-*}$ to convert the upper grading into a lower grading.
We equip the category of dg-modules with its standard tensor product so that this category inherits a symmetric monoidal structure,
with a symmetry operator defined by using the usual sign rule of homological algebra,

In our study, we freely use the language and the results of the theory of model categories.
In particular, we often use the generic term `weak-equivalence' for a quasi-isomorphism,
since the quasi-isomorphisms represent the class of weak-equivalences of model structures
in the usual categories of dg-objects (dg-modules, dg-algebras, dg-operads, {\dots}).

\section{The operads of little discs and the embedding calculus}\label{sec:background}
We review the basic definition of the operad of little $n$-discs in this section, as we explained in the account of the plan of this paper.
We also make explicit the isomorphism between the homology of the operads of little discs and the operads of graded Poisson algebras
which we mention in the introduction of this paper.
We review the general definition of the notion of an $E_n$-operad in a second step, and we eventually give some explanations
on the correspondence between the (derived) mapping spaces of the operads
of little discs (equivalently, of $E_n$-operads)
and the Goodwillie--Weiss tower of embedding spaces.

\subsubsection{The operad of little $n$-discs}\label{background:little-discs-operads}
Let $D(q,r)$ denote the disc of center $q\in\RR^n$ and radius $r>0$ in the Euclidean space $\RR^n$.
Let us also set $\DD^n = D(0,1)$ for the unit $n$-disc.
The little $n$-discs of the eponym operad are $n$-discs $D = D(q,r)$ such that $D\subset\DD^n$.
In the definition of the operad of little $n$-discs, we use that giving such an $n$-disc $D\subset\DD^n$
amounts to giving an embedding $c: \DD^n\hookrightarrow\DD^n$
such that $c(v) = r v + q$. We then have $D = c(\DD^n)$.
In fact, in what follows, we rather use this definition of a little $n$-disc in terms of an embedding $c: \DD^n\hookrightarrow\DD^n$,
and we just set $c = c(\DD^n)$ by an abuse of notation when we consider the image
of this map. In a sense, we only use that $c: \DD^n\hookrightarrow\DD^n$ is determined by giving this image $c = c(\DD^n)$
when we provide a graphical illustration of our constructions.

The operad of little $n$-discs is a structure defined by the collection of spaces $\DOp_n = \{\DOp_n(r),r\in\NN\}$
where $\DOp_n(r)$ consists of $r$-tuples of little $n$-discs $\underline{c} = \{c_1,\dots,c_r\}$
such that $\mathring{c}_i\cap\mathring{c}_j = \varnothing$
for all pairs $i\not=j$.
In other words, we can represent an element of the space $\DOp_n(r)$ as a configuration of little $n$-discs, numbered from $1$ to $r$,
and with non overlapping interiors
as in the following example:
\begin{equation}\label{eqn:background:little-discs-operads:element-picture}
\underline{c} = \vcenter{\hbox{\includegraphics{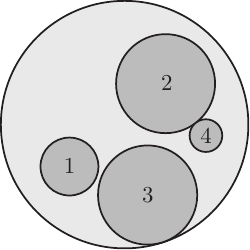}}}\in\DOp_2(4).
\end{equation}

The space $\DOp_n(r)$ inherits an action of the symmetric group on $r$-letters $\Sigma_n$.
For $\sigma\in\Sigma_n$ and $\underline{c} = \{c_1,\dots,c_r\}$,
we explicitly set $\sigma\cdot\underline{c} = (c_{\sigma^{-1}(1)},\dots,c_{\sigma^{-1}(r)})$.
In our graphical representation, this operation is given by applying the permutation $\sigma\in\Sigma_n$
to the index labelling of the little discs.

In addition, we have composition operations
\begin{equation}\label{eqn:background:little-discs-operads:composition-operations}
\circ_i: \DOp_n(k)\times\DOp_n(l)\rightarrow\DOp_n(k+l-1),
\end{equation}
defined for all $k,l\geq 0$, $i\in\{1,\dots,k\}$,
and which are given by the formula $\underline{a}\circ_i\underline{b} = (a_1,\dots,a_{i-1},a_i\circ b_1,\dots,a_i\circ b_l,a_{i+1},\dots,a_l)$,
for all $\underline{a} = (a_1,\dots,a_k)\in\DOp_n(k)$, $\underline{b} = (b_1,\dots,b_k)\in\DOp_n(l)$,
where we consider the composite of the map $a_i: \DD^n\rightarrow\DOp^n$ with the little disc embeddings $b_1,\dots,b_l: \DD^n\rightarrow\DOp^n$.
In our graphical representation, this operation is obtained by putting the configuration of little $n$-discs $\underline{b} = (b_1,\dots,b_k)$
in the $i$th little disc $a_i$ of the configuration $\underline{a} = (a_1,\dots,a_k)$
as depicted in the following figure:
\begin{equation}\label{eqn:background:little-discs-operads:composition-picture}
\vcenter{\hbox{\includegraphics{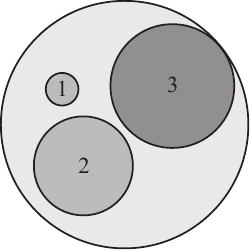}}}\;\circ_3\;\vcenter{\hbox{\includegraphics{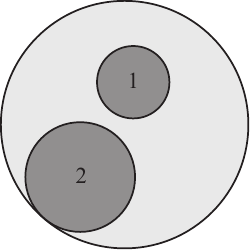}}}
\;=\;\vcenter{\hbox{\includegraphics{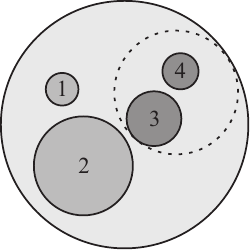}}}.
\end{equation}

These composition operations (\ref{eqn:background:little-discs-operads:composition-operations})
satisfy natural equivariance relations with respect to the action of the symmetric groups.
To be explicit, for $\sigma\in\Sigma_k$, $\tau\in\Sigma_l$, we have an identity of the form:
\begin{equation}\label{eqn:background:little-discs-operads:composition-equivariance}
(\sigma_*\cdot\underline{a})\circ_{\sigma(i)}(\tau\cdot\underline{b}) = (\sigma\circ_{\sigma(i)}\tau)\cdot(\underline{a}\circ_i\underline{b}),
\end{equation}
for a certain permutation $\sigma\circ_{\sigma(i)}\tau\in\Sigma_{k+l-1}$.
If we determine a permutation $\rho\in\Sigma_r$ by a sequence of values $\rho = (\rho(1),\dots,\rho(r))$,
then we precisely have $\sigma\circ_{\sigma(i)}\tau = (\sigma'(1),\dots,\sigma'(i-1),\tau'(1),\dots,\tau'(l),\sigma'(i+1),\dots,\sigma'(k))$,
where we set $\sigma'(x) = \sigma(x)$ if $\sigma(x)<\sigma(i)$ and $\sigma'(x) = \sigma(x)+l-1$ if $\sigma(x)>\sigma(i)$,
while we take $\tau'(y) = \tau(y) + i-1$ for all $y\in\{1,\dots,l\}$.

The composition operations also satisfy the following associativity relations:
\begin{equation}\label{eqn:background:little-discs-operads:composition-associativity}
(\underline{a}\circ_i\underline{b})\circ_{i+j-1}\underline{c} = \underline{a}\circ_i(\underline{b})\circ_j\underline{c})
\quad\text{and}
\quad(\underline{a}\circ_i\underline{b})\circ_{j+l-1}\underline{c} = (\underline{a}\circ_j\underline{c})\circ_i\underline{a}
\end{equation}
for all $\underline{a}\in\DOp_n(k)$, $\underline{b}\in\DOp_n(l)$, $\underline{c}\in\DOp_n(m)$,
where we assume $i\in\{1,\dots,k\}$ and $j\in\{1,\dots,l\}$ in the first case,
whereas we take $i,j\in\{1,\dots,k\}$ such that $i<j$
in the second case.
Besides, in $\DOp_n(1)$, we have a distinguished unit element $1\in\DOp_n(1)$,
the unit little $n$-disc, given by the identity embedding $\id: \DD^n\hookrightarrow\DD^n$,
such that the following unit relations hold in $\DOp_n(r)$:
\begin{equation}\label{eqn:background:little-discs-operads:composition-unit}
1\circ_1\underline{c} = \underline{c}\quad\text{and}\quad\underline{c}\circ_i 1 = \underline{c},
\end{equation}
for all $\underline{c}\in\DOp_n(r)$ and $i\in\{1,\dots,r\}$.

The operad of little $n$-discs is precisely the structure defined by the collection of little $n$-discs spaces $\DOp_n = \{\DOp_n(r),r\in\NN\}$
equipped with the unit element $1\in\DOp_n(1)$
and the composition products (\ref{eqn:background:little-discs-operads:composition-operations})
which satisfy the above relations (\ref{eqn:background:little-discs-operads:composition-equivariance}-\ref{eqn:background:little-discs-operads:composition-unit}).
(We go back to the general definition of an operad in the next paragraph.)
In the theory of iterated loop spaces, the elements $\underline{c}\in\DOp_n(r)$
are used to parameterize $r$-fold operations $\mu_{\underline{c}}: \DGOmega^n X^{\times r}\rightarrow\DGOmega^n X$,
where $\DGOmega^n X$ denotes the $n$-fold loop space
associated to a based space $X$.
For this purpose, we use that any element $\alpha\in\DGOmega^n X$ can be represented by a map $\alpha: \DD^n\rightarrow X$
such that $\alpha(\partial\DD^n) = *$,
and $\mu_{\underline{c}}: \DGOmega^n X^{\times r}\rightarrow\DGOmega^n X$
is defined by the obvious composition operation.

\subsubsection{The general notion of an operad}\label{background:operads}
The operad of little $n$-discs $\DOp_n$ is an instance of operad in the category of topological spaces,
but the notion of an operad makes sense in the general setting of a symmetric monoidal category $\CCat$,
with a tensor product $\otimes: \CCat\times\CCat\rightarrow\CCat$,
a unit object $\unit\in\CCat$
and symmetry isomorphisms $c_{X,Y}: X\otimes Y\xrightarrow{\simeq} Y\otimes X$
such that $c_{X,Y} c_{Y X} = \id$.

In brief, an operad is generally defined as a collection of objects
\begin{equation}\label{eqn:background:operads:collection}
\POp = \{\POp(r),r\in\NN\},
\end{equation}
such that the symmetric group $\Sigma_r$ acts on $\POp(r)\in\CCat$ for each $r\in\NN$,
together with composition products
\begin{equation}\label{eqn:background:operads:composition-operations}
\circ_i: \POp(k)\otimes\POp(l)\rightarrow\POp(k+l-1),
\end{equation}
and a unit morphism $\eta: \unit\rightarrow\POp(1)$,
which satisfy an obvious diagrammatic counterpart of the equivariance, unit and associativity relations
of the previous paragraph~\S\ref{background:little-discs-operads}(\ref{eqn:background:little-discs-operads:composition-equivariance}-\ref{eqn:background:little-discs-operads:composition-unit}).

If we can think in terms of elements, then we can regard the elements of an operad $p\in\POp(r)$
as abstract operations $p: A^{\otimes r}\rightarrow A$, as in the case of the operad of little $n$-discs.
Indeed, we can associate any operad with a category of algebras, which precisely consists
of the objects of the base category $A\in\CCat$
equipped with an action of such operations $p: A^{\otimes r}\rightarrow A$, for $p\in\POp(r)$.
In general, we call $\POp$-algebras this category of algebras
which we associate to an operad $\POp$.
In what follows, we use the term ``arity'' for the index $r\in\NN$
of the components of an operad (and for any kind of structure related to operads, like the cooperads, the symmetric collections, {\dots},
which we consider later on in this survey).

Throughout this article, we use the notation $\Op = \CCat\Op$ for the category of operads in a symmetric monoidal category $\CCat$,
where we obviously define an operad morphism $\phi: \POp\rightarrow\QOp$
as a collection of morphisms in the base category $\phi: \POp(r)\rightarrow\QOp(r)$
which preserve the structure operations associated to our objects.
For the operad of little $n$-discs, we take $\CCat = \Top$, the category of topological spaces,
together with $\otimes = \times$, the cartesian product,
and $\unit = *$, the one-point set.
In what follows, we also consider operads in the cartesian category of simplicial sets $(\sSet,\times,*)$,
in a category of modules $(\Mod,\otimes,\kk)$, where $\kk$ is any ground ring,
in the category of dg-modules $(\dg\Mod,\otimes,\kk)$,
and in the category of graded modules $(\gr\Mod,\otimes,\kk)$.

We use that a lax symmetric monoidal functor induces a functor between operad categories.
We first consider the case of the singular homology theory $\DGH_*(-) = \DGH_*(-,\kk)$ with coefficients in the ground ring $\kk$,
which we regard as a lax symmetric monoidal functor between the category of topological spaces $\Top$
and the category of graded modules $\gr\Mod$.
We get that this functor carries the little $n$-discs operad $\DOp_n$
to an operad in graded modules $\DGH_*(\DOp_n) = \{\DGH_*(\DOp_n(r)),r\in\NN\}$.
We then have the following statement:

\begin{thm}[{F. Cohen~\cite{Cohen}}]\label{thm:background:little-discs-homology}
For $n\geq 2$, we have an isomorphism of operads in graded modules $\DGH_*(\DOp_n)\simeq\PoisOp_n$,
where $\PoisOp_n$ is the operad governing $n$-Poisson algebras.
For $n=1$, we have $\DGH_*(\DOp_1)\simeq\AsOp$
where $\AsOp$ is the operad governing associative algebras.
\end{thm}

\begin{proof}[Explanations]
The operad of $n$-Poisson algebras $\PoisOp_n$ is defined by a presentation by generators and relations.
This presentation reflects the definition of an $n$-Poisson algebra as a graded vector space $A$
equipped with a commutative product $\mu: A\otimes A\rightarrow A$
together with an operation $\lambda: A\otimes A\rightarrow A$,
of degree $n-1$, which is odd when $n$ is odd, even when $n$ is even,
and which satisfies an obvious graded generalization of the classical relations of Poisson brackets.

To formally define this presentation of the $n$-Poisson operad, we consider abstract generating operations $\mu,\lambda\in\PoisOp_n(2)$
such that $\deg(\mu) = 0$, $\deg(\lambda) = n-1$,
and which we equip with the action of the symmetric group such that $(1\ 2)\cdot\mu = \mu$
and $(1\ 2)\cdot\lambda = (-1)^n\lambda$.
Then we take the modules spanned by all formal operadic composites of these operations moded out by the operadic ideal generated by the relations
such that $\mu\circ_1\mu\equiv\mu\circ_2\mu$, $\lambda\circ_1\lambda\equiv\lambda\circ_2\lambda + (2\ 3)\cdot\lambda\circ_1\lambda$,
and $\lambda\circ_1\mu\equiv\mu\circ_2\lambda + (2\ 3)\cdot\mu\circ_1\lambda$,
which are obvious operadic counterparts of the associativity relation, of the Jacobi relation,
and of the Poisson distribution formula.
In what follows, we also use variables in our notation of the elements an operad,
so that we may write $\mu = \mu(x_1,x_2)$ and $\lambda = \lambda(x_1,x_2)$
for the generating operations of the $n$-Poisson operad.
In addition, we may also use the standard algebraic notation $\mu(x_1,x_2) = x_1 x_2$
and $\lambda(x_1,x_2) = [x_1,x_2]$
for these operations. Then we get that elements of the graded module $\PoisOp_n(r)$ can be interpreted as Poisson polynomials
on $r$ variables $\pi = \pi(x_1,\dots,x_r)$ that are of degree one with respect to each variable $x_i$, $i = 1,\dots,r$
(like $\pi(x_1,\dots,x_6) = [[x_4,x_1],x_2]\cdot x_6\cdot  [x_3,x_5]$).

The above presentation by generators and relations actually returns a version of the $n$-Poisson operad with no term in arity zero,
but for our purpose, we consider a version of the $n$-Poisson operad such that $\PoisOp_n(0) = \kk e$,
with an extra element of arity zero $e\in\PoisOp_n(0)$ that models the unit element
of a Poisson algebra.
To determine the composition product of an operation with this extra element, we just use the formulas $\mu\circ_1 e = \mu\circ_2 e = 1$
and $\lambda\circ_1 e = \lambda\circ_2 e = 0$ for the generating operations
of our operad.

The operad of associative algebras $\AsOp$ is defined similarly, as the operad generated by an operation of degree zero $\mu\in\AsOp(2)$,
on which the symmetric group $\Sigma_2$ acts freely, moded out by the operadic ideal of relations
generated by the associativity identity $\mu\circ_1\mu\equiv\mu\circ_2\mu$. We still assume $\AsOp(0) = \kk e$
for an element of arity zero $e\in\AsOp(0)$ such that $\mu\circ_1 e = \mu\circ_2 e = 1$.

If we go back to the operad of little $n$-discs $\DOp_n$,
then we easily see that we have a homotopy equivalence $\DOp_n(2)\sim\Sphere^{n-1}$,
where $\Sphere^{n-1}$ denotes the $n-1$-sphere (we go back to this subject in~\S\ref{background:fulton-macpherson-operad}).
Thus, in the case $n\geq 2$, we have $\DGH_0(\DOp_n(2)) = \kk\mu_0$, $\DGH_{n-1}(\DOp_n(2)) = \kk\lambda_0$, and $\DGH_*(\DOp_n(2)) = 0$ for $*\not=n-1$,
for some generating homology classes $\mu_0,\lambda_0\in\DGH_*(\DOp_n(2))$.
The claim is that these generating elements $\mu_0,\lambda_0\in\DGH_*(\DOp_n(2))$ satisfy the structure relations
of the generating operations of the $n$-Poisson operad
in $\DGH_*(\DOp_n)$,
so that we have a well-defined operad morphism $\phi: \PoisOp_n\rightarrow\DGH_*(\DOp_n)$
which is determined by the assignment $\phi(\mu) := \mu_0$ and $\phi(\lambda) := \lambda_0$
on generators.
Then one proves that this operad morphism is actually an isomorphism.

The definition of the isomorphism $\phi: \AsOp\rightarrow\DGH_*(\DOp_1)$
in the case $n=1$ is similar. In this case, we just have $\DOp_1(2)\sim\Sphere^0\Rightarrow\DGH_0(\DOp_1(2)) = \kk\mu_0\oplus\kk(1\ 2)\mu_0$,
where $(1\ 2)\mu_0$ denotes the image of the homology class $\mu_0$
under the action of the transposition in $\DGH_0(\DOp_1(2))$. Then we determine our isomorphism by the assignment $\phi(\mu) := \mu_0$
for the generating operation of the associative operad $\mu\in\AsOp(2)$.

Note that the operad $\AsOp$ is formed within the base category of (ungraded) $\kk$-modules.
Hence, when we claim that we have an isomorphism $\AsOp\simeq\DGH_*(\DOp_1)$ we implicitly assert that the homology of the operad $\DOp_1$
is concentrated in degree zero. In fact, we can easily see that we have a weak-equivalence relation in the category of operads $\DOp_1\sim\PiOp$,
where $\PiOp$ a set-theoretic counterpart of the associative operad
defined by $\PiOp(r) = \Sigma_r$ for each $r\in\NN$.
In what follows, we also use the name ``permutation operad'' to refer to this operad $\PiOp$.

We go back to the notion of a weak-equivalence in the category of topological operads in~\S\ref{background:en-operads},
where we explain the definition of an $E_n$-operad.
But, in fact, we just use the obvious definition, where a weak-equivalence of topological operads is a morphism $\phi: \POp\xrightarrow{\sim}\QOp$
which forms a weak-equivalence of topological spaces arity-wise $\phi: \POp(r)\xrightarrow{\sim}\QOp(r)$, for $r\in\NN$.
Then we say that a pair of topological operads $\POp,\QOp\in\Top\Op$ are weakly-equivalent (as operads) and we write $\POp\sim\QOp$
when we have a zigzag of such weak-equivalences that link $\POp$ and $\QOp$
in the category of operads.
In the case $\POp = \DOp_1$ and $\QOp = \PiOp$, we actually have a direct morphism $\phi: \DOp_1\xrightarrow{\sim}\PiOp$ (easy exercise: define this morphism).
In this relation $\DOp_1\sim\PiOp$, we just regard the operad in sets $\PiOp$ as a discrete operad in topological spaces.
\end{proof}

\subsubsection{The notion of a bimodule over an operad}\label{background:operad-bimodules}
Let $\POp$ be any operad in a symmetric monoidal category $\CCat$. We have a natural notion of bimodule associated to $\POp$
which consists of collections
\begin{equation}\label{eqn:background:operad-bimodules:collection}
\MOp = \{\MOp(r),r\in\NN\},
\end{equation}
such that the object $\MOp(r)$ is equipped with an action of the symmetric group on $r$-letters $\Sigma_r$, for each $r\in\NN$,
as in the case of operads,
but where we now have composition products of the form
\begin{equation}\label{eqn:background:operad-bimodules:composition-operations}
\circ_i: \POp(k)\otimes\MOp(l)\rightarrow\MOp(k+l-1)\quad\text{and}\quad\circ_i: \MOp(k)\otimes\POp(l)\rightarrow\MOp(k+l-1),
\end{equation}
for $k,l\in\NN$ and $i\in\{1,\dots,k\}$. We assume that these composition operations satisfy an obvious generalization
of the equivariance, unit and associativity relations of operads.
We use the notation $\POp\BiMod$ for this category of $\POp$-bimodules, where a morphism of $\POp$-bimodules $\phi: \MOp\rightarrow\NOp$
obviously consists of a collection of morphisms in the base category $\phi: \MOp(r)\rightarrow\NOp(r)$
which preserve the structure operations associated to our objects
again.

In the literature, the names ``infinitesimal bimodule'' (see~\cite{AroneTurchin}), ``linear bimodule'' (see~\cite{DwyerHess})
and ``abelian bimodule'' (see~\cite[\S III.2.1]{FresseBook})
are also used for this category of bimodules,
because we have another natural notion of bimodule associated to an operad, which is defined by using that the category of operads
is identified with a category of monoids in a certain monoidal category. But we do not use the latter category of operadic bimodules.
Therefore, we prefer to simplify the terminology for our category of bimodules.

In what follows, we mainly use that any operad $\POp$ forms a bimodule over itself and if we have an operad morphism $\phi: \POp\rightarrow\QOp$,
then the operad $\QOp$ naturally inherits a $\POp$-bimodule structure by restriction
of its natural $\QOp$-bimodule structure.

We consider the category of bimodules in topological spaces which is associated to a little discs operad $\DOp_m$, for some $m\geq 1$.
We have an operad morphism $\iota: \DOp_m\rightarrow\DOp_n$ which is defined by mapping any collection of little $m$-discs in $\DD^m$
to the collection of little $n$-discs with the same radius and centers in $\DD^n$,
as in the following example where we take $m=1$ and $n=2$:
\begin{equation}\label{eqn:background:operad-bimodules:little-discs-embeddings-picture}
\vcenter{\hbox{\includegraphics[scale=.66]{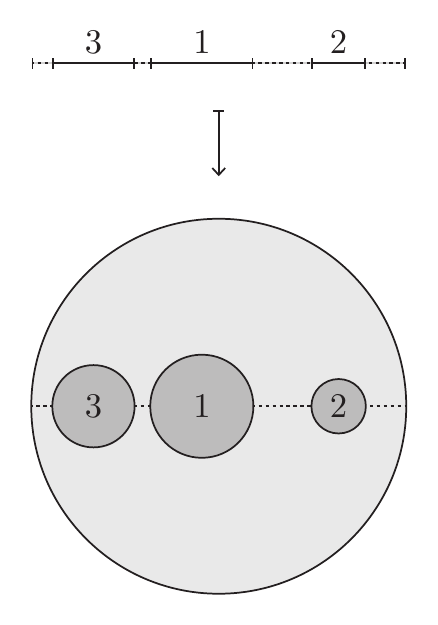}}}.
\end{equation}
We use this morphism to provide the operad $\DOp_n$ with the structure of a $\DOp_m$-bimodule.
We also consider the example of $\DOp_m$-bimodule formed by the operad $\DOp_m$ itself.

The category $\DOp_m\BiMod$ inherits a natural simplicial model structure with as class of weak-equivalences
the morphisms of $\DOp_m$-bimodules $\phi: \MOp\xrightarrow{\sim}\NOp$
which define a weak-equivalence of topological spaces $\phi: \MOp(r)\xrightarrow{\sim}\NOp(r)$
in each arity $r\in\NN$.
The mapping spaces $\Map_{\DOp_m\BiMod}(-,-)$ of this simplicial model structure are inherited from the category of topological spaces.
For our purpose, we consider derived mapping spaces $\Map_{\DOp_m\BiMod}^h(-,-)$
which are defined by $\Map_{\DOp_m\BiMod}^h(\MOp,\NOp) := \Map_{\DOp_m\BiMod}(\ROp,\SOp)$,
for any pair of objects $\MOp,\NOp\in\DOp_m\BiMod$,
where $\ROp$ is a cofibrant replacement of $\MOp$ in the category of $\DOp_m$-bimodules,
while $\SOp$ is a fibrant replacement of $\NOp$ (actually, we can take $\SOp = \NOp$ because every $\DOp_m$-bimodule
is fibrant in the setting of topological spaces).

We have the following result:

\begin{thm}[{D. Sinha~\cite{SinhaKnot}, G. Arone and V. Turchin~\cite{AroneTurchin}}]\label{thm:background:embedding-bimodule-model}
For $n-m\geq 3$, we have a weak-equivalence
\begin{equation*}
\overline{\Emb}_c(\RR^m,\RR^n)\sim\Map_{\DOp_m\BiMod}^h(\DOp_m,\DOp_n),
\end{equation*}
between the space of smooth embeddings with compact support modulo immersions $\overline{\Emb}_c(\RR^m,\RR^n)$
and the derived mapping space associated to the objects $\MOp = \DOp_m$ and $\NOp = \DOp_n$
in the category of $\DOp_m$-bimodules $\DOp_m\BiMod$.
\end{thm}

\begin{proof}[Explanations and references]
The paper~\cite{SinhaKnot} only deals with the case $m=1$ of this theorem, and proves a result of a different form,
which involves the totalization of a cosimplicial space rather than mapping spaces of operadic bimodules.
But this totalization can be interpreted as a cosimplicial model
of mapping spaces of operadic bimodules $\Map_{\PiOp\BiMod}^h(\PiOp,-)$,
where we consider the permutation operad $\PiOp$ (the set-theoretic version of the associative operad)
instead of the operad of little $1$-discs $\DOp_1$.
Hence, since we have $\DOp_1\sim\PiOp$, we get that the result obtained in~\cite{SinhaKnot} is equivalent to the statement of the theorem
in the case $m=1$.
The general case $m\geq 1$ of the theorem is addressed in the paper~\cite{AroneTurchin},
from which we also borrow this operadic bimodule formulation
of the result.

Both papers~\cite{AroneTurchin,SinhaKnot} rely on models of the Goodwillie--Weiss approximations of the space $\overline{\Emb}_c(\RR^m,\RR^n)$
and actually give a proof of the existence of a level-wise weak-equivalence:
\begin{equation}
T_k\overline{\Emb}_c(\RR^m,\RR^n)\sim\Map_{\DOp_m\BiMod^{\leq k}}^h(\DOp_m,\DOp_n),
\end{equation}
where $T_k\overline{\Emb}_c(\RR^m,\RR^n)$ denotes the $k$th level
of the Goodwillie--Weiss tower, whereas $\Map_{\DOp_m\BiMod^{\leq k}}^h(-,-)$
denotes a mapping space associated to a category of truncated
bimodules $\DOp_m\BiMod^{\leq k}$. (This category just consists of bimodules that are defined up to arity $k$.)
The statement cited in the theorem is obtained by taking a limit $k\rightarrow\infty$
and by using that the Goodwillie--Weiss tower $T_k\overline{\Emb}_c(\RR^m,\RR^n)$
converges to $\overline{\Emb}_c(\RR^m,\RR^n)$
when $n-m\geq 3$.
The Goodwillie--Weiss tower is defined in~\cite{GoodwillieWeissEmbeddings,WeissEmbeddings}.
The convergence statement which we cite in this account is stated in these papers and rely on results
of Goodwillie and Goodwillie--Klein~\cite{GoodwillieKlein}.
We refer to~\cite{WeissSurvey} for a survey on the embedding calculus and for further references
on this subject. We just give a few additional ideas to explain the difference of the approach between the papers~\cite{AroneTurchin,SinhaKnot}.

In short, the space $T_k\overline{\Emb}_c(\RR^m,\RR^n)$ represents the value at $U = \RR^m$
of a polynomial approximation of degree $\leq k$ (in the homotopy sense)
of the functor $U\mapsto\overline{\Emb}_c(U,\RR^n)$
on the category of open subsets of the Euclidean space $\RR^m$
whose complement is bounded.
This polynomial approximation property asserts that we have a weak-equivalence $\overline{\Emb}_c(U,\RR^n)\rightarrow T_k\overline{\Emb}_c(U,\RR^n)$
when $U$ has the form $U = \coprod_{i=1}^l\mathring{\DD}{}^m$, with $l\leq k$,
and that the functor $U\mapsto T_k\overline{\Emb}_c(U,\RR^n)$
is universal with this property. In~\cite{AroneTurchin}, the equivalence between this space and the mapping space of truncated operadic bimodules
is obtained by observing that this category of open sets of the form $U = \coprod_{i=1}^l\mathring{\DD}{}^m$, $l\leq k$,
is homotopy equivalent to a category of operators
associated to the operad $\DOp_m$.

In~\cite{SinhaKnot}, the starting point is a model of the space $T_k\overline{\Emb}_c(\RR^1,\RR^n)$
which involves the totalization of a cosimplicial space,
as we explain earlier in our explanations of this theorem.
This simplicial set is defined by using a ``compactification'' of the space of configurations of points in $\RR^n$.
This compactification is identified with an operad $\KOp_n$, the Kontsevich operad,
such that we have a weak-equivalence relation $\KOp_n\sim\DOp_n$
in the category of operads in topological spaces $\Top\Op$. (We go back to the definition of this operad
in~\S\ref{background:fulton-macpherson-operad}.)
\end{proof}

We mentioned in the introduction that the mapping spaces of operadic bimodules considered in the previous theorem are weakly-equivalent
to iterated loop spaces on mapping spaces of operads. To be precise, we have the following general statement:

\begin{thm}[{W. Dwyer and K. Hess~\cite{DwyerHess}, V. Turchin~\cite{TurchinDelooping}, J. Ducoulombier and V. Turchin~\cite{DucoulombierTurchin}}]\label{thm:background:bimodule-delooping}
For any $m\geq 1$, and for any operad in topological spaces $\QOp$ such that $\QOp(0) = \QOp(1) = *$, we have a weak-equivalence:
\begin{equation*}
\Map_{\DOp_m\BiMod}^h(\DOp_m,\QOp)\sim\DGOmega^{m+1}\Map_{\Op}^h(\DOp_m,\QOp),
\end{equation*}
where we consider a derived mapping space in the category of $\DOp_m$-bimodules on the left hand side and the $m+1$-fold loop space
of the derived mapping space associated the objects $\DOp_m,\QOp\in\Op$
in the category of operads $\Op$ on the right hand side.
\end{thm}

\begin{proof}[Explanations and references]
The papers~\cite{DwyerHess,TurchinDelooping} actually prove that the totalization of Sinha's cosimplicial space
is weakly-equivalent to a double delooping
of the operadic mapping space~$\Map_{\Op}^h(\DOp_1,\QOp)$. Thus, these papers establish
a statement which is equivalent to the claim of this theorem in the case $m=1$
since we observed that Sinha's cosimplicial space has an interpretation as a cosimplicial model
of a mapping space of operadic bimodules (see our explanations on the result
of the previous theorem).

The paper~\cite{DucoulombierTurchin} establishes the general case of the present theorem.
This paper also establishes a delooping result $\Map_{\DOp_m\BiMod^{\leq k}}^h(\DOp_m,\QOp)\sim\DGOmega^{m+1}\Map_{\Op^{\leq k}}^h(\DOp_m,\QOp)$
for the tower of mapping spaces of truncated bimodules $\Map_{\DOp_m\BiMod^{\leq k}}^h(\DOp_m,\QOp)$
which correspond to the levels of the Goodwillie--Weiss tower $T_k\overline{\Emb}_c(\RR^m,\RR^n)$
in the result of Theorem~\ref{thm:background:embedding-bimodule-model}
(an analogous result is also established in~\cite{TurchinDelooping} in the case $m=1$).
The category of operads $\Op^{\leq k}$ which we consider in this context
consists of the operads that are defined up to arity $k$.
\end{proof}

The results of Theorem~\ref{thm:background:embedding-bimodule-model}
and Theorem~\ref{thm:background:bimodule-delooping}
imply the following statement,
which was also established by Boavida--Weiss in~\cite{BoavidaWeiss}
without going through these intermediate results:

\begin{thm}\label{thm:background:embedding-operad-model}
For $n-m\geq 3$, we have a weak-equivalence
\begin{equation*}
\overline{\Emb}_c(\RR^m,\RR^n)\sim\DGOmega^{m+1}\Map_{\Op}^h(\DOp_m,\DOp_n)
\end{equation*}
between the space of smooth embeddings with compact support modulo immersions
and the $m+1$st fold loop space of the derived mapping space
of the operads $\DOp_m,\DOp_n\in\Top\Op$.
\end{thm}

\begin{proof}[References]
In the case $m=1$, Dwyer--Hess explicitly state this result in their paper~\cite{DwyerHess},
as a consequence of their work on the delooping of mapping spaces
of operadic bimodules
and of the result of Sinha~\cite{SinhaKnot} (see also Turchin's paper~\cite{TurchinDelooping}
for an analogous observation).
Ducoulombier--Turchin also state the general case of this result as one of the main motivations for their work in~\cite{DucoulombierTurchin}.
In the paper~\cite{BoavidaWeiss}, Boavida--Weiss give a direct proof of the result of this theorem, in the general case $m\geq 1$,
by using ideas of higher category theory and by elaborating on the methods of the embedding calculus.
\end{proof}

\subsubsection{The notion of an $E_n$-operad}\label{background:en-operads}
In our account of the proofs of Theorem~\ref{thm:background:embedding-bimodule-model},
we mention that D. Sinha uses a compactification
of the configuration spaces of points in $\RR^n$ to produce an operad $\KOp_n$
such that we have the weak-equivalence relation $\KOp_n\sim\DOp_n$
in the category of operads.

In general, we use this standard notation $X\sim Y$ to assert that a pair of objects in a model category $X,Y\in\CCat$
(or more generally in a category equipped with weak-equivalences)
is linked by a zigzag of weak-equivalences $X\xleftarrow{\sim}\cdot\xrightarrow{\sim}\cdot\ldots\cdot\xrightarrow{\sim} Y$.
In the context of operads in topological spaces, we consider as a class of weak-equivalences the morphisms $\phi: \POp\xrightarrow{\sim}\QOp$
which form a weak-equivalence in the base category arity-wise $\phi: \POp(r)\xrightarrow{\sim}\QOp(r)$
as we explain in our account of Cohen's determination of the homology of the operad
of little $n$-discs (Theorem~\ref{thm:background:little-discs-homology}).
In fact, this class or morphisms represents the class of weak-equivalences
of a model structure on the category of topological operads.
Later on, we will see that we have an analogous result for the category of operads in simplicial sets, in dg-modules, {\dots} (with some restrictions
on the component of arity zero in the dg-module contexts).

Now, we just define the class of $E_n$-operads in topological spaces as the class of operads $\EOp_n\in\Top\Op$ such that $\EOp_n\sim\DOp_n$.
In subsequent applications, we consider analogous notions of $E_n$-operad in the category of simplicial sets,
in the category of dg-modules, {\dots}
Then we consider the image of the operad of little $n$-discs under an appropriate functor to define our reference model
of the class of $E_n$-operads in our base category. For instance, we define an $E_n$-operad in simplicial sets
as an operad $\EOp_n\in\sSet\Op$ such that $\POp\sim\Sing_{\bullet}(\DOp_n)$,
where $\Sing_{\bullet}: X\mapsto\Sing_{\bullet}(X)$
denotes the singular complex functor on the category of topological spaces. (We just use that this functor is strongly symmetric monoidal
to obtain that it preserves operads.)
In fact, we have $\EOp_n\sim\Sing_{\bullet}(\DOp_n)\Leftrightarrow|\EOp_n|\sim\DOp_n$,
where $|-|$ denotes the classical geometric realization functor on the category of simplicial sets.
Thus, an $E_n$-operad in simplicial sets is exactly an operad $\EOp_n\in\sSet\Op$ whose geometric realization $|\EOp_n|$
defines an $E_n$-operad in topological spaces.

In the literature, some authors require that the components of an $E_n$-operad $\EOp_n(r)$
are equipped with a free action of the symmetric groups $\Sigma_r$.
This assumption ensures that the categories of algebras associated to our $E_n$-operads are homotopy equivalent.
but we do need this result in what follows, and therefore we do not make any requirement
other than the existence of a weak-equivalence relation $\EOp_n\sim\DOp_n$
in the definition of an $E_n$-operad.

In the case $n=1$, we obtain that the permutation operad $\PiOp$ (the set-theoretic counterpart of the associative operad)
forms an $E_1$-operad since we have $\PiOp\sim\DOp_1$ (see our comments
on Cohen's determination of the homology of the operads
of little discs in Theorem~\ref{thm:background:little-discs-homology}).
Thus, we can also determine the class of $E_1$-operads as the class of operads $\EOp_1$
such that $\EOp_1\sim\PiOp$.
In a sense, the goal of sections~\S\S\ref{sec:e2-operads}-\ref{sec:rational-homotopy}
is to explain the definition of simple models of $E_n$-operads for $n\geq 2$,
and to give another characterization of the class of $E_n$-operads.

\subsubsection{The Kontsevich and the Fulton--MacPherson operads}\label{background:fulton-macpherson-operad}
The Kontsevich operad $\KOp_n$, which is used in Sinha's study of the embedding spaces $\overline{\Emb}_c(\RR,\RR^n)$, is an instance of an $E_n$-operad.
This operad is a variant of another $E_n$-operad, the Fulton--MacPherson operad $\FMOp_n$,
which is also defined by a compactification of the configuration spaces
of points in $\RR^n$.
The Fulton--MacPherson operad is weakly-equivalent to the operad of little $n$-discs,
and hence defines a model of an $E_n$-operad too.

In short, recall that the configuration spaces of points in $\RR^n$
are the spaces $\FOp(\RR^n,r)$, $r\in\NN$,
such that $\FOp(\RR^n,r) = \{(a_1,\dots,a_r)\in(\RR^n)^{\times r}|\text{$a_i\not=a_j$ for all $i\not=j$}\}$.
The spaces underlying the Fulton--MacPherson operad $\FMOp_n$ consist of an extension of these configuration spaces,
where points can collide $q_i\rightarrow q_j$ to produce infinitesimal configurations
with different scale levels, as represented in the following schematic picture:
\begin{equation}\label{eqn:background:fulton-macpherson-operad:picture}
\vcenter{\hbox{\includegraphics[scale=.8]{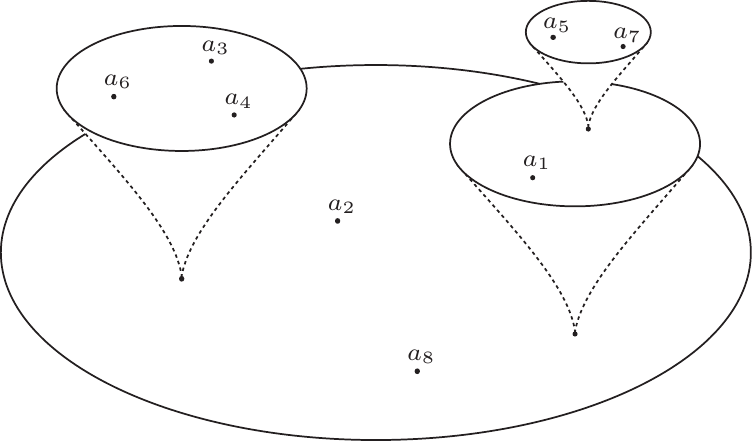}}}.
\end{equation}

In an infinitesimal configuration, we retain the collision direction between two colliding points $u_{ij} = (a_j-a_i)/||a_j-a_i||$,
as well as the speed ratio of the collision $r_{ijk} = ||a_j-a_i||/||a_k-a_i||$
when three points collide together. The latter information is just forgotten
in the definition of the Kontsevich operad $\KOp_n$.
Because of this modification, the configuration space $\FOp(\RR^n,r)$ embeds into $\FMOp_n(r)$
but not into $\KOp_n(r)$.
Thus, the space $\KOp_n(r)$ is not a compactification of the configuration space $\FOp(\RR^n,r)$
in the proper sense.
Nevertheless, we have a weak-equivalence of spaces $\KOp_n(r)\sim\FOp(\RR^n,r)$ in each arity $r\in\NN$,
and a weak-equivalence of operads $\pi: \FMOp_n\xrightarrow{\sim}\KOp_n$
when we consider the whole operad structure
attached to our objects.

Let us observe that we have an obvious weak-equivalence $\omega: \DOp_n(r)\xrightarrow{\sim}\FOp(\mathring{\DD}{}^n,r)$
from the underlying spaces of the operad of little $n$-discs $\DOp_n(r)$
to the configuration spaces associated to the open discs $\mathring{\DD}{}^n$
and hence, to the configuration spaces associated to the Euclidean space $\RR^n$,
since $\mathring{\DD}{}^n\simeq\RR^n\Rightarrow\FOp(\mathring{\DD}{}^n,r)\simeq\FOp(\RR^n,r)$. (This homotopy equivalence
carries a configuration of little $n$-discs to the configuration of the disc centers.)
In~\cite{Salvatore}, a zigzag of weak-equivalences of operads $\DOp_n\xleftarrow{\sim}\FMOp_n'\xrightarrow{\sim}\FMOp$
is defined by considering an analogue of the Fulton--MacPherson operad $\FMOp_n'$
where we take a compactification of the configuration spaces of little discs
instead of points. This argument gives a proof that the Fulton--MacPherson operad $\FMOp_n$
is an instance of an $E_n$-operad.

The Fulton--MacPherson operad $\FMOp_n$ is cofibrant with respect to a certain model structure (the Reedy model structure of~\cite[\S II.8.4]{FresseBook})
on the subcategory of operads $\Op_*$ such that $\POp(0) = *$ (see~\cite{Salvatore} and use the result of~\cite[Theorem II.8.4.11]{FresseBook}).
This is not the case of the Kontsevich operad $\KOp_n$ but this operad contains the permutation operad,
unlike the Fulton--MacPherson operad, and this is exactly the observation
used by Sinha in the definition of his cosimplicial space model of the embedding spaces $\overline{\Emb}_c(\RR,\RR^n)$.
The Fulton--MacPherson operad is used by Kontsevich in his proof of the formality of $E_n$-operads. (We go back to this subject in~\S\ref{sec:rational-homotopy}.)

\section{Braids and the homotopy theory of $E_2$-operads}\label{sec:e2-operads}
The main goal of this paper is to explain a combinatorial description of the rational homotopy of the derived mapping spaces $\Map_{\Op}^h(\DOp_m,\DOp_n)$
and of the homotopy automorphism spaces $\Aut_{\Op}^h(\DOp_n)$
associated to the operads of little discs $\DOp_n$, $n = 2,3,\dots$.
In a first step, we examine the particular case $n=2$ of the homotopy theory of $E_n$-operads.

In the previous paragraph, we observed that we have a homotopy equivalence $\DOp_n(r)\xrightarrow{\sim}\FOp_n(\mathring{\DD}{}^n,r)$
for each $r\in\NN$, where $\FOp_n(\mathring{\DD}{}^n,r)$
denotes the configuration space of $r$ points in $\mathring{\DD}{}^n$.
In the case $n=2$, this result implies that $\DOp_2(r)$ forms an Eilenberg--MacLane space $K(P_r,1)$,
where $P_r$ is the pure braid group on $r$ strands in $\mathring{\DD}{}^2$.

The first purpose of this section is to explain that we can elaborate on this result in order to get a model of the class of $\EOp_2$-operads
in the category of operad in groupoids.
In short, we check that we have a relation $\DOp_2\sim\DGB(\CoB)$, where we consider the classifying spaces of a certain operad in groupoids,
the operad of colored braids $\CoB$.
Besides, we will see that this operad $\CoB$ governs the category of strict braided monoidal categories as category of algebras.
Then we use a variant of this operad, the operad of parenthesized braids, which we associate to the category of general braided monoidal categories,
in order to define the Grothendieck--Teichm\"uller group as a group of automorphisms of an operad in the category of groupoids,
and we check that the homotopy of the space of homotopy automorphisms of the little $2$-discs operad
is identified with this group $\GT$.

To be more precise, we deal with versions of the Grothendieck--Teichm\"uller which are associated to various completions of operads in groupoids,
and we accordingly consider various completions of the little $2$-discs operad (namely, the profinite completion and the rationalization)
when we examine the relationship between the Grothendieck--Teichm\"uller group and the homotopy of $E_2$-operads.
In this section, we use a simple definition of these completion operations which we form at the level of the groupoid models of our operads.
In the next section, we revisit the definition of the particular case of the rationalization of operads
by using the Sullivan rational homotopy theory of spaces.

\subsubsection{The operad of colored braids}\label{e2-operads:colored-braids}
Briefly recall that a braid on $r$-strands is an isotopy class of paths $\alpha: [0,1]\rightarrow\FOp(\mathring{\DD}{}^2\times[0,1],r)$
with $\alpha(t) = (\alpha_1(t),\dots,\alpha_r(t))\in\FOp(\mathring{\DD}{}^2\times[0,1],r)$
such that $\alpha_i(t) = (z_i(t),t)$ for each $t\in[0,1]$,
and where we assume that $\alpha(0) = (\alpha_1(0),\dots,\alpha_r(0))$ (respectively, $\alpha(1) = (\alpha_1(1),\dots,\alpha_r(1))$)
is a permutation of fixed contact points $((z_1^0,0),\dots,(z_r^0,0))$ (respectively, $((z_1^0,1),\dots,(z_r^0,1))$)
on the equatorial line $y=0$ of the disc $\DD^2\times\{0\}$ (respectively, $\DD^2\times\{1\}$).
Thus, we have $z_i^0 = (x_i^0,0)$ for $i = 1,\dots,r$, and by convention we can also assume that the contact points
are ordered so that $x_1^0<\dots<x_r^0$.
In what follows, we use the usual representation of an isotopy class of braid in terms of a diagram, which is given by a projection
onto the plan $y=0$ in the space $\mathring{\DD}{}^2\times[0,1]$.
The assumption $\alpha(t)\in\FOp(\mathring{\DD}{}^2\times[0,1],r)$ is equivalent to the requirement that we have $z_i(t)\not=z_j(t)$
for all pairs $i\not=j$.
In this definition, we assume that the strands of a braid are indexed by the set $\{1,\dots,r\}$.
This assumption is not standard in the definition of a braid, but we use this convention
in our definition of colored braids. Intuitively, the indices $i\in\{1,\dots,r\}$ are colors which we assign to the strands of our braids,
as in the following picture:
\begin{equation}\label{eqn:e2-operads:colored-braids:picture}
\alpha = \vcenter{\hbox{\includegraphics[scale=0.66]{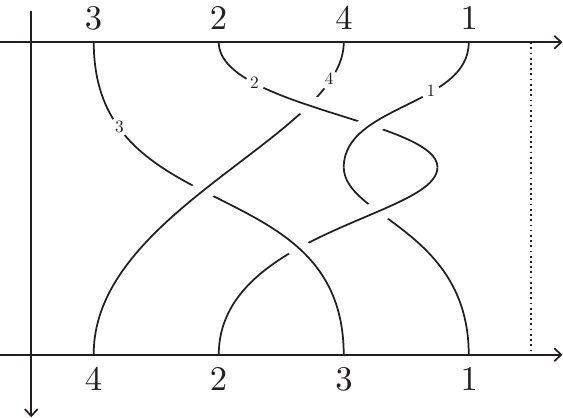}}}.
\end{equation}

Formally, The operad of colored braids is an operad in the category of groupoids $\CoB\in\Grd\Op$ whose components $\CoB(r)$ are groupoids
with the permutations on $r$ letters as objects and the isotopy classes of colored braids as morphisms.
The source (respectively, the target) of a morphism is the permutation of the set $\{1,\dots,r\}$
that corresponds to the permutation of the contact points $((z_1^0,\epsilon),\dots,(z_r^0,\epsilon))$
in the sequence $\alpha(0) = (\alpha_1(0),\dots,\alpha_r(0))$ (respectively, $\alpha(1) = (\alpha_1(1),\dots,\alpha_r(1))$).
For instance, the above example of colored braid depicts a morphism $\alpha\in\Mor\CoB$
with the permutation $s = (2,4,3,1)$ as source and the permutation $t = (3,4,1,2)$
as target.
The composition of braids is given by the standard concatenation operation on paths. Note simply that this operation
preserves the indexing when we consider a pair of composable morphisms in our groupoid.
Note also that our convention is to orient braids from the top to the bottom and we compose braids accordingly.

The action of the symmetric group $\Sigma_r$ on $\CoB(r)$
is given by the obvious re-indexing operation
of the strands of our braids, and the operadic composition operations $\circ_i: \CoB(k)\times\CoB(l)\rightarrow\CoB(k+l-1)$
are functors which are defined on morphisms by a cabling operation on the strands of our braids.
In short, to define a composite $\alpha\circ_i\beta$, where $\alpha\in\CoB(k)$ and $\beta\in\CoB(l)$,
we insert the braid $\beta$ on the $i$th strand of the braid $\alpha$,
as in the example given in the following picture:
\begin{equation}\label{eqn:e2-operads:colored-braids:composition-picture}
\vcenter{\hbox{\includegraphics[scale=0.8]{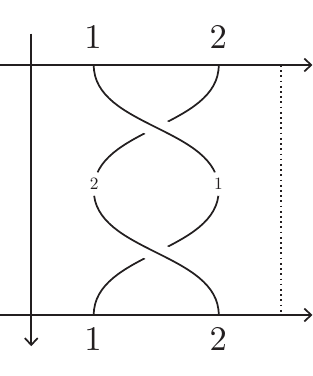}}}
\;\circ_1\;\vcenter{\hbox{\includegraphics[scale=0.8]{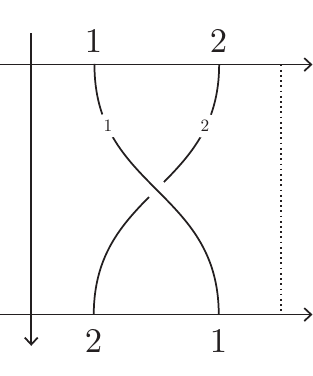}}}
\;=\;\vcenter{\hbox{\includegraphics[scale=0.8]{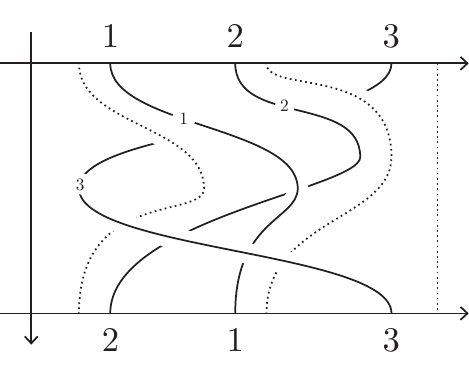}}}.
\end{equation}
The operadic unit is the trivial braid with one strand $1\in\CoB(1)$. Note that we actually have $\CoB(0) = \CoB(1) = *$
in this operad $\CoB$.

The following theorem gives the connection between the operad of little $2$-discs
and the operad of colored braids:

\begin{thm}\label{thm:e2-operads:fundamental-groupoid}
We have an equivalence in the category of operads in groupoids $\pi\DOp_2\sim\CoB$, where $\pi\DOp_2$ is the operad in groupoids
defined by the fundamental groupoids $\pi\DOp_2(r)$ of the spaces of little $2$-discs $\DOp_2(r)$, $r\in\NN$.
\end{thm}

\begin{proof}[Explanations and references]
In the category of operads in groupoids $\Grd\Op$, we say that a morphism is an equivalence $\phi: \POp\xrightarrow{\sim}\QOp$
when this morphism defines an equivalence of categories arity-wise $\phi: \POp(r)\xrightarrow{\sim}\QOp(r)$,
for each $r\in\NN$.
Then we say that operads in groupoids $\POp,\QOp\in\Grd\Op$ are equivalent when these operads can be connected by a zigzag
of equivalences $\POp\xleftarrow{\sim}\cdot\xrightarrow{\sim}\cdot\ldots\cdot\xrightarrow{\sim}\QOp$
in the category of operads in groupoids.

To define the operadic structure of the collection of fundamental groupoids $\pi\DOp_2 = \{\pi\DOp_2(r),r\in\NN\}$,
we just use that the fundamental groupoid functor is strongly symmetric monoidal.
The statement of this theorem is formally established in~\cite[Theorem I.5.3.4]{FresseBook}.
In short, the idea is to consider the operad such that $\pi\DOp_2\shortmid_{\DOp_1} = \{\pi\DOp_2(r)\shortmid_{\DOp_1(r)},r\in\NN\}$,
where $\pi\DOp_2(r)\shortmid_{\DOp_1(r)}\subset\pi\DOp_2(r)$
is the full subcategory of the operad in groupoids $\pi\DOp_2(r)$ generated by the configuration of little $2$-discs
whose centers lie on the equatorial axis $y=0$ (equivalently, we consider the image of the map $\iota: \DOp_1\hookrightarrow\DOp_2$
defined in~\S\ref{background:operad-bimodules}).
Then one can prove that we have an equivalence of operads in groupoids $\pi\DOp_2\shortmid_{\DOp_2}\xrightarrow{\sim}\CoB$
induced by the mapping $\DOp_2(r)\xrightarrow{\sim}\FOp(\mathring{\DD}{}^2,r)$
at the space level.

The embedding $\pi\DOp_2\shortmid_{\DOp_1}\hookrightarrow\pi\DOp_2$ also defines an equivalence of operad in groupoids for a trivial reason.
Therefore, we do have a zigzag of equivalences $\pi\DOp_2\xleftarrow{\sim}\pi\DOp_2\shortmid_{\DOp_2}\xrightarrow{\sim}\CoB$
which gives the required equivalence between the fundamental groupoid operad $\pi\DOp_2$
and the operad of colored braids $\CoB$.
\end{proof}

Now we can apply the classifying space functor $\DGB(-)$ to go back to spaces (or to simplicial sets).
This functor $\DGB: \Grd\rightarrow\Top$ is also strongly symmetric monoidal,
and hence, preserves operad structures.
For our purpose, we consider the operad $\DGB(\CoB)$ defined by the collection of the classifying spaces
of the colored braid groupoids $\DGB(\CoB(r))$, $r\in\NN$. We have the following result:

\begin{thm}[{Z. Fiedorowicz~\cite{Fiedorowicz}, see also~\cite[\S I.5.2]{FresseBook}}]\label{thm:e2-operads:classifying-space-operad}
We have a weak-equivalence $\DOp_2\sim\DGB(\CoB)$ in the category of topological operads.
\end{thm}

\begin{proof}[Explanations and references]
This theorem is established by Fiedorowicz in the cited reference~\cite{Fiedorowicz}
by using arguments of covering theory. In~\cite[\S I.5.3]{FresseBook},
we explain that we can also deduce this weak-equivalence relation $\DOp_2\sim\DGB(\CoB)$
from the result of the previous theorem and the observation that each space $\DOp_2(r)$
is an Eilenberg--MacLane space, so that we have a weak-equivalence $\DOp_2(r)\sim\DGB(\pi\DOp_2(r))$,
for each arity $r\in\NN$.
Indeed, we can elaborate on the proof of this result to establish that we have a weak-equivalence of operads $\DOp_2\sim\DGB(\pi\DOp_2)$,
and then we just use that the zigzag of equivalences of the proof of Theorem~\ref{thm:e2-operads:fundamental-groupoid}
induces a zigzag of weak-equivalences $\DGB(\pi\DOp_2)\xleftarrow{\sim}\DGB(\pi\DOp_2\shortmid_{\DOp_1})\xrightarrow{\sim}\DGB(\CoB)$
when we pass to classifying spaces.
\end{proof}

\subsubsection{The operad of parenthesized braids}\label{e2-operads:parenthesized-braids}
The operads of colored braids are not sufficient for our purpose. To define the Grothendieck--Teichm\"uller group, we need a variant of this operad,
which we call the parenthesized braid operad $\PaB$.

The objects of the colored braid operad form an operad in sets $\Ob\CoB$
which is identified with the permutation operad.
Thus we have $\Ob\CoB = \PiOp$.
To define the operad of parenthesized braids, we just take a pullback of the operad of coloured braids $\PaB = \omega^*\CoB$
along a morphism $\omega: \OmegaOp\rightarrow\PiOp$,
where $\OmegaOp$ is a free operad in sets generated by a single operation $\mu\in\OmegaOp(2)$, on which the symmetric group $\Sigma_2$
acts freely, and where we take an extra zero-ary operation $*\in\OmegaOp(0)$
such that $\mu\circ_1 * = 1 = \mu\circ_2 *$.
The elements of the sets $\OmegaOp(r)$ underlying this operad $\OmegaOp$ are identified with planar binary rooted trees
whose leaves are indexed by the elements of the set $\{1,\dots,r\}$ (with the convention that this set of binary trees
is equal to the one-point set in arity zero, so that we have $\OmegaOp(0) = *$).
In~\cite[\S I.6.1]{FresseBook},
we call this operad $\OmegaOp$ the magma operad, because the category of algebras associated to this operad
is identified with the category of Bourbaki's non-commutative magmas (with unit).

Formally, we define the groupoids $\PaB(r)$ underlying the operad $\PaB$ by taking $\Ob\PaB(r) := \OmegaOp(r)$ as object set
and $\Mor_{\PaB(r)}(p,q) := \Mor_{\CoB(r)}(\omega(p),\omega(q))$
for the morphism sets, for all $p,q\in\OmegaOp(r)$.
The operadic composition operations of the operad of parenthesized braids
are defined by an obvious lifting of the composition operations
of the operad of colored braids.
In~\cite[\S I.6.2]{FresseBook},
we represent a parenthesized braid by a braid whose contact points form the centers of a diadic decomposition
of the axis $y=0$ in the disc $\mathring{\DD}{}^2$,
because one can observe that such decomposition are in bijection with the elements of the magma operad.
For instance, the following braid
\begin{equation}\label{eqn:e2-operads:parenthesized-braids:picture}
\beta = \vcenter{\hbox{\includegraphics[scale=.8]{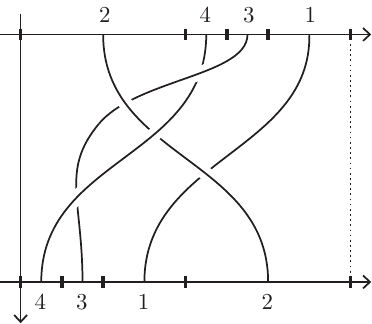}}}.
\end{equation}
represents a morphism of the groupoid $\PaB$ with the object $p = (1\ 2\ 4)\cdot(\mu\circ_2(\mu\circ_1\mu))$ as a source
and the object $q = (1\ 4\ 2\ 3)\cdot(\mu\circ_1(\mu\circ_1\mu))$ as a target.

In the operad $\PaB$, we consider the morphisms
\begin{equation}\label{eqn:e2-operads:parenthesized-braids:braiding-and-associator-picture}
\tau = \vcenter{\hbox{\includegraphics{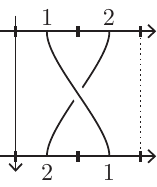}}}
\qquad\text{and}
\qquad\alpha = \vcenter{\hbox{\includegraphics{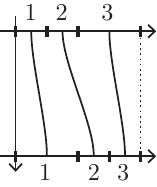}}},
\end{equation}
which we call the braiding and the associator respectively.

We aim to give an interpretation of the parenthesized braid operad in classical algebraic language.
We use that the object $\mu\in\OmegaOp(2)$ can be regarded as an abstract operation
on $2$ variables $\mu = \mu(x_1,x_2)$ (as we explain in the context of the $n$-Poisson operad
in our survey of Theorem~\ref{thm:background:little-discs-homology}).
We also use the notation of a tensor product for this operation $\mu(x_1,x_2) = x_1\otimes x_2$,
because we are going to see that $\mu$ represents a universal tensor product operation
within the operad $\PaB$.
We get that $(1\ 2)\mu = \mu(x_2,x_1)$ represent an operation $\mu(x_2,x_1) = x_2\otimes x_1$,
where the variables $(x_1,x_2)$ are transposed
when we use this variable interpretation of our operation. We also get that $\mu\circ_1\mu = \mu(\mu(x_1,x_2),x_3)$
represents the result of the substitution of the variable $x_1$ by the operation $\mu = \mu(x_1,x_2)$
in $\mu = \mu(x_1,x_2)$, while $\mu\circ_2\mu = \mu(x_1,\mu(x_2,x_3))$
represents the result of the substitution of the second variable $x_2$ by the same operation $\mu = \mu(x_1,x_2)$
with an index shift of the variables. We equivalently have $\mu(\mu(x_1,x_2),x_3) = (x_1\otimes x_2)\otimes x_3$
and $\mu(x_1,\mu(x_2,x_3)) = x_1\otimes(x_2\otimes x_3)$.
We accordingly get that the braiding $\tau = \tau(x_1,x_2)$ represents an isomorphism such that
\begin{equation}\label{eqn:e2-operads:parenthesized-braids:braiding}
\tau(x_1,x_2): x_1\otimes x_2\rightarrow x_2\otimes x_1
\end{equation}
in the morphism set $\Mor_{\PaB(2)}(\mu,(1\ 2)\mu)$ of our operad in groupoids $\PaB$,
while the associator $\alpha = \alpha(x_1,x_2,x_3)$
represents an isomorphism such that
\begin{equation}\label{eqn:e2-operads:parenthesized-braids:associator}
\alpha(x_1,x_2,x_3): (x_1\otimes x_2)\otimes x_3\rightarrow x_1\otimes(x_2\otimes x_3)
\end{equation}
in  $\Mor_{\PaB(3)}(\mu\circ_1\mu,\mu\circ_2\mu)$.
We also have the interpretation
\begin{equation}\label{eqn:e2-operads:parenthesized-braids:unit-relations}
x_1\otimes * = x_1 = *\otimes x_1
\end{equation}
of the operadic composition relations $\mu\circ_1 * = 1 = \mu\circ_2 *$.

We easily see that the braiding and the associator satisfy the following coherence relations with respect to the unit object:
\begin{equation}\label{eqn:e2-operads:parenthesized-braids:unit-coherence}
\alpha(e,x_1,x_2) = \alpha(x_1,e,x_2) = \alpha(x_1,x_2,e) = \id\quad\text{and}\quad\tau(x_1,e) = \id = \tau(e,x_1).
\end{equation}
We easily check, moreover, that the following pentagon equation
\begin{multline}\label{eqn:e2-operads:parenthesized-braids:pentagon}
x_1\otimes\alpha(x_2,x_3,x_4)\cdot\alpha(x_1,x_2\otimes x_3,x_4)\cdot\alpha(x_1,x_2,x_3)\otimes x_4\\
= \alpha(x_1,x_2,x_3\otimes x_4)\cdot\alpha(x_1\otimes x_2,x_3,x_4)
\end{multline}
holds in $\PaB(4)$, as well as the following hexagon relations
\begin{multline}\label{eqn:e2-operads:parenthesized-braids:hexagon-1eft}
x_2\otimes\tau(x_1,x_3)\cdot\alpha(x_2,x_1,x_3)\cdot\tau(x_1,x_2)\otimes x_3\\
= \alpha(x_2,x_3,x_1)\cdot\tau(x_1,x_2\otimes x_3)\cdot\alpha(x_1,x_2,x_3)
\end{multline}
and
\begin{multline}\label{eqn:e2-operads:parenthesized-braids:hexagon-right}
\tau(x_1,x_3)\otimes x_2\cdot\alpha(x_1,x_3,x_2)^{-1}\cdot x_1\otimes\tau(x_2,x_3)\\
= \alpha(x_3,x_1,x_2)^{-1}\cdot\tau(x_1\otimes x_2,x_3)\cdot\alpha(x_1,x_2,x_3)^{-1}
\end{multline}
in $\PaB(3)$ (exercise: check these relations). Note however that the isomorphism $\tau(x_1,x_2)$ is not involutive
in the sense that $\tau(x_2,x_1)\cdot\tau(x_1,x_2)\not=\id$,
because the braid which represents this isomorphism in the braid group
is not involutive too.

Thus we obtain that the object $\mu(x_1,x_2) = x_1\otimes x_2\in\Ob\PaB(2)$ can be interpreted as an abstract tensor product operation
that can be used to govern the structure of a braided monoidal category, with the object $*\in\OmegaOp(0) = \Ob\PaB(0)$
as a strict unit, but a general associativity isomorphism $\alpha = \alpha(x_1,x_2,x_3)$
and a general braiding isomorphism $\tau = \tau(x_1,x_2)$,
which are depicted in~(\ref{eqn:e2-operads:parenthesized-braids:braiding-and-associator-picture}).
The operad $\PaB$, equipped with these generating elements, is actually the universal operad
that governs such structures. This claim is established in the following statement:

\begin{thm}[{see~\cite[Theorem I.6.2.4]{FresseBook}}]\label{thm:e2-operads:parenthesized-braid-universality}
Let $\MOp\in\Cat\Op$ be an operad in the category of small categories $\CCat = \Cat$.
Fixing an operad morphism $\phi: \PaB\rightarrow\MOp$ amounts to fixing a unit object $e\in\Ob\MOp(0)$,
a product object $m\in\Ob\MOp(2)$,
an associativity isomorphism $a\in\Mor_{\MOp(3)}(m\circ_1 m,m\circ_2 m)$,
and a braiding isomorphism $c\in\Mor_{\MOp(2)}(m,(1\ 2) m)$
which satisfy strict unit relations $m\circ_1 e = 1 = e\circ_1 m$
together with the coherence constraints of the unit, pentagon
and hexagon relations~(\ref{eqn:e2-operads:parenthesized-braids:unit-relations}-\ref{eqn:e2-operads:parenthesized-braids:hexagon-right})
in the operad $\MOp$.
\end{thm}

\begin{proof}[Explanations]
This theorem is established in the cited reference.
The morphism associated to the quadruple $(m,e,a,c)$ given in the theorem
is obviously determined by the formulas $\phi(\mu) = m$, $\phi(*) = e$, $\phi(\alpha) = a$ and $\phi(\tau) = c$.
The claim is that this assignments determines a well-defined morphism on $\PaB$.
The proof of this result follows from a combination
of an operadic interpretation of the MacLane coherence theorem and of the classical presentation of the braid group
by generators and relations.

This statement implies that the category of algebras governed by the operad $\PaB$ in the category of categories
is identified with a category of braided monoidal categories
with a strict unit,
but general associativity isomorphisms.
Let us mention that the operad of colored braid has a similar interpretation, but where we have strict associativity identities
instead of associativity isomorphisms.
We refer to~\cite[\S I.6.2]{FresseBook} for more detailed explanations on this subject.
\end{proof}

\subsubsection{The Grothendieck--Teichm\"uller group}\label{e2-operads:grothendieck-teichmuller-group}
The Grothendieck--Teichm\"uller group is defined as a group of automorphisms of the parenthesized braid operad.
To be more precise, we have to consider completions of this operad in applications.
These completions operations are performed at the groupoid level.
In what follows, we mainly consider the case of the Malcev completion, which we denote by $G\sphat_{\QQ}$ for any groupoid $G\in\Grd$,
and the profinite completion, which we denote by $G\sphat$ (yet another natural example of completion operation
is the $p$-profinite completion,
but we do not consider this variant of the profinite completion
in this survey).
In all cases, the considered completion operation does not change the object sets of our groupoids
and are natural generalizations, for groupoids, of classical completion operations
on groups.
Recall simply that the Malcev completion of groups is an extension of the classical rationalization of abelian groups
which combines a pro-nilpotent completion with a rationalization operation.
In the case of a free group for instance, we can identify the elements of the Malcev completion
with infinite products of iterated commutators with rational exponents. (We refer to~\cite[\S I.8]{FresseBook}
for a detailed survey of this subject.)

To define the rationalization $\PaB\sphat_{\QQ}$ (respectively, the profinite completion $\PaB\sphat$) of our operad,
we just perform the arity-wise completion operation $\PaB\sphat_{\QQ}(r) = \PaB(r)\sphat_{\QQ}$ (respectively, $\PaB\sphat(r) = \PaB(r)\sphat$).
Then we define the rational (respectively, the profinite) Grothendieck--Teichm\"uller group $\GT(\QQ)$ (respectively, $\GT\sphat$)
as the group of automorphisms of the operad $\PaB\sphat_{\QQ}$ (respectively, $\PaB\sphat$)
which reduce to the identity mapping on objects.

Let us mention that any morphism 
such that $\phi: \PaB\rightarrow\PaB\sphat_{\QQ}$
admits a unique extension to the completed operad $\phi\sphat: \PaB\sphat_{\QQ}\rightarrow\PaB\sphat_{\QQ}$
and a similar result holds in the context of the profinite completion $\PaB\sphat$.
By Theorem~\ref{thm:e2-operads:parenthesized-braid-universality},
such a morphism $\phi: \PaB\rightarrow\PaB\sphat_{\QQ}$ (respectively, $\phi: \PaB\rightarrow\PaB\sphat$)
is fully determined by giving a triple $(m,a,c)$
such that $m = \phi(\mu)$, $a = \phi(\alpha)$ and $c = \phi(\tau)$. Note that we automatically have $\phi(*) = *$
since $\PaB(0) = *\Rightarrow\PaB(0)^{\QQ} = *$ and similarly $\PaB(0)\sphat = *$.
For our purpose, we also set $m = \phi(\mu) = \mu$ in order to ensure that $\phi$ is the identity mapping on objects.

Now we necessarily have $\phi(\tau) = \tau\cdot\tau^{2\nu}$ for some parameter $\nu\in\QQ$ (respectively, $\nu\in\ZZ$),
where we use that $\tau^2\in\Mor_{\PaB(2)}(\mu,\mu)$
is identified with an element of the pure braid group on $2$ strands $P_2$
and $\tau^{2\nu}$ with $\nu\in\QQ$ (respectively, $\nu\in\ZZ\sphat$)
represents an element in the Malcev completion of this group $(P_2)\sphat_{\QQ}$ (respectively in the profinite completion $P_2\sphat$).
We similarly have $\phi(\alpha) = \alpha\cdot f$ for some morphism $f\in\Mor_{\PaB(3)\sphat_{\QQ}}(\mu\circ_1\mu,\mu\circ_1\mu)$
(respectively, $f\in\Mor_{\PaB(3)\sphat}(\mu\circ_1\mu,\mu\circ_1\mu)$),
which is represented by an element of the Malcev completion of the pure braid group on three strand $(P_3)\sphat_{\QQ}$
(respectively, of the profinite completion $P_3\sphat$).
We have $P_3 = \langle K\rangle\times\langle x_{12},x_{23}\rangle$, where $K$ is a central element, whereas $x_{12}$ and $x_{23}$
denote the pure braids
such that:
\begin{equation}\label{eqn:e2-operads:grothendieck-teichmuller-group:pure-3-braid-generators}
x_{12} = \vcenter{\hbox{\includegraphics{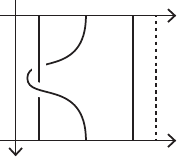}}},\quad x_{23} = \vcenter{\hbox{\includegraphics{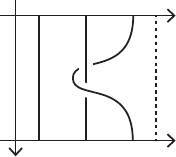}}}.
\end{equation}
We can easily deduce from the unit relation $a\circ_2 * = \id$ that the associator $\phi(\alpha) = \alpha\cdot f$
has no factor in $\langle K\rangle\sphat_{\QQ}$ (respectively, in $\langle K\rangle\sphat$).
We eventually get that our morphism $\phi$ is determined by an assignment
of the form:
\begin{equation}\label{eqn:e2-operads:grothendieck-teichmuller-group:mapping}
\phi(\mu) := \mu,\quad\phi(\tau) := \tau\cdot\tau^{2\nu} = \tau^{\lambda},\quad\phi(\alpha) := \alpha\cdot f(x_{12},x_{23}),
\end{equation}
where we set $\lambda = 1+2\nu$ for $\nu\in\QQ$ (respectively, $\nu\in\ZZ\sphat$)
and we assume $f = f(x_{12},x_{23})\in\langle x_{12},x_{23}\rangle\sphat_{\QQ}$ (respectively, $f = f(x_{12},x_{23})\in\langle x_{12},x_{23}\rangle\sphat$).
In what follows, we use that this element $f$ in the Malcev completion (respectively, in the profinite completion)
of a free group $F = \langle x_{12},x_{23}\rangle$ can (at least symbolically) be represented by a series
on two abstract variables $f = f(x,y)$.

Then one can prove that the unit relations $a\circ_1 e = \id = a\circ_2 e$
in the coherence constraints of Theorem~\ref{thm:e2-operads:parenthesized-braid-universality}
are equivalent to the identities:
\begin{equation}\label{eqn:e2-operads:grothendieck-teichmuller-group:unit-constraint}
f(x,1) = x = f(1,x)
\end{equation}
while the hexagon relations are equivalent to the following system of equations
\begin{gather}\label{eqn:e2-operads:grothendieck-teichmuller-group:symmetry-constraint}
f(x,y)\cdot f(y,x) = 1,\\
\label{eqn:e2-operads:grothendieck-teichmuller-group:hexagon-constraint}
f(x,y)\cdot x^{\nu}\cdot f(z,x)\cdot z^{\nu}\cdot f(y,z)\cdot y^{\nu} = 1,
\end{gather}
for a triple of variables $(x,y,z)$ such that $z y x = 1$.
The pentagon equation is equivalent to the following relation in the Malcev completion of the pure braid group on $4$ strands $(P_4)\sphat_{\QQ}$
(respectively, in the profinite completion $P_4\sphat$):
\begin{equation}\label{eqn:e2-operads:grothendieck-teichmuller-group:pentagon-constraint}
f(x_{23},x_{34}) f(x_{13} x_{12},x_{34} x_{24}) f(x_{12},x_{23}) = f(x_{12},x_{24} x_{23}) f(x_{23} x_{13},x_{34}),
\end{equation}
where in general, we use the notation $x_{ij}$ for the pure braid group elements such that:
\begin{equation}\label{eqn:e2-operads:grothendieck-teichmuller-group:pure-r-braid-generators}
x_{ij} = \vcenter{\hbox{\includegraphics[scale=.66]{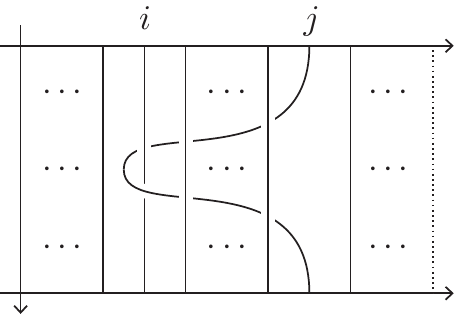}}}.
\end{equation}
(We refer to~\cite{Drinfeld} and to~\cite[Proof of Proposition I.11.1.4]{FresseBook} for a more detailed analysis of these equations.)

Let us mention that the composition of morphisms
corresponds to the following operation
\begin{equation}\label{eqn:e2-operads:grothendieck-teichmuller-group:composition}
(\lambda,f(x,y))*(\mu,g(x,y)) = (\lambda\mu,f(x,y)\cdot g(x^{\lambda},f(x,y)^{-1}\cdot y^{\lambda}\cdot f(x,y)))
\end{equation}
on this set of pairs $(\lambda,f)$ (exercise: check this formula).
Eventually, we get that an element of the Grothendieck--Teichm\"uller group $\gamma\in\GT(\QQ)$ (respectively, $\gamma\in\GT\sphat$),
which corresponds to a morphism $\phi: \PaB\rightarrow\PaB\sphat_{\QQ}$ (respectively, $\phi: \PaB\rightarrow\PaB\sphat$)
that induces an isomorphism on the Malcev completion $\phi\sphat: \PaB\sphat_{\QQ}\xrightarrow{\simeq}\PaB\sphat_{\QQ}$
(respectively on the profinite completion $\phi\sphat: \PaB\sphat\xrightarrow{\simeq}\PaB\sphat$),
is equivalent to a pair $(\lambda,f)\in\QQ\times\langle x_{12},x_{23}\rangle\sphat_{\QQ}$
(respectively, $(\lambda,f)\in\ZZ\sphat\times\langle x_{12},x_{23}\rangle\sphat$)
which is invertible with respect to this composition operation.
In the case of the rational Grothendieck--Teichm\"uller group $\GT(\QQ)$ we can see that a necessary and sufficient condition
for this invertibility property is given by $\lambda\in\QQ^{\times}$ (see~\cite{Drinfeld} and~\cite[Proposition I.11.1.5]{FresseBook}),
but we do not have such a simple characterization of the isomorphisms
in the profinite setting.

This representation of the elements of the Grothendieck--Teichm\"uller group in terms of pairs $(\lambda,f)$
and the equations~(\ref{eqn:e2-operads:grothendieck-teichmuller-group:unit-constraint}-\ref{eqn:e2-operads:grothendieck-teichmuller-group:pentagon-constraint})
is actually Drinfeld's original definition of the Grothendieck--Teichm\"uller group
in~\cite{Drinfeld}.
The correspondence between this definition and the operadic definition which we summarize in this paragraph is established with full details
in the book~\cite[\S I.11.1]{FresseBook}, but the ideas underlying this operadic interpretation
were already implicitly present in Drinfeld's work~\cite{Drinfeld}. We also refer to~\cite{BarNatan}
for another formalization of this interpretation,
which uses ideas close to the language of universal algebra.
In the introduction of the paper, we mentioned that the Grothendieck--Teichm\"uller group
was defined by using ideas of the Grothendieck program in Galois theory. In fact, we have an embedding $\Gal(\bar{\QQ}/\QQ)\hookrightarrow\GT\sphat$
which is defined by using an action of the absolute Galois group on genus zero curves with marked points (see~\cite{Drinfeld}).
For the rational Grothendieck--Teichm\"uller group, a result of F. Brown's (see~\cite{BrownTate})
implies that we have an analogous embedding $\Gal_{\MT(\ZZ)}\hookrightarrow\GT(\QQ)$,
where $\Gal_{\MT(\ZZ)}$ now denotes the motivic Galois group
of a category of integral mixed Tate motives (see also~\cite{DeligneGoncharov,Terasoma}
for the definition of this group and for the definition of this mapping).

We go back to the definition of the Grothendieck--Teichm\"uller group $\GT(\QQ)$ (respectively, $\GT\sphat$)
in terms of operad isomorphisms $\phi_{\gamma}\sphat: \PaB\sphat_{\QQ}\xrightarrow{\simeq}\PaB\sphat_{\QQ}$
(respectively, $\phi_{\gamma}\sphat: \PaB\sphat\xrightarrow{\simeq}\PaB\sphat$).
We can regard the classifying space operad $\EOp_2^{\QQ} = \DGB(\PaB\sphat_{\QQ})$ (respectively, $\EOp_2\sphat = \DGB(\PaB\sphat$)
as a model for the rationalization (respectively, for the profinite completion)
of the $E_2$-operad $\EOp_2 = \DGB(\PaB)$.
We deduce from the functoriality of the classifying space construction that any element $\gamma\in\GT(\QQ)$ (respectively, $\gamma\in\GT\sphat$)
induces an automorphism $\phi_{\gamma}\sphat: \EOp_2^{\QQ}\xrightarrow{\sim}\EOp_2^{\QQ}$
(respectively, $\phi_{\gamma}: \EOp_2\sphat\xrightarrow{\sim}\EOp_2\sphat$)
at the topological operad level,
and hence, defines an element in the homotopy automorphism space $\Aut_{\Op}^h(\EOp_2^{\QQ})$ (respectively, $\Aut_{\Op}^h(\EOp_2\sphat)$).
We claim that this correspondence $\gamma\mapsto\phi_{\gamma}\sphat$
induces a bijection when we pass to the group of homotopy classes
of homotopy automorphisms.
We deduce this claim from a more precise determination of our spaces of homotopy automorphisms in the homotopy category.
We give the expression of this result in the rational setting first:

\begin{thm}[{B. Fresse \cite[Theorem III.5.2.5]{FresseBook}}]\label{thm:e2-operads:rational-homotopy-automorphisms}
We have $\Aut_{\Op}^h(\EOp_2^{\QQ})\sim\GT(\QQ)\ltimes\DGB(\QQ)$.
\end{thm}

\begin{proof}[Explanations]
The factor $\DGB(\QQ)$ in the expression of this theorem corresponds to a rationalization of the group of rotations $\SO(2)\sim\DGB(\ZZ)$
which acts on the little $2$-discs model of the class of $E_2$-operads $\EOp_2 = \DOp_2$
by rotating the configurations of little $2$-discs
in this operad $\DOp_2$.
We equip this factor $\DGB(\QQ)$ with the obvious additive group structure.

To define an action of the Grothendieck--Teichm\"uller group on this group $\DGB(\QQ)$
we use that any isomorphism $\phi_{\gamma}\sphat: \PaB\sphat_{\QQ}\rightarrow\PaB\sphat_{\QQ}$
determines an automorphism of the Malcev completion of the pure braid group $(P_2)\sphat_{\QQ}\simeq\QQ$
when we consider the action of this isomorphism $\phi_{\gamma}\sphat$ on the automorphism group of the object $\mu\in\Ob\PaB(2)$
and we use the relation $\Mor_{\PaB(2)\sphat_{\QQ}}(\mu,\mu) = (P_2)\sphat_{\QQ}$ (we refer to~\cite[\S III.5.2]{FresseBook} for details).

Let us insist that we consider derived homotopy automorphism spaces in the statement of this theorem.
In the model category approach, these homotopy automorphism spaces are defined by taking the actual spaces of homotopy automorphism spaces
associated to a cofibrant-fibrant replacement $\ROp_2^{\QQ}$
of our operad $\EOp_2^{\QQ}$.
To associate an element of this derived homotopy automorphism space to an element of the Grothendieck--Teichm\"uller group $\gamma\in\GT(\QQ)$
we use that an automorphism $\phi_{\gamma}\sphat: \EOp_2^{\QQ}\rightarrow\EOp_2^{\QQ}$
automatically admits a lifting to this cofibrant-fibrant replacement $\ROp_2^{\QQ}$.
The claim is that all homotopy automorphisms of this cofibrant-fibrant model $\ROp_2^{\QQ}$
are homotopic to such morphisms, and that this correspondence gives
all the homotopy of the space $\Aut_{\Op}^h(\EOp_2^{\QQ})$
up to the factor $\DGB(\QQ)$.

The book \cite[\S\S III.1-5]{FresseBook} gives a proof of this result by using spectral sequence methods
and an operadic cohomology theory which provides approximations of our homotopy automorphism spaces.
This method is close to the methods which are used in the next sections, when we study the homotopy automorphism spaces of $E_n$-operads
for any value of the dimension parameter $n\geq 2$.
\end{proof}

We get the same result as in Theorem~\ref{thm:e2-operads:profinire-homotopy-automorphisms}
in the profinite setting:

\begin{thm}[{G. Horel \cite{Horel}}]\label{thm:e2-operads:profinire-homotopy-automorphisms}
We have $\Aut_{\Op}^h(\DOp_2\sphat)\sim\GT\sphat\ltimes\DGB(\ZZ\sphat)$.
\end{thm}

\begin{proof}[Explanations]
The article \cite{Horel} gives a proof of the result of this theorem
by using the correspondence with groupoids.
In short, the idea of this paper is to observe that the operad $\PaB\sphat$
represents a cofibrant object with respect to some model structure
on the category of operads in groupoids.
Then we can use model category arguments (combined with higher category methods) to prove that we can transport the computation
of the homotopy automorphism space $\Aut_{\Op}^h(\EOp_2\sphat)$
to the computation of the homotopy automorphism space
associated to this object $\PaB$
in the category of operads in groupoids.
\end{proof}

\subsubsection{The Drinfeld--Kohno Lie algebra operad}\label{e2-operads:drinfeld-kohno-lie-algebra-operads}
Besides the colored braid and the parenthesized braid operads, which are defined by using the structures
of the braid groups, we consider operads in Lie algebras
which are associated to infinitesimal versions of the pure braid groups.
To be explicit, in these infinitesimal versions, we consider the Drinfeld--Kohno Lie algebras (also called
the Lie algebras of infinitesimal braids) which are defined by a presentation of the form:
\begin{equation}\label{eqn:e2-operads:drinfeld-kohno-lie-algebra-operads:presentation}
\palg(r) = \LLie(t_{ij},\{i,j\}\subset\{1,\dots,r\})/<[t_{ij},t_{kl}],[t_{ij},t_{ik}+t_{jk}]>,
\end{equation}
for $r\in\NN$, where $\LLie(-)$ denotes the free Lie algebra functor, we associate a generator $t_{ij}$ such that $t_{ij} = t_{ji}$
to each pair $\{i\not=j\}\subset\{1,\dots,r\}$,
and we take the ideal generated by the commutation relations
\begin{equation}\label{eqn:e2-operads:drinfeld-kohno-lie-algebra-operads:commutation-relation}
[t_{ij},t_{kl}]\equiv 0,
\end{equation}
for all quadruples $\{i,j,k,l\}\subset\{1,\dots,r\}$ such that $\sharp\{i,j,k,l\} = 4$,
together with the Yang-Baxter relations
\begin{equation}\label{eqn:e2-operads:drinfeld-kohno-lie-algebra-operads:yang-baxter-relation}
[t_{ij},t_{ik}+t_{jk}]\equiv 0,
\end{equation}
for all triples $\{i,j,k\}\subset\{1,\dots,r\}$ such that $\sharp\{i,j,k\} = 3$.
This definition makes sense over any ground ring $\kk$, but from the next paragraph on, we will assume that the ground ring
is a field of characteristic zero.

Note that this Lie algebra $\palg(r)$ inherits a weight grading from the free Lie algebra since this ideal is generated by homogeneous relations.
To be more explicit, if we use the notation $L_m = \LLie_m(-)$, for the homogeneous component of weight $m$
of the free Lie algebra $L = \LLie(t_{ij},\{i,j\}\subset\{1,\dots,r\})$,
then we have the decomposition $\palg(r) = \bigoplus_{m\geq 1}\palg(r)_m$,
where we set $\palg(r)_m = L_m/L_m\cap<[t_{ij},t_{kl}],[t_{ij},t_{ik}+t_{jk}]>$,
for $m\geq 1$.
In fact, we have the identity $\palg(r)_* = \gr_*^{\Gamma} P_r$, where on the right-hand side
we consider the graded Lie algebra of the sub-quotients of the central series
filtration of the pure braid group $\gr_*^{\Gamma} P_r$ (see~\cite[Theorem I.10.0.4]{FresseBook} for a detailed proof of this statement).

The collection $\palg = \{\palg(r),r\in\NN\}$ inherits the structure of an operad in the category of Lie algebras,
where we take the direct sum of Lie algebras to define our symmetric monoidal structure.
The action of the symmetric group $\Sigma_r$ on the Lie algebra $\palg(r)$ is defined, on generators,
by the obvious re-indexing operation $\sigma\cdot t_{ij} = t_{\sigma(i)\sigma(j)}$,
for all $\sigma\in\Sigma_r$.
The composition products are given by Lie algebra morphisms of the form
\begin{equation}\label{eqn:e2-operads:drinfeld-kohno-lie-algebra-operads:composition-operations}
\circ_i: \palg(k)\oplus\palg(l)\rightarrow\palg(k+l-1),
\end{equation}
defined for all $k,l\in\NN$ and $i\in\{1,dots,k\}$,
and which satisfy the equivariance, unit and associativity relations of operads
in the category of Lie algebras.
For generators $t_{a b}\in\palg(k)$ and $t_{c d}\in\palg(l)$,
we explicitly set:
\begin{align}\label{eqn:e2-operads:drinfeld-kohno-lie-algebra-operads:left-composition-definition}
t_{a b}\circ_i 0 & = \begin{cases} t_{a+l-1 b+l-1},\quad\text{if $i<a<b$}, \\
t_{a b+l-1}+\cdots+t_{a+l-1 b+l-1},\quad\text{if $i=a<b$}, \\
t_{a b+l-1},\quad\text{if $a<i<b$}, \\
t_{a b}+\cdots+t_{a j+l-1},\quad\text{if $a<i=b$}, \\
t_{a b},\quad\text{if $a<b<i$},
\end{cases}
\intertext{and}\label{eqn:e2-operads:drinfeld-kohno-lie-algebra-operads:right-composition-definition}
0\circ_i t_{c d} & = t_{c+i-1 d+i-1}\quad\text{for all $i$}.
\end{align}
The operadic unit is just given by the zero morphism $0: 0\rightarrow\palg(1)$ with values the zero object $\palg(1) = 0$.

In fact, these operations reflect the composition structure of the operad of colored braids,
in the sense that we can identify the components of homogeneous weight
of this operad $\palg(-)_m$, $m\geq 1$,
with the fibers of a tower of operads $\CoB/\DGGamma_m\CoB$, $m\geq 1$, which we deduce from the central series filtration
of the pure braid group (we refer to~\cite[\S I.10.1]{FresseBook} for more explanations on this correspondence).

We call this operad $\palg$ the Drinfeld--Kohno Lie algebra operad. We consider generalizations of this operad
when we study the Sullivan model of $E_n$-operads. This subsequent study is our main motivation
for the recollections of this paragraph, but the Drinfeld--Kohno Lie algebra operad also occurs in the theory of Drinfeld's associators
and in the definition of a graded version of the Grothendieck--Teichm\"uller group.
We just give a brief overview of this subject to complete the account of this section.

\subsubsection{The operad of chord diagrams and associators}\label{e2-operads:chord-diagram-operad}
To define the set of Drinfeld's associators, we consider an operad in groupoids, the chord diagram operad $\CD\sphat_{\kk}$,
defined over any characteristic zero field $\kk$,
and such that we have the relation $\CD(r)\sphat_{\kk} = \exp\hat{\palg}(r)$ for each $r\in\NN$,
where we consider the exponential group associated to a completion
of the Drinfeld Kohno Lie algebra $\hat{\palg}(r)$.

To be more precise, we explained in the previous paragraph that the Drinfeld Kohno Lie algebra
admits a weight decomposition $\palg(r) = \bigoplus_{m\geq 1}\palg(r)_m$
such that $\palg(r)_m = L_m/L_m\cap<[t_{ij},t_{kl}],[t_{ij},t_{ik}+t_{jk}]>$,
where $L_m = \LLie_m(-)$ denotes the homogeneous component of weight $m$ of the free Lie algebra $L = \LLie(-)$.
To form the completed Lie algebra $\hat{\palg}(r)$, we just replace the direct sum of this decomposition
by a product.
Thus we have $\hat{\palg}(r) = \prod_{m\geq 1}\palg(r)_m$ so that the elements of this completed Lie algebra $\hat{\palg}(r)$
are represented by infinite series of Lie polynomials (modulo the ideal
generated by the commutation and the Yang-Baxter relations).
The exponential group $\CD(r)\sphat_{\kk} = \exp\hat{\palg}(r)$ consists of formal exponential elements $e^{\xi}$, $\xi\in\hat{\palg}(r)$,
together with the group operation given by the Campbell-Hausdorff formula
at the level of the completed Lie algebra $\hat{\palg}(r)$.
Note that we identify this group $\CD(r)\sphat_{\kk} = \exp\hat{\palg}(r)$ with a groupoid with a single object
when we regard the chord diagram operad $\CD\sphat_{\kk} = \{\exp\hat{\palg}(r),r\in\NN\}$
as an operad in the category of groupoids.
The structure operations of this operad $\CD\sphat_{\kk}$ are induced by the structure operations
of the Drinfeld--Kohno Lie algebra operad on $\hat{\palg}$.

Let us mention that we also have an identity $\CD(r)\sphat = \GLike(\hat{\UFree}(\hat{\palg}(r)))$,
where we consider the set of group-like elements $\GLike(-)$ in the complete enveloping algebra
of the Drinfeld--Kohno Lie algebra $\hat{\UFree}(\hat{\palg}(r))$.
Indeed, the group-like elements are identified with actual exponential series
of Lie algebra elements $\xi\in\palg(r)$
within the complete enveloping algebras $\hat{\UFree}(\hat{\palg}(r))$.
The name ``chord diagram'' comes the theory of Vassiliev invariants, where a monomial $t_{i_1 j_1}\cdots t_{i_m j_m}\in\hat{\UFree}(\hat{\palg}(r))$
is associated to a diagram with $r$ vertical strands numbered from $1$ to $r$, and $l$ chords corresponding to the factors $t_{i_k j_k}$,
as in the following picture:
\begin{equation}\label{eqn:e2-operads:chord-diagram-operad:picture}
t_{1 2} t_{1 2} t_{3 6} t_{2 4} =
\vcenter{\xymatrix@W=0mm@H=0mm@R=1.33mm@C=2.66mm@M=0mm{ *+<2pt>{\scriptstyle{1}}\ar@{-}[ddddd] & *+<2pt>{\scriptstyle{2}}\ar@{-}[ddddd] &
*+<2pt>{\scriptstyle{3}}\ar@{-}[ddddd] & *+<2pt>{\scriptstyle{4}}\ar@{-}[ddddd] &
*+<2pt>{\scriptstyle{5}}\ar@{-}[ddddd] & *+<2pt>{\scriptstyle{6}}\ar@{-}[ddddd] \\
& *{\scriptscriptstyle{\bullet}}\ar@{-}[rr] && *{\scriptscriptstyle{\bullet}} && \\
&& *{\scriptscriptstyle{\bullet}}\ar@{-}[rrr] &&& *{\scriptscriptstyle{\bullet}} \\
*{\scriptscriptstyle{\bullet}}\ar@{-}[r] & *{\scriptscriptstyle{\bullet}} &&&& \\
*{\scriptscriptstyle{\bullet}}\ar@{-}[r] & *{\scriptscriptstyle{\bullet}} &&&& \\
&&&&& }}.
\end{equation}
The composition products of the chord diagram operad have a simple description
in terms of chord diagram insertions too.

In the previous paragraph, we explained that the components of homogeneous weight of the Drinfeld--Kohno Lie algebra operad
represent the fibers of a tower decomposition
of the parenthesized braid operad (and of the colored braid operad equivalently).
In fact, a stronger result holds when we work over a field of characteristic zero.
To be more explicit, we consider the Malcev completion of the operad $\PaB$,
and a natural extension of this construction for ground fields
such that $\QQ\subset\kk$.
Then we may wonder about the existence of equivalences of operads in groupoids $\phi_{\alpha}\sphat: \PaB\sphat_{\kk}\xrightarrow{\sim}\CD\sphat_{\kk}$,
which would be equivalent to a splitting of this tower decomposition over $\kk$.
By Theorem~\ref{thm:e2-operads:parenthesized-braid-universality},
the morphism of operad in groupoids $\phi_a: \PaB\xrightarrow{\sim}\CD\sphat_{\kk}$
which would induce such an equivalence on the completion
is determined by the choice of a braiding $c\in\exp\hat{\palg}(2)$
and of an associativity isomorphism $a\in\exp\hat{\palg}(3)$.
The braiding has the form $c = \exp(\kappa t_{12}/2)$, for some parameter $\kappa\in\kk^{\times}$, since $\hat{\palg}(2) = \kk t_{12}$,
and one can prove that the associator is necessarily of the form $a = \exp f(t_{12},t_{23})$,
for some Lie power series $f(t_{12},t_{23})\in\hat{\LLie}(t_{12},t_{23})$.
Thus, the existence of an equivalence of operads in groupoids $\phi_{\alpha}\sphat: \PaB\sphat_{\kk}\xrightarrow{\sim}\CD\sphat_{\kk}$
reduces to the existence of such a Lie power series $f(t_{12},t_{23})\in\hat{\LLie}(t_{12},t_{23})$
such that $a = \exp f(t_{12},t_{23})$ satisfies the unit, pentagon and hexagon constraints
of Theorem~\ref{thm:e2-operads:parenthesized-braid-universality},
for some given parameter $\kappa\in\kk^{\times}$.

The set of Drinfeld's associators precisely refers to this particular set of associators $a = \exp f(t_{12},t_{23})$
which we associate to the chord diagram operad $\CD\sphat_{\kk}$.
This notion was introduced by Drinfeld in the paper~\cite{Drinfeld}, to which we also refer for an explicit
expression of the unit, pentagon and hexagon constraints (see also the survey of~\cite[\S I.10.2]{FresseBook}).
Let us mention that further reductions occur in the pentagon and hexagon constraints in the definition
of Drinfeld's associators.
Indeed, by a result of Furusho (see~\cite{FurushoEquations}),
the hexagon constraints can be satisfied as soon as we have a power series that fulfills
the unit and the pentagon constraints.

We have the following main result:

\begin{thm}[{V.I. Drinfeld~\cite{Drinfeld}}]\label{thm:e2-operads:drinfeld-associator-existence}
The set of Drinfeld's associators is not empty, for any choice of field of characteristic zero as ground field $\kk$, including $\kk = \QQ$,
so that we do have an operad morphism $\phi_{\alpha}: \PaB\rightarrow\CD\sphat_{\QQ}$
which induces an equivalence when we pass to the Malcev completion $\phi_{\alpha}\sphat: \PaB\sphat_{\QQ}\xrightarrow{\sim}\CD\sphat_{\QQ}$.
\end{thm}

\begin{proof}[Explanations and references]
In~\cite{Drinfeld}, Drinfeld gives an explicit construction of a complex associator by using the monodromy of the Knizhnik--Zamolodchikov connection.
This associator, which is usually called the Knizhnik--Zamolodchikov associator in the literature,
can also be identified with a generating series of polyzeta values.
Descent arguments can be used to establish the existence of a rational associator from the existence
of this complex associator (see again~\cite{Drinfeld}
and~\cite{BarNatan} for different proofs of this descent statement),
so that the result of this theorem holds over $\kk = \QQ$, and not only over $\kk = \CC$.

Let us mention that another explicit definition of an associator, defined over the reals,
is given by Alekseev--Torossian in~\cite{AlekseevTorossian},
by using constructions introduced by Kontsevich
in his proof of the formality of the operads of little discs (see~\cite{KontsevichMotives}).
\end{proof}

\subsubsection{The operad of parenthesized chord diagrams, the graded Grothendieck--Teich\-m\"uller group, and other related objects}\label{e2-operads:graded-grothendieck-teichmuller-group}
The existence of associators can be used to get insights into the structure the rational Grothendieck--Teichm\"uller group $\GT(\QQ)$.
Indeed, the definition implies that the set of associators inherits a free and transitive action of the rational Grothendieck--Teichm\"uller group.
To go further into the applications of associators, one introduces a parenthesized version of the chord diagram operad $\PaCD\sphat_{\QQ}$
(by using the same pullback construction as in the case of the parenthesized braid operad $\PaB$)
and a group of automorphisms, denoted by $\GRT(\QQ)$, which we associate to this object $\PaCD\sphat_{\QQ}$.
Then one can easily check that every equivalence of operads in groupoids $\phi\sphat_a: \PaB\sphat_{\QQ}\xrightarrow{\sim}\CD\sphat_{\QQ}$
lifts to an isomorphism $\phi\sphat_a: \PaB\sphat_{\QQ}\xrightarrow{\simeq}\PaCD\sphat_{\QQ}$
so that the existence of rational associators implies the existence of a group isomorphism $\GT(\QQ)\simeq\GRT(\QQ)$
by passing to automorphism groups.

This group $\GRT(\QQ)$ is usually called the graded Grothendieck--Teichm\"uller group in the literature,
because this group is identified with the pro-algebraic group
associated to a Lie algebra $\grt$, the graded Grothendieck--Teichm\"uller Lie algebra,
which is equipped with a weight decomposition $\grt = \bigoplus_{m\geq 0}\grt_m$
such that $\grt_m\simeq\DGF_m\GT(\QQ)/\DGF_{m+1}\GT$, for all $m\geq 1$,
for some natural filtration of the Grothendieck--Teichm\"uller group $\GT(\QQ) = \DGF_0\GT(\QQ)\supset\DGF_1\GT(\QQ)\supset\cdots\supset\DGF_m\GT(\QQ)\supset\cdots$.
We again refer to~\cite{Drinfeld} for a proof of these results (see also the survey of~\cite[\S I.11.4]{FresseBook}).

The groups $\GT(\QQ)$, $\GRT(\QQ)$, and the Lie algebra $\grt$ are also related to other objects
of arithmetic geometry and group theory.
In~\S\ref{e2-operads:grothendieck-teichmuller-group}, we already recalled that, by a result of F. Brown (see~\cite{BrownTate}),
the Grothendieck--Teichm\"uller group $\GT(\QQ)$ contains a realization of the motivic Galois group
of a category of integral mixed Tate motives $\Gal_{\MT(\ZZ)}$.
In fact, one conjectures that these groups are isomorphic (Deligne-Ihara).
Furthermore, one can prove that the Galois group $\Gal_{\MT(\ZZ)}$ reduces to the semi-direct product of the multiplicative group
with the prounipotent completion of a free group on a sequence of generators $s_3,s_5,\dots,s_{2n+1},\dots$ (see~\cite{DeligneGoncharov}).
This result implies that the embedding $\Gal_{\MT(\ZZ)}\hookrightarrow\GT(\QQ)$
is equivalent to an embedding of the form $\kk\oplus\LLie(s_3,s_5,\dots,s_{2n+1},\dots)\hookrightarrow\grt$
when we pass to the category of Lie algebras.
In this context, we can re-express the Deligne-Ihara conjecture as the conjecture that this embedding of Lie algebras
is an isomorphism.

In our explanations on Theorem~\ref{thm:e2-operads:drinfeld-associator-existence},
we also explained that the Knizhnik--Zamolodchikov associator represents the generating series of polyzeta values.
The polyzeta values satisfy certain equations, called the regularized double shuffle relations,
which can be expressed in terms of the Knizhnik--Zamolodchikov associator, and one conjectures that all relations
between polyzetas follows from the double shuffle relations and from the fact that the Knizhnik--Zamolodchikov associator defines
a group-like power series.
By a result of Furusho~\cite{FurushoDouble}, the pentagon condition for associators implies the regularized double shuffle relations.
This result implies that the Grothendieck--Teichm\"uller group embeds in a group
defined by solutions of regularized double shuffle relations with a degeneration condition,
and one conjectures again that this embedding is an isomorphism.

The theory of associators is also used by Alekseev--Torossian in the study of the solutions of the Kashiwara--Vergne conjecture,
a problem about the Campbell-Hausdorff formula motivated by questions
of harmonic analysis.
These authors notably proved in~\cite{AlekseevTorossian} that the set of Drinfeld's associators
embeds into the set of solutions of the Kashiwara--Vergne conjecture.
In particular, one can deduce the existence of such solutions from the existence of associators.
In addition, one can prove that the action of the Grothendieck--Teichm\"uller group
on Drinfeld's associators lifts to the set of solutions
of the Kashiwara--Vergne conjecture.
This action is still transitive so that we get that the Grothendieck--Teichm\"uller group
embeds into a group of automorphisms associated to this set of solutions.
The conjecture is that this group embedding is an isomorphism, yet again.

\section{The rational homotopy of $E_n$-operads and formality theorems}\label{sec:rational-homotopy}
We now tackle the study of the homotopy of $E_n$-operads for a general value of the dimension parameter $n\geq 1$.
We already explained that $E_1$-operads are the operads that are homotopy equivalent to the permutation operad in~\S\ref{background:en-operads}.
Therefore, we put this case apart and we focus on the case $n\geq 2$.
The main goal of this section is to explain the definition of rational models of $E_n$-operads
and a characterization of the class of $E_n$-operads
up to rational homotopy. In basic terms, this class consists of the operads $\ROp$
which are linked to the operad of little $n$-discs $\DOp_n$
by a zigzag of rational homotopy equivalences of operads.

Recall that a map of simply topological spaces is a rational homotopy equivalence $f: X\xrightarrow{\sim_{\QQ}} Y$
if this map induces a bijection on homotopy groups $f_*: \pi_*(X)\otimes_{\ZZ}\QQ\xrightarrow{\simeq}\pi_*(Y)\otimes_{\ZZ}\QQ$.
In what follows, we consider a generalization of this notion in the context for spaces which,
like the underlying spaces of the little $2$-disc operad,
are not necessarily simply connected.
Then we assume that a rational homotopy equivalence also induces an isomorphism
on the Malcev completion
of the fundamental group $f_*: \pi_1(X,x)\sphat_{\QQ}\xrightarrow{\simeq}\pi_1(Y,f(x))\sphat_{\QQ}$.
For spaces $X$ and $Y$, we also write $X\sim_{\QQ} Y$ when $X$ and $Y$
can be connected by a zigzag of rational homotopy equivalences
$X\xleftarrow{\sim_{\QQ}}\cdot\xrightarrow{\sim_{\QQ}}\cdot\ldots\cdot\xrightarrow{\sim_{\QQ}} Y$.

In the context of operads, we just consider morphisms $\phi: \POp\xrightarrow{\sim_{\QQ}}\QOp$
which define a rational homotopy equivalence of spaces arity-wise $\phi: \POp(r)\xrightarrow{\sim_{\QQ}}\QOp(r)$,
and we again write $\POp\sim_{\QQ}\QOp$ when our objects $\POp$ and $\QOp$
can be connected by a zigzag of such rational homotopy equivalences.
Thus, we mainly aim to determine the class of operads such that $\ROp\sim_{\QQ}\DOp_n$.

We develop a rational homotopy theory of operads to address this problem.
We rely on the Sullivan rational homotopy of spaces, which we briefly review in the next paragraph.
We explain the construction of an operadic extension of the Sullivan model afterwards.
We eventually check that the $n$-Poisson cooperad, the dual structure of the $n$-Poisson operad which occurs
in Theorem~\ref{thm:background:little-discs-homology},
defines a Sullivan model of the little $n$-discs operad $\DOp_n$,
and as such determines a model for the class of $E_n$-operads
up to rational homotopy. We need a cofibrant resolution of the $n$-Poisson cooperad
to perform computations with this model.
We will explain that such a cofibrant resolution is given by the Chevalley--Eilenberg cochain complex of a graded version
of the Drinfeld--Kohno Lie algebra operad of the previous section.
We consider another resolution in our construction, namely a cooperad of graphs,
and we also explain the definition
of this object.
We use the latter model in the next section, when we explain
a graph complex description of the rational homotopy type
of mapping spaces of $E_n$-operads.

In order to apply the methods of rational homotopy theory, we take $\kk = \QQ$ as a ground ring for our categories of modules from now on,
and we also consider the cohomology with coefficients in this field $\DGH^*(-) = \DGH^*(-,\QQ)$.
We similarly take $\DGH_*(-) = \DGH_*(-,\QQ)$ for the homology.

\subsubsection{Recollections on the Sullivan rational homotopy theory of spaces}\label{rational-homotopy:sullivan-models}
Let $\dg^*\ComCat$ be the category of commutative algebras in upper non-negatively graded dg-modules
(the category of commutative cochain dg-algebras for short).
Briefly recall that the Sullivan model for the rational homotopy of a space takes values in this category $\dg^*\ComCat$
and is obtained by applying the Sullivan functor of PL differential forms,
a version of the de Rham cochain complex which is defined over $\QQ$ (instead of $\RR$)
and on the category of simplicial sets $\sSet$ (instead of the category of smooth manifolds):
\begin{equation}\label{eqn:rational-homotopy:sullivan-models:PL-form-functor}
\DGOmega^*: \sSet^{op}\rightarrow\dg^*\ComCat.
\end{equation}
In the particular case of a simplex $\Delta^n = \{0\leq x_1\leq\cdots\leq x_n\leq 1\}$, we explicitly have:
\begin{equation}\label{eqn:rational-homotopy:sullivan-models:PL-form-algebras}
\DGOmega^*(\Delta^n) = \QQ[x_1,\dots,x_n,dx_1,\dots,dx_n],
\end{equation}
where $dx_1,\dots,dx_n$ represents the differential of the variables $x_1,\dots,x_n$
in this commutative cochain dg-algebra.

The Sullivan functor $\DGOmega^*: \sSet^{op}\rightarrow\dg^*\ComCat$ has a left adjoint
\begin{equation}\label{eqn:rational-homotopy:sullivan-models:realization-functor}
\DGG_{\bullet}: \dg^*\ComCat\rightarrow\sSet^{op},
\end{equation}
which is given by the formula $\DGG_{\bullet}(A) = \Mor_{\dg^*\ComCat}(A,\DGOmega^*(\Delta^{\bullet}))$,
for any commutative cochain dg-algebra $A\in\dg^*\ComCat$.

In addition, one can prove that this pair of adjoint functors define a Quillen adjunction.
Then we set:
\begin{equation}\label{eqn:rational-homotopy:sullivan-models:derived-realization}
\langle A\rangle := \text{derived functor of $\DGG_{\bullet}(A)$} = \Mor_{\dg^*\ComCat}(R_A,\DGOmega^*(\Delta^{\bullet})),
\end{equation}
where $R_A\xrightarrow{\sim}A$ is any cofibrant resolution of $A$ in $\dg^*\ComCat$.
If $X$ satisfies reasonable finiteness and nilpotence assumptions,
then
\begin{equation}\label{eqn:rational-homotopy:sullivan-models:rationalization}
X^{\QQ} := \langle\DGOmega^*(X)\rangle
\end{equation}
defines a rationalization of the space $X$ in the sense that we have the identities
\begin{equation}\label{eqn:rational-homotopy:sullivan-models:homotopy-group-rationalization}
\pi_*(X^{\QQ}) := \begin{cases} \pi_*(X)\otimes_{\ZZ}\QQ, & \text{for $*\geq 2$}, \\
\pi_1(X)\sphat_{\QQ}, & \text{for $*=1$},
\end{cases}
\end{equation}
where we again use the notation $(-)\sphat_{\QQ}$ for the Malcev completion functor on groups.
Besides, one can prove that the unit of the derived adjunction relation between the functors $\DGG_{\bullet}$ and $\DGOmega^*$
defines a map $\eta: X\rightarrow X^{\QQ}$
which corresponds to the usual rationalization map at the level of these homotopy groups.

\subsubsection{The category of Hopf dg-cooperads}\label{rational-homotopy:Hopf-dg-cooperads}
To extend the Sullivan model to operad, the idea is to consider cooperads in the category of commutative cochain dg-algebras,
where the cooperad is a structure which is dual to an operad
in the sense of the theory of categories.

In general, a cooperad in a symmetric category $\CCat$ consists of a collection of objects $\COp = \{\COp(r),r\in\NN\}$,
together with an action of the symmetric group $\Sigma_r$ on $\COp(r)$, for each $r\in\NN$,
and composition coproducts
\begin{equation}\label{eqn:rational-homotopy:Hopf-dg-cooperads:composition-coproducts}
\circ_i^*: \COp(k+l-1)\rightarrow\COp(k)\otimes\COp(l),
\end{equation}
defined for all $k,l\in\NN$, $i\in\{1,\dots,k\}$, and which satisfy equivariance, unit and coassociativity relations
dual to the equivariance, unit and coassociativity axioms of operads.
To handle difficulties we consider a subcategory of cooperads
such that $\COp(0) = \COp(1) = \unit$
where $\unit$ is the unit object of our base category, and we use the notation $\Op_{*1}^c$ for this category of cooperads.
This restriction enables us to simplify some constructions, because the composition coproducts are automatically conilpotent
when we put the component of arity zero apart and we assume $\COp(1) = \unit$.
In some cases, we consider a category of cooperads $\Op_{*N}^c$ such that we still have $\COp(0) = \unit$,
but where $\COp(1)$ does not necessarily reduces to the unit object.
More care is necessary in this case, and we actually assume an extra conilpotence condition for the composition coproducts
that involve the component of arity one (we refer to~\cite{FresseExtendedRatHomotopy} for the precise expression
of this conilpotence condition).

Let us mention that in~\cite[\S II.12]{FresseBook} we consider a category of $\Lambda$-cooperads
such that $\Lambda\Op_{\varnothing 1}^c\cong\Op_{*1}^c$
rather than this category of cooperads $\Op_{*1}^c$.
The idea is to model the coproducts with factors of arity zero $\COp(0) = \unit$
by using a diagram structure that extends the action of the symmetric groups.
This extra structure is used to perform certain categorical constructions,
but we can neglect these difficulties in this article.

We use the name `Hopf cochain dg-cooperad' for the category of cooperads in the category of commutative cochain dg-algebras $\CCat = \dg^*\ComCat$
and we also adopt the notation $\dg^*\Hopf\Op_{*1}^c$ for this category of cooperads.
We also consider a category of operads in simplicial sets satisfying $\POp(0) = \POp(1) = *$
in order to deal with the restrictions imposed by the definition
of our category of cooperads in our model. We use the notation $\sSet\Op_{*1}$ for this category of operads.
We have the following statement:

\begin{thm}[{B. Fresse~\cite[\S II.10, \S II.12]{FresseBook}}]\hspace*{2mm}\label{rational-homotopy:operads}
\begin{enumerate}
\item\label{item:rational-homotopy:operads:realization-functor}
The functor $\DGG_{\bullet}: \dg^*\ComCat\rightarrow\sSet^{op}$ induces a functor $\DGG_{\bullet}: \dg^*\Hopf\Op_{*1}^c\rightarrow\sSet\Op_{*1}^{op}$
from the category of Hopf cochain dg-cooperads $\dg^*\Hopf\Op_{*1}^c$ to the category of operads in simplicial sets $\sSet\Op_{*1}^{op}$.
For an object $\KOp\in\dg^*\Hopf\Op_{*1}^c$, we explicitly set $\DGG_{\bullet}(\KOp)(r) = \DGG_{\bullet}(\KOp(r))$
and we use that $\DGG_{\bullet}(-)$ is strongly symmetric monoidal
to equip the collection of these simplicial sets $\DGG_{\bullet}(\KOp) = \{\DGG_{\bullet}(\KOp(r)),r\in\NN\}$
with the structure of an operad.
\item\label{item:rational-homotopy:operads:PL-form-functor}
The functor $\DGG_{\bullet}: \dg^*\Hopf\Op_{*1}^c\rightarrow\sSet\Op_{*1}^{op}$
admits a right adjoint $\DGOmega_{\sharp}^*: \sSet\Op_{*1}^{op}\rightarrow\dg^*\Hopf\Op_{*1}^c$
and the pair of functors $(\DGG_{\bullet},\DGOmega_{\sharp}^*)$
defines a Quillen adjunction.
\item\label{item:rational-homotopy:operads:coherence}
For any cofibrant operad $\ROp\in\sSet\Op_{*1}$ such that $\DGH^*(\ROp(r)) = \DGH^*(\ROp(r),\QQ)$ forms a finite dimensional $\QQ$-module
in each arity $r\in\NN$ and in each degree $*\in\NN$,
we have a weak-equivalence $\DGOmega_{\sharp}^*(\ROp)(r)\xrightarrow{\sim}\DGOmega^*(\ROp(r))$
between the component of arity $r$ of the Hopf cochain dg-cooperad $\DGOmega_{\sharp}^*(\ROp)\in\dg^*\Hopf\Op_{*1}^c$
and the image of the space $\ROp(r)$ under the Sullivan functor $\DGOmega^*(-)$,
for any $r\in\NN$.
\end{enumerate}
\end{thm}

\begin{proof}[Explanations]
The first claim of this theorem follows from the observation that the functor $\DGG_{\bullet}(-)$
is strongly symmetric monoidal.
The functor $\DGOmega^*(-)$, on the other hand, is only weakly monoidal.
To be more precise, in the case of this functor, we have a K\"unneth morphism $\nabla: \DGOmega^*(X)\otimes\DGOmega^*(Y)\rightarrow\DGOmega^*(X\times Y)$
which is a quasi-isomorphism but not an isomorphism.
Hence, for an operad in simplicial set $\POp$, we only get that the composition product $\circ_i: \POp(k)\times\POp(l)\rightarrow\POp(k+l-1)$
induces a morphism which fits in a zigzag of morphisms
of commutative cochain dg-algebras
$\DGOmega^*(\POp(k+l-1))\xrightarrow{\circ_i^*}\DGOmega^*(\POp(k)\times\POp(l))\xleftarrow{\sim}\DGOmega^*(\POp(k))\otimes\DGOmega^*(\POp(l))$.
The idea is to use the lifting adjoint theorem to produce the functor of the second claim $\DGOmega_{\sharp}^*: \sSet\Op_{*1}^{op}\rightarrow\dg^*\Hopf\Op_{*1}^c$
and to fix this problem.
Then the crux lies in the verification of the third claim,
for which we refer to the cited reference.
\end{proof}

For an operad in simplicial sets $\POp\in\sSet\Op_{*1}$, we now set $\POp^{\QQ} = \langle\DGR\DGOmega_{\sharp}^*(\POp)\rangle$,
where we use the notation $\DGR\DGOmega_{\sharp}^*(-)$
for the right derived functor of the functor of the previous theorem $\DGOmega_{\sharp}^*: \sSet\Op_{*1}^{op}\rightarrow\dg^*\Hopf\Op_{*1}^c$,
and we again use the notation $\langle-\rangle$ for the left derived functor of the Sullivan realization
on operads $\DGG_{\bullet}: \dg^*\Hopf\Op_{*1}^c\rightarrow\sSet\Op_{*1}^{op}$.
The equivalence $\DGOmega_{\sharp}^*(\POp)(r)\sim\DGOmega^*(\POp(r))$
implies that we have the following result
at the level of this realization:

\begin{thm}[{B. Fresse~\cite[Theorem II.10.2.1 and Theorem II.12.2.1]{FresseBook}}]\label{thm:rational-homotopy:operad-rationalization}
For any operad $\POp\in\sSet\Op_{*1}$ such that $\DGH^*(\POp(r)) = \DGH^*(\POp(r),\QQ)$ forms a finite dimensional $\QQ$-module
in each arity $r\in\NN$ and in each degree $*\in\NN$,
we have $\POp^{\QQ}(r)\sim\POp(r)^{\QQ}$, where we consider the component of arity $r$ of the operad $\POp^{\QQ}$
on the left-hand side and the Sullivan rationalization of the space $\POp(r)$
on the right-hand side.\qed
\end{thm}

For the operad of little $n$-discs $\DOp_n$, we now set $\DGR\DGOmega^*_{\sharp}(\DOp_n) = \DGOmega_{\sharp}^*(\EOp_n)$,
where $\EOp_n$ is any cofibrant model of $E_n$-operad in simplicial sets such that $\EOp_n(0) = \EOp_n(1) = *$,
and we still write $\DOp_n^{\QQ} = \langle\DGR\DGOmega^*_{\sharp}(\DOp_n)\rangle$.
To apply the rational homotopy theory to the class of $E_n$-operads, we aim to determine the model of these objects $\DGR\DGOmega_{\sharp}^*(\DOp_n)$.

Recall that we have a homotopy equivalence $\DOp_n(r)\sim\FOp(\RR^n,r)$ between the underlying spaces of the operad of little $n$-discs $\DOp_n(r)$
and the configuration spaces of points in the Euclidean space $\FOp(\RR^n,r)$ (see~\S\ref{background:fulton-macpherson-operad}).
In a first step, we recall the following results about the cohomology algebras of these spaces:

\begin{thm}[{V.I. Arnold~\cite{Arnold}, F. Cohen~\cite{Cohen}}]\label{thm:rational-homotopy:arnold-presentation}
Let $n\geq 2$. For each $r\in\NN$, the graded commutative algebra $\DGH^*(\DOp_n(r))\simeq\DGH^*(\FOp(\RR^n,r))$
has a presentation of the form:
\begin{equation*}
\DGH^*(\FOp(\mathring{\DD}{}^n,r)) = \frac{\Sym(\omega_{ij},\{i,j\}\subset\{1,\dots,r\})}{(\omega_{ij}^2,\omega_{ij}\omega_{jk}+\omega_{jk}\omega_{ki}+\omega_{ki}\omega_{ij})}
\end{equation*}
where a generator $\omega_{ij}$ of upper degree $\deg^*(\omega_{ij}) = n-1$ is assigned to each pair $i\not=j$
with the symmetry formula $\omega_{ij} = (-1)^n\omega_{ji}$. (We use the notation $\Sym(-)$ for the symmetric algebra.)
\end{thm}

\begin{proof}[Explanations]
The result established by V.I. Arnold in~\cite{Arnold} concerns the case $n=2$ of this statement.
The already cited work of F. Cohen~\cite{Cohen} gives the general case $n\geq 2$.
We also refer to Sinha's survey~\cite{SinhaHomology} for a gentle introduction to the computation of this theorem.
Let us simply mention that the classes $\omega_{ij}\in\DGH^*(\FOp(\RR^n,r))$
represent the pullbacks of the fundamental class of the $n-1$-sphere $\omega\in\DGH^*(\Sphere^{n-1})$
under the maps $\pi_{ij}: \FOp(\RR^n,r)\rightarrow\Sphere^{n-1}$
such that $\pi_{ij}(a_1,\dots,a_r) = (a_j-a_i)/||a_j-a_i||$, for any element $\underline{a} = (a_1,\dots,a_r)\in\FOp(\RR^n,r)$.
In what follows, we refer to these cohomology classes $\omega_{ij}$ as the Arnold classes,
and we refer to the presentation of this theorem
as the Arnold presentation.
\end{proof}

Then we have the following result:

\begin{prop}\label{thm:rational-homotopy:little-discs-operad-cohomology}
The cohomology algebras $\DGH^*(\DOp_n(r))$, $r\in\NN$, form a Hopf cooperad in graded modules such that $\DGH^*(\DOp_n)\cong\PoisOp_n^c$,
where $\PoisOp_n^c$ denotes the cooperad dual to $\PoisOp_n$
in graded modules.
For a Poisson monomial $\pi(x_1,\dots,x_r)\in\PoisOp_n(r)$, we have the duality formula:
\begin{equation*}
\langle\omega_{ij},\pi(x_1,\dots,x_r)\rangle = \begin{cases} 1, & \text{if $\pi(x_1,\dots,x_r)=x_1\dots[x_i,x_j]\cdots\widehat{x_j}\cdots x_r$},\\
0, & \text{otherwise}.
\end{cases}
\end{equation*}
\end{prop}

\begin{proof}[Explanations and references]
The graded algebras $\DGH^*(\DOp_n(r))$ inherit an action of the symmetric groups
by functoriality of the cohomology, and the composition products of the operad $\DOp_n$ induce cooperad composition coproducts,
because, as the homology of the spaces $\DOp_n(r)$
form free modules of finite rank over the ground ring
in each arity,
we can invert the K\"unneth morphism
to form a composite $\DGH^*(\DOp(k+l-1))\xrightarrow{\circ_i^*}\DGH^*(\DOp(k)\times\DOp(l))\xleftarrow{\simeq}\DGH^*(\DOp(k))\otimes\DGH^*(\DOp(l))$
in the category of commutative graded algebras, for every $k,l\in\NN$ and $i\in\{1,\dots,r\}$.
Note that we have $\DOp_n(0) = *\Rightarrow\DGH^*(\DOp_n(0)) = \QQ$ and $\DOp_n(r)\sim *\Rightarrow\DGH^*(\DOp_n(0)) = \QQ$,
so that the collection $\DGH^*(\DOp_n) = \{\DGH^*(\DOp_n(r),r\in\NN\}$ satisfies our connectedness condition in the definition of a cooperad.

The dual graded modules of the components of the $n$-Poisson operad $\PoisOp_n^c(r) = \Hom_{\gr\Mod}(\PoisOp_n(r),\QQ)$
also inherit a cooperad structure by duality, because the objects $\PoisOp_n(r)$
form graded modules of finite rank.
Then we just use the duality between the cohomology and the homology $\DGH^*(\DOp(r)) = \Hom_{\gr\Mod}(\DGH_*(\DOp_n(r)),\QQ)$
to deduce the identity $\DGH^*(\DOp_n)\cong\PoisOp_n^c$ from the relation $\DGH_*(\DOp_n)\cong\PoisOp_n$
given in Theorem~\ref{thm:background:little-discs-homology}.

In the duality formula of the theorem, we use that the representation of the elements of the $n$-Poisson operad
in terms of Poisson polynomials that we explain in our survey of Theorem~\ref{thm:background:little-discs-homology}.
Recall that the representative of the Poisson bracket operation in the homology of the little $n$-discs operad $\lambda\in\DGH_*(\DOp_n(2))$
corresponds to the fundamental class of the $n-1$-sphere when we use the relation $\DOp_n(2)\sim\Sphere^{n-1}$.
Therefore, we trivially have the relation $\langle\omega_{12},\lambda(x_1,x_2)\rangle = 1$ in arity $2$.
To get the general formula of the theorem, we use that the cohomology classes $\omega_{ij}$
are identified with the pullbacks of the class $\omega_{12}$ under the map $\pi_{ij}: \FOp(\RR^n,r)\rightarrow\FOp(\RR^n,2)$
such that $\pi_{ij}(a_1,\dots,a_r) = (a_i,a_j)$
in the cohomology algebra $\DGH^*(\DOp_n(r))\simeq\DGH^*(\FOp(\RR^n,r))$
and that this operation corresponds to composites with a zero-ary operation $*\in\DOp_n(0)$
which represents an algebra unit $e\in\PoisOp_n(0)$ when we pass to the $n$-Poisson operad $\PoisOp_n$.
We refer to the paper~\cite{SinhaHomology} for a more thorough study of this duality relation between the $n$-Poisson polynomials
and the elements of the cohomology algebras $\DGH^*(\DOp_n(r))\simeq\DGH^*(\FOp(\RR^n,r))$
in the Arnold presentation.
\end{proof}

We can regard the object $\PoisOp_n^c\simeq\DGH^*(\DOp_n)$ as a Hopf cochain dg-cooperad equipped with a trivial differential.
We have to make explicit a cofibrant resolution of this object for the applications of our methods
of the rational homotopy theory of operads.
We use graded analogues of the Drinfeld--Kohno Lie algebra operad of the previous section
to give a first construction of such a resolution.

\subsubsection{The graded Drinfeld--Kohno Lie algebra operads and the associated Chevalley--Eilenberg cochain complexes}\label{rational-homotopy:graded-drinfeld-kohno-lie-algebra-operad}
The graded analogues of the Drinfeld--Kohno Lie algebra operad, which we define for every value of the parameter $n\geq 2$,
are denoted by~$\palg_n$. The ungraded Drinfeld--Kohno Lie algebra operad of~\S\ref{e2-operads:drinfeld-kohno-lie-algebra-operads}
corresponds to the case $n=2$.
Thus, we have $\palg = \palg_2$ with the notation of~\S\ref{e2-operads:drinfeld-kohno-lie-algebra-operads}.

To define the Lie algebras~$\palg_n(r)$, we use the same presentation
as in~\S\ref{e2-operads:drinfeld-kohno-lie-algebra-operads}(\ref{eqn:e2-operads:drinfeld-kohno-lie-algebra-operads:presentation}):
\begin{equation}\label{eqn:rational-homotopy:graded-drinfeld-kohno-lie-algebra-operad:presentation}
\palg_n(r) = \LLie(t_{ij},\{i,j\}\subset\{1,\dots,r\})/<[t_{ij},t_{kl}],[t_{ij},t_{ik}+t_{jk}]>,
\end{equation}
but we now take $\deg(t_{ij}) = n-2$ and we assume the graded symmetry relation $t_{ji} = (-1)^n t_{ij}$,
for every pair $\{i,j\}\subset\{1,\dots,r\}$.
Then we take the same construction as in~\S\ref{e2-operads:drinfeld-kohno-lie-algebra-operads}
to provide these Lie algebras with an action of the symmetric groups
and with additive composition products $\circ_i: \palg_n(k)\oplus\palg_n(l)\rightarrow\palg_n(k+l-1)$,
so that the collection $\palg_n = \{\palg_n(r),r\in\NN\}$ inherits the structure
of an operad in the category of graded Lie algebras.

Note that the graded Lie algebras $\palg_n(r)$ still inherit a weight grading from the free Lie algebra,
and hence, form weight graded objects in the category of graded modules.
Besides, we can form a completed version of the operads $\hat{\palg}_n$, as in the case $n=2$
in~\S\ref{e2-operads:drinfeld-kohno-lie-algebra-operads},
but for $n\geq 3$, we trivially have $\hat{\palg}_n = \palg_n$
because the components of homogeneous weight $m\geq 1$ of the Lie algebras $\palg_n(r)$
are concentrated in a degree $* = m(n-2)$ and we have $m(n-2)\rightarrow\infty$
when $n\geq 3$.

We consider the Chevalley--Eilenberg cochain complexes $\DGC_{CE}^*(\hat{\galg})$ associated to the complete Lie algebras $\hat{\galg} = \hat{\palg}_n(r)$.
The cofibrant objects of the category of commutative cochain dg-algebras are retracts of dg-algebras of the form $R = (\Sym(V),\partial)$,
where $\Sym(V)$ is the symmetric algebra on an upper graded dg-module $V$
equipped with a filtration $\DGF_1 V\subset\DGF_2 V\subset\cdots\supset\DGF_m V\subset\cdots\subset V$
and where we have a differential $\partial$ such that $\partial(\DGF_m V)\subset\Sym(\DGF_{m-1} V)$.
The Chevalley--Eilenberg cochain complex is precisely defined by an expression of this form $\DGC_{CE}^*(\galg) = (\Sym(\QQ[-1]\otimes\hat{\galg}^{\vee}),\partial)$,
where $\QQ[-1] = \QQ\ecell$ denotes the graded module generated by a single element $\ecell$ in lower degree $-1$ (equivalently, in upper degree one)
and $\hat{\galg}^{\vee}$ denotes the (continuous) dual of the completed Lie algebra $\hat{\galg} = \hat{\palg}_n(r)$.
The differential $\partial$ is induced by the dual map of the Lie bracket $[-,-]$
on $\hat{\galg}$.

The commutative cochain dg-algebras $\DGC_{CE}^*(\hat{\palg}_n(r))$
inherit an action of the symmetric groups by functoriality of the Chevalley--Eilenberg cochain complex,
as well as composition coproducts $\circ_i^*: \DGC_{CE}^*(\hat{\palg}(k+l-1))\rightarrow\DGC_{CE}^*(\hat{\palg}(k))\otimes\DGC_{CE}^*(\hat{\palg}(l))$,
which are given by the composites of the morphisms $\circ_i^*: \DGC_{CE}^*(\hat{\palg}(k+l-1))\rightarrow\DGC_{CE}^*(\hat{\palg}(k)\oplus\hat{\palg}(l))$
induced by the composition products of the operad $\palg_n$
with the K\"unneth isomorphisms $\DGC_{CE}^*(\hat{\palg}(k)\oplus\hat{\palg}(l))\simeq\DGC_{CE}^*(\hat{\palg}(k))\otimes\DGC_{CE}^*(\hat{\palg}(l))$.
Hence, we get that the collection $\DGC_{CE}^*(\hat{\palg}_n) = \{\DGC_{CE}^*(\hat{\palg}_n(r)),r\in\NN\}$
inherits the structure of Hopf cochain dg-cooperad. In addition, one can prove that this Hopf cochain dg-cooperad
is cofibrant (see~\cite[Theorem II.14.1.7]{FresseBook}).

Then we have the following results:

\begin{thm}[{T. Kohno~\cite{Kohno}}]\label{thm:rational-homotopy:drinfeld-kohno-cohomology-algebra-qiso}
We have a quasi-isomorphism of commutative cochain dg-algebras
\begin{equation*}
\DGC_{CE}^*(\hat{\palg}_n(r))\xrightarrow{\sim}\DGH^*(\FOp_n(\mathring{\DD}{}^n,r))
\end{equation*}
such that $t_{ij}^{\vee}\mapsto\omega_{ij}$ for each pair $\{i,j\}\subset\{1,\dots,r\}$
and $p^{\vee}\mapsto 0$ when $p^{\vee}$ is the dual basis element of a homogeneous Lie polynomial $p\in\palg_n(r)_m$
of weight $m>1$.
\end{thm}

\begin{proof}[Explanations]
The cited reference gives the case $n=2$ of this statement. The general result can be deduced from the observation
that the cohomology algebra $\DGH^*(\FOp_n(\mathring{\DD}{}^n,r)$
forms a Koszul algebra with the enveloping algebra of the Lie algebra $\palg_n(r)$
as dual associative algebra. We refer to~\cite[Theorem II.14.1.14]{FresseBook} for further explanations on this approach.
\end{proof}

\begin{prop}\label{prop:rational-homotopy:drinfeld-kohno-cohomology-cooperad-qiso}
The quasi-isomorphisms of Theorem~\ref{thm:rational-homotopy:drinfeld-kohno-cohomology-algebra-qiso}
define a weak-equivalence of Hopf cochain dg-cooperads
\begin{equation*}
\DGC_{CE}^*(\hat{\palg}_n)\xrightarrow{\sim}\DGH^*(\DOp_n) = \PoisOp_n^c,
\end{equation*}
where we regard the cohomology of the little $n$-discs operad $\DGH^*(\DOp_n)$ as a Hopf cochain dg-cooperad
equipped with a trivial differential.
\end{prop}

\begin{proof}[Explanations]
The result of this proposition follows from the straightforward verification that the quasi-isomorphisms
of Theorem~\ref{thm:rational-homotopy:drinfeld-kohno-cohomology-algebra-qiso}
preserve the cooperad structures attached to our objects.
\end{proof}

We deduce from the above proposition that the object $\DGC_{CE}^*(\hat{\palg}_n)$
defines a cofibrant resolution of the object $\PoisOp_n^c = \DGH^*(\DOp_n)$
in the category of Hopf cochain dg-cooperads.
In our constructions, we actually consider a second cofibrant resolution, which is given by a Hopf cochain dg-cooperad of graphs $\GraphOp_n^c$.

\subsubsection{The graph cooperad}\label{rational-homotopy:graph-cooperad}
The Hopf cochain dg-cooperad $\GraphOp_n^c$ precisely consists of graphs $\gamma\in\GraphOp_n^c(r)$
with undistinguishable internal vertices $\bullet$ and external vertices indexed by $1,\dots,r$,
as in the following pictures:
\begin{equation}\label{eqn:rational-homotopy:graph-cooperad:picture}
\vcenter{\xymatrix@!0@C=1.5em@R=1.5em@M=0pt{ & *{\bullet}\ar@{-}[dl]\ar@{-}[d]\ar@{-}[dr] & \\
*+<6pt>[o][F]{1} & *+<6pt>[o][F]{2} & *+<6pt>[o][F]{3} }}.
\end{equation}
The degree of such a graph is determined by assuming that each internal vertex $\bullet$
contributes to the degree by $\deg(\bullet) = n$
and that each edge contributes to the degree by $\deg(-) = 1-n$ (in the lower grading convention).
Thus, we have $\deg(\gamma) = (1-n)v + n e$ (in the lower grading convention again), where $v$ denotes the number of internal vertices
and $e$ denotes the number of edges in the graph $\gamma\in\GraphOp_n^c(r)$.
In fact, we can regard our graphs as tensor products of symbolic elements given by the internal vertices and by the edges of our objects.
In particular, we assume that graphs equipped with odd symmetries vanish in~$\GraphOp_n^c(r)$. We also assume that each edge is oriented
and that a reversal of orientation is equivalent to the multiplication by a sign $(-1)^n$ in~$\GraphOp_n^c(r)$.
For our purpose, we allow graphs with double edges, but not loops (edges with the same origin and endpoint),
and we assume that each internal vertex is at least trivalent though the latter conditions are not essential.
Besides, we assume that each connected component of our graph contains at least one external vertex.

The differential of graphs is defined by contracting edges in order to merge internal vertices
together or in order to merge internal vertices with external vertices, as shown
schematically in the following picture:
\begin{equation}\label{eqn:rational-homotopy:graph-cooperad:differential}
\delta\;\vcenter{\xymatrix@!0@C=1em@R=1em@M=0pt{ \ar@{-}[drr] & \ar@{-}[dr] & *{\cdots} & \ar@{-}[dl] & \ar@{-}[dll] \\
&& *{\bullet}\ar@{-}[d] && \\
&& *{\bullet} && \\
\ar@{-}[urr] & \ar@{-}[ur] & *{\cdots} & \ar@{-}[ul] & \ar@{-}[ull] }}
\;=\;\vcenter{\xymatrix@!0@C=1em@R=1em@M=0pt{ \ar@{-}[drr] & \ar@{-}[dr] & *{\cdots} & \ar@{-}[dl] & \ar@{-}[dll] \\
&& *{\bullet} && \\
\ar@{-}[urr] & \ar@{-}[ur] & *{\cdots} & \ar@{-}[ul] & \ar@{-}[ull] }}
\qquad\text{and}
\qquad\delta\;\vcenter{\xymatrix@!0@C=1em@R=1.3em@M=0pt{ & \ar@{-}[dr] & *{\cdots} & \ar@{-}[dl] & \\
\ar@{-}[drr] & *{\cdots} & *{\bullet}\ar@{-}[d] & *{\cdots} & \ar@{-}[dll] \\
&& *+<6pt>[o][F]{i} && }}
\;=\;\vcenter{\xymatrix@!0@C=1em@R=1.3em@M=0pt{ \ar@{-}[drr] & \ar@{-}[dr] & *{\cdots} & \ar@{-}[dl] & \ar@{-}[dll] \\
&& *+<6pt>[o][F]{i} && }}.
\end{equation}
For instance, we have the formula:
\begin{equation}\label{eqn:rational-homotopy:graph-cooperad:differential-example}
\delta\,\vcenter{\xymatrix@!0@C=1.5em@R=2em@M=0pt{ & *{\bullet}\ar@{-}[dl]\ar@{-}[d]\ar@{-}[dr] & \\
*+<6pt>[o][F]{1} & *+<6pt>[o][F]{2} & *+<6pt>[o][F]{3} }}
= \vcenter{\xymatrix@!0@C=1.5em@R=2em@M=0pt{ && \\
*+<6pt>[o][F]{1}\ar@/^1em/@{-}[r]\ar@/^1.66em/@{-}[rr] & *+<6pt>[o][F]{2} & *+<6pt>[o][F]{3} }}
\pm\vcenter{\xymatrix@!0@C=1.5em@R=2em@M=0pt{ && \\
*+<6pt>[o][F]{1} & *+<6pt>[o][F]{2}\ar@/^1em/@{-}[r]\ar@/_1em/@{-}[l] & *+<6pt>[o][F]{3} }}
\pm\vcenter{\xymatrix@!0@C=1.5em@R=2em@M=0pt{ && \\
*+<6pt>[o][F]{1} & *+<6pt>[o][F]{2} & *+<6pt>[o][F]{3}\ar@/_1em/@{-}[l]\ar@/_1.66em/@{-}[ll] }}
\end{equation}
in $\GraphOp_n^c(3)$.
The product is given by the amalgamated sum of graphs along external vertices. For instance, we have the formula:
\begin{equation}\label{eqn:rational-homotopy:graph-cooperad:product}
\vcenter{\xymatrix@!0@C=1.5em@R=1.5em@M=0pt{ && \\
*+<6pt>[o][F]{1}\ar@/^1em/@{-}[r]\ar@/^1.66em/@{-}[rr] & *+<6pt>[o][F]{2} & *+<6pt>[o][F]{3} }}
= \vcenter{\xymatrix@!0@C=1.5em@R=1.5em@M=0pt{ && \\
*+<6pt>[o][F]{1}\ar@/^1.66em/@{-}[rr] & *+<6pt>[o][F]{2} & *+<6pt>[o][F]{3} }}
\cdot\vcenter{\xymatrix@!0@C=1.5em@R=1.5em@M=0pt{ && \\
*+<6pt>[o][F]{1}\ar@/^1em/@{-}[r] & *+<6pt>[o][F]{2} & *+<6pt>[o][F]{3} }}.
\end{equation}

The cooperad coproduct $\circ_i^*: \GraphOp_n^c(k+l-1)\rightarrow\GraphOp_n^c(k)\otimes\GraphOp_n^c(l)$,
where we fix $k,l\in\NN$, $i\in\{1,\dots,k\}$,
has the form $\circ_i^*(\gamma) = \sum_{\alpha\subset\gamma} \gamma/\alpha\otimes\alpha$,
where the sum runs over all the subgraphs $\alpha\subset\gamma$ that contain the external vertices indexed by $i,\dots,i+l-1$,
and $\gamma/\alpha$ denotes the graph obtained by collapsing this subgraph
to a single external vertex (which we index by $i$ in the result of the operation,
while we shift the index of the vertices such that $j>i$ by $j\mapsto j-l+1$).
Let us observe that we have $\GraphOp_n^c(1)\not=\QQ$
in general, so that our object $\GraphOp_n^c$ belongs to the extended category of Hopf cochain dg-cooperads $\dg^*\Hopf\Op_{*N}^c$
but not to the category of connected Hopf cochain dg-cooperads $\dg^*\Hopf\Op_{*1}^c$.

We easily see that the commutative cochain dg-algebras of graphs defined in this paragraph $\GraphOp_n^c(r)$
have a structure
of the form $\GraphOp_n^c(r) = (\Sym(\QQ[-1]\otimes\ICG_n^c(r)),\partial)$ (like the Chevalley--Eilenberg cochain dg-algebras
of the previous paragraph),
where $\ICG_n^c(r)$ is a complex of graphs which are connected when we remove the external vertices
inside $\GraphOp_n^c(r)$. (In what follows, we refer to such graphs as internally connected graphs.)
We just perform an extra degree shift in the definition of this complex of internally connected graphs in order to get
a $\QQ[-1]$ factor on the generating dg-module of our symmetric algebra (as in the definition of the Chevalley--Eilenberg cochain complex
of a Lie dg-algebra).
We can actually use this expression to identify the object $\ICG_n^c(r)$ with the dual of an $L_{\infty}$-algebra (a strongly homotopy Lie algebra).
We can use this symmetric algebra structure $\GraphOp_n^c(r) = (\Sym(\QQ[-1]\otimes\ICG_n^c(r)),\partial)$
to prove that $\GraphOp_n^c$ forms a cofibrant object in the category $\dg^*\Hopf\Op_{*N}^c$,
and we also have the following proposition:

\begin{prop}[{M. Kontsevich~\cite{KontsevichMotives}}]\label{prop:rational-homotopy:graph-cohomology-cooperad-qiso}
We have a quasi-isomorphism of Hopf dg-cooperads
\begin{equation*}
\GraphOp_n^c\xrightarrow{\sim}\DGH^*(\DOp_n) = \PoisOp_n^c
\end{equation*}
such that $\xymatrix@!0@C=1.5em@R=1.5em@M=0pt{ *+<6pt>[o][F]{1}\ar@{-}@/^1em/[r] & *+<6pt>[o][F]{2} }\mapsto\omega_{12}$.
\end{prop}

\begin{proof}
We define a morphism of commutative cochain dg-algebras $\GraphOp_n^c(r)\xrightarrow{\sim}\DGH^*(\FOp_n(r))$
by taking $\gamma\mapsto 0$ when $\gamma$ is an internally connected graph with a non-empty set of internal vertices
and by using an obvious extension of the assignment of the proposition
when we consider graphs with no internal vertices.
We easily check that this map preserves differentials, and hence, gives a well-defined morphism
of commutative cochain dg-algebras -- observe simply that the differential
of~\S\ref{rational-homotopy:graph-cooperad}(\ref{eqn:rational-homotopy:graph-cooperad:differential-example})
is carried to the Arnold relation in $\DGH^*(\DOp_n(3))$.
We refer to the cited reference~\cite{KontsevichMotives} and to~\cite{LambrechtsVolic} for a proof that this map
defines a quasi-isomorphism in each arity $r\in\NN$. Then we can easily check that these quasi-isomorphisms
preserve cooperad structures, and hence, define a quasi-isomorphism
in the category of Hopf cochain dg-cooperads.
\end{proof}

Recall that we set $\DGR\DGOmega^*_{\sharp}(\DOp_n) = \DGOmega_{\sharp}^*(\EOp_n)$ for the topological operad of little $n$-discs $\DOp_n$,
where $\EOp_n$ is any cofibrant model of $E_n$-operad in simplicial sets
such that $\EOp_n(0) = \EOp_n(1) = *$. We have the following result:

\begin{thm}[{B. Fresse and T. Willwacher~\cite[Theorem A']{FresseWillwacher}}]\label{thm:rational-homotopy:en-operads-formality}
We have the relation
\begin{equation*}
\PoisOp_n^c\sim\DGR\DGOmega^*_{\sharp}(\DOp_n),
\end{equation*}
in the category of Hopf cochain dg-cooperads.
\end{thm}

\begin{proof}[Explanations]
This theorem asserts that the operad of little $n$-discs is formal in the sense of our operadic counterpart
of the Sullivan rational homotopy theory of spaces. We refer to the cited reference~\cite[Theorem A']{FresseWillwacher}
for a proof of an intrinsic formality theorem which implies this operadic formality result
in the case $n\geq 3$. (In the next statement, we will explain that the case $n=2$ of this theorem
follows from the existence of Drinfeld's associators.)

Let us mention however that the result of this theorem can be deduced from Kontsevich's proof of the formality of $E_n$-operads
when we pass to real coefficients (see~\cite{KontsevichMotives}).
Indeed, the construction of Kontsevich can be used to define a collection
of quasi-isomorphisms $\GraphOp_n^c(r)\xrightarrow{\sim}\DGOmega_{sa}^*(\FMOp_n(r))$,
where we consider the Fulton--MacPherson operad $\FMOp_n$ of~\S\ref{background:fulton-macpherson-operad}
and $\DGOmega_{sa}^*(\FMOp_n(r))$
denotes a cochain dg-algebra of semi-algebraic forms
on the space $\FMOp_n(r)$ (see~\cite{HardtLambrechtsVolic,LambrechtsVolic}). Then one can observe that these morphisms can be associated
to a strict morphism of Hopf cochain dg-cooperads $\GraphOp_n^c\xrightarrow{\sim}\DGOmega_{\sharp}^*(\FMOp_n)$
(by using a general coherence statement of~\cite[Proposition II.12.1.3]{FresseBook}).

The approach of the cited reference~\cite{FresseWillwacher} does not use this constructions
and gives a formality quasi-isomorphism which is defined over the rationals
by using obstruction theory methods. The claim of this reference is that the $E_n$-operads are intrinsically rationally formal for $n\geq 3$
in the sense that every Hopf cochain dg-cooperad $\KOp_n$ which satisfies $\DGH_*(\KOp_n)\simeq\PoisOp_n^c$
and is equipped with an extra-involution operad $J: \KOp_n\rightarrow\KOp_n$
such that $J(\lambda) = -\lambda$ in the case $4|n$
satisfies $\KOp_n\sim\PoisOp_n^c$. We apply this claim to the Hopf cochain dg-cooperad $\KOp_n = \DGR\DGOmega^*_{\sharp}(\DOp_n)$
to get the statement of the theorem.

Let us also mention that we have an extension of this formality result for the morphisms $\DOp_m\rightarrow\DOp_n$
which link the operads of little discs when $n-m\geq 2$ (see~\cite[Theorem C]{FresseWillwacher}).
We go back to this subject in the next section.
\end{proof}

Recall that we set $\DOp_n^{\QQ} = \langle\DGR\DGOmega^*_{\sharp}(\DOp_n)\rangle$
to define a model for the rationalization of the little $n$-discs operad
in topological spaces.
The result of the previous theorem has the following corollary:

\begin{cor}\label{cor:rational-homotopy:graph-cooperad:formality-realization}
We have $\DOp_n^{\QQ} = \langle\PoisOp_n^c\rangle$, for any $n\geq 2$, where we consider the image of the dual cooperad of the $n$-Poisson operad $\PoisOp_n^c$
under the operadic upgrading of the Sullivan realization functor $\langle-\rangle$.\qed
\end{cor}

\begin{proof}[Explanations]
We just use the implication $\PoisOp_n^c\sim\DGR\DGOmega^*_{\sharp}(\DOp_n)\Rightarrow\langle\PoisOp_n^c\rangle\sim\langle\DGR\DGOmega^*_{\sharp}(\DOp_n)\rangle$
to get the result of this corollary.
This result, together with the observations of Proposition~\ref{prop:rational-homotopy:drinfeld-kohno-cohomology-cooperad-qiso}
and Proposition~\ref{prop:rational-homotopy:graph-cohomology-cooperad-qiso},
implies that we can take either $\langle\PoisOp_n^c\rangle = \DGG_{\bullet}(\DGC_{CE}^*(\hat{\palg}_n))$
or $\langle\PoisOp_n^c\rangle = \DGG_{\bullet}(\GraphOp_n^c)$
to get a model of the rationalization $\DOp_n^{\QQ}$.

We can observe that we have an identity $\DGG_{\bullet}(\DGC_{CE}^*(\hat{\palg}_n(r))) = \DGMC_{\bullet}(\hat{\palg}_n(r))$, for each $r\in\NN$,
where we consider a Maurer--Cartan space associated to the complete Lie algebra $\hat{\palg}_n(r)$ (we review the definition
of this construction in the next sections). Hence, the results of this section
gives a simple algebraic model of the rational homotopy type of $E_n$-operads.
Let us mention that, in the case $n=2$, we have an identity $\DGC_{CE}^*(\hat{\palg})\sim\DGB(\CD\sphat_{\QQ})$,
where we consider the chord diagram operad of~\S\ref{e2-operads:chord-diagram-operad}
(recall also that we use the notation $\hat{\palg} = \hat{\palg}_2$
for the ungraded Drinfeld--Kohno Lie algebra operad which occurs in this case $n=2$).
Thus, since we have on the other hand $\DOp_2^{\QQ} = \DGB(\PaB\sphat_{\QQ})$ (see~\S\ref{sec:e2-operads}),
we can deduce the existence of a weak-equivalence $\DOp_2^{\QQ}\sim\langle\PoisOp_n^c\rangle$
from the operadic interpretation of Drinfeld's associators
given in~\S\ref{e2-operads:chord-diagram-operad} (see~\cite[\S II.14.2]{FresseBook}).
\end{proof}

We now consider a counterpart in the category of dg-modules $\dg^*\Mod$ of the formality result
of Theorem~\ref{thm:rational-homotopy:en-operads-formality}.
We use the notation $\DGC_*(-)$ for both the singular complex functor from the category of topological spaces to the category of dg-modules
and for the standard normalized chain complex functor on simplicial sets.
These functors are lax symmetric monoidal and therefore carry operads in topological spaces (respectively, in simplicial sets)
to operads in dg-modules.
Furthermore, in the case of a cofibrant operad in simplicial sets $\ROp$, we have the duality relation $\DGOmega_{\sharp}^*(\ROp)^{\vee}\sim\DGC_*(\ROp)$
in the category of dg-operads $\dg\Op_*$ when we consider the dual in dg-modules
of the Hopf cochain dg-cooperad $\DGOmega_{\sharp}^*(\ROp)$
of Theorem~\ref{thm:rational-homotopy:en-operads-formality}.
Therefore, the result of Theorem~\ref{thm:rational-homotopy:en-operads-formality}
implies the following statement,
which was also obtained by the authors cited in this statement
by other method:

\begin{thm}[{D. Tamarkin~\cite{TamarkinFormality}, M. Kontsevich~\cite{KontsevichMotives}}]\label{thm:rational-homotopy:graph-cooperad:chain-formality}
We have the relation
\begin{equation*}
\PoisOp_n\sim\DGC_*(\DOp_n),
\end{equation*}
in the category of dg-operads.
\end{thm}

\begin{proof}
Let us simply mention that Tamarkin's proof of this result, which works in the case $n=2$, relies on the correspondence between formality equivalences
and associators, whereas Kontsevich's proof, which works for every $n\geq 2$ but requires to pass to real coefficients,
relies on the definition of semi-algebraic forms associated
to graphs (as we explain in our survey of Theorem~\ref{thm:rational-homotopy:en-operads-formality}).
\end{proof}

This result is exactly the formality theorem mentioned in the introduction of this paper for the class of $E_n$-operads in dg-modules.

\section{The rational homotopy of mapping spaces on the operads of little discs}\label{sec:mapping-spaces-homotopy}
We now tackle the main objective of this paper, namely the computation of the homotopy of the mapping spaces $\Map_{\Top\Op}^h(\DOp_m,\DOp_n^{\QQ})$
and of the homotopy automorphism spaces $\Aut_{\Op}^h(\DOp_n^{\QQ})$, for all $n\geq 2$.
Thus, we aim to generalize the computation carried out in~\S\ref{sec:e2-operads}
in the case of the automorphism space $\Aut_{\Op}^h(\DOp_2^{\QQ})$.
In the case of the mapping spaces $\Map_{\Top\Op}^h(\DOp_m,\DOp_n^{\QQ})$,
we are also going to check that we have the relation $\Map_{\Top\Op}^h(\DOp_m,\DOp_n^{\QQ})\sim\Map_{\Top\Op}^h(\DOp_m,\DOp_n)^{\QQ}$
when $n-m\geq 3$, so that the results explained in this section
gives a full computation of the rational homotopy of the mapping spaces $\Map_{\Top\Op}^h(\DOp_m,\DOp_n)$
that occur in the operadic description of the embedding spaces $\overline{\Emb}_c(\RR^m,\RR^n)$
in~\S\ref{sec:background}.

To carry out these computations, we use the graph complex model $\GraphOp_n^c$ of the rational homotopy of the operad $\DOp_n$.
Hence, we naturally obtain, as a main outcome, a graph complex description
of the homotopy the spaces $\Map_{\Top\Op}^h(\DOp_m,\DOp_n^{\QQ})$ and $\Aut_{\Op}^h(\DOp_n^{\QQ})$.
In the case of the mapping spaces $\Map_{\Top\Op}^h(\DOp_m,\DOp_n^{\QQ})$,
we express the result as the Maurer--Cartan space $\DGMC_{\bullet}(\HGC_{m n})$
associated to a Lie dg-algebra of hairy graphs $\HGC_{m n}$. We explain the definition of this object first
and we explain our computation afterwards.

In the case of the automorphism spaces $\Aut_{\Op}^h(\DOp_n^{\QQ})$, we get that our object is homotopy equivalent to a cartesian product
of Eilenberg--MacLane spaces (like any $H$-group
in rational homotopy theory). Thus, we can focus on the computation of the homotopy groups in this case.
We give a description of these groups in terms of the homology of a non-hairy graph complex $\GC_n$
of which we also explain the definition beforehand.
This graph complex $\GC_n$ is a graded version of a complex introduced by Kontsevich in~\cite{KontsevichSymplectic},
and therefore, this complex is usually called the Kontsevich graph complex in the literature.

\subsubsection{The hairy graph complex}\label{mapping-spaces-homotopy:hairy-graph-complex}
The hairy graph complex $\HGC_{m n}$ explicitly consists of formal series of connected graphs with internal vertices $\bullet$,
internal edges $\bullet\!\!-\!\!\bullet$, which link internal vertices together,
and external edges (the hairs) $\bullet\!\!-$, which are open at one extremity,
as in the following examples:
\begin{equation}\label{eqn:mapping-spaces-homotopy:hairy-graph-complex:picture}
\begin{tikzpicture}[scale=.5]
\draw (0,0) circle (1);
\draw (-180:1) node[int]{} -- +(-1.2,0);
\end{tikzpicture},
\quad\begin{tikzpicture}[scale=.6]
\node[int] (v) at (0,0){};
\draw (v) -- +(90:1) (v) -- ++(210:1) (v) -- ++(-30:1);
\end{tikzpicture},
\quad\begin{tikzpicture}[scale=.5]
\node[int] (v1) at (-1,0){};\node[int] (v2) at (0,1){};\node[int] (v3) at (1,0){};\node[int] (v4) at (0,-1){};
\draw (v1)  edge (v2) edge (v4) -- +(-1.3,0) (v2) edge (v4) (v3) edge (v2) edge (v4) -- +(1.3,0);
\end{tikzpicture},
\quad\begin{tikzpicture}[scale=.6]
\node[int] (v1) at (0,0){};\node[int] (v2) at (180:1){};\node[int] (v3) at (60:1){};\node[int] (v4) at (-60:1){};
\draw (v1) edge (v2) edge (v3) edge (v4) (v2)edge (v3) edge (v4)  -- +(180:1.3) (v3)edge (v4);
\end{tikzpicture}.
\end{equation}
This complex $\HGC_{m n}$ is equipped with a lower grading. The degree of a graph $\gamma\in\HGC_{m n}$
is determined by assuming that each vertex contributes by~$\deg(\bullet) = n$,
that each internal edge contributes by $\deg(\bullet\!\!-\!\!\bullet) = 1-n$,
that each hair contributes by $\deg(\bullet\!\!-) = m-n+1$,
and by adding a global degree shift by $-m$.
Thus, we have $\deg(\gamma) = n v + (1-n) e + (m-n+1) h - m$, where $v$ denotes the number of internal vertices,
the letter $e$ denotes the number of internal edges
and $h$ denotes the number of hairs of the graph $\gamma\in\HGC_{m n}$.
The differential of the hairy graph complex is defined by the blow-up of internal vertices:
\begin{equation}\label{eqn:mapping-spaces-homotopy:hairy-graph-complex:differential}
\delta\;\vcenter{\xymatrix@!0@C=1em@R=1em@M=0pt{ \ar@{-}[drr] & \ar@{-}[dr] & *{\cdots} & \ar@{-}[dl] & \ar@{-}[dll] \\
&& *{\bullet} && \\
\ar@{-}[urr] & \ar@{-}[ur] & *{\cdots} & \ar@{-}[ul] & \ar@{-}[ull] }}
\;=\;\vcenter{\xymatrix@!0@C=1em@R=1em@M=0pt{ \ar@{-}[drr] & \ar@{-}[dr] & *{\cdots} & \ar@{-}[dl] & \ar@{-}[dll] \\
&& *{\bullet}\ar@{-}[d] && \\
&& *{\bullet} && \\
\ar@{-}[urr] & \ar@{-}[ur] & *{\cdots} & \ar@{-}[ul] & \ar@{-}[ull] }}.
\end{equation}

We equip the hairy graph complex with the Lie bracket such that:
\begin{equation}\label{eqn:mapping-spaces-homotopy:hairy-graph-complex:Lie-bracket}
\left[\vcenter{\xymatrix@!0@C=1em@R=1.5em@M=0pt{ && *+<5pt>{\gamma_1}\ar@{-}[dl]\ar@{-}[dll]\ar@{-}[dr]\ar@{-}[drr] && \\
&& *{\cdots} && }},\vcenter{\xymatrix@!0@C=1em@R=1.5em@M=0pt{ && *+<5pt>{\gamma_2}\ar@{-}[dl]\ar@{-}[dll]\ar@{-}[dr]\ar@{-}[drr] && \\
&& *{\cdots} && }}\right]
\;=\;\sum\pm\vcenter{\xymatrix@!0@C=1em@R=1.5em@M=0pt{ && *+<5pt>{\gamma_1}\ar@{-}[dl]\ar@{-}[dll]\ar@{-}[dr]\ar@{-}[drr] &&&& \\
&& *{\cdots} & *+<5pt>{\gamma_2}\ar@{-}[dl]\ar@{-}[dll]\ar@{-}[dr]\ar@{-}[drr] &&& \\
&&& *{\cdots} &&& }}
- \sum\pm\vcenter{\xymatrix@!0@C=1em@R=1.5em@M=0pt{ && *+<5pt>{\gamma_2}\ar@{-}[dl]\ar@{-}[dll]\ar@{-}[dr]\ar@{-}[drr] &&&& \\
&& *{\cdots} & *+<5pt>{\gamma_1}\ar@{-}[dl]\ar@{-}[dll]\ar@{-}[dr]\ar@{-}[drr] & \\
&&& *{\cdots} &&& }},
\end{equation}
where the first sum runs over the re-connections of a hair of the graph $\gamma_1$ to a vertex of the graph $\gamma_2$,
and similarly in the second sum, with the role of the graphs $\gamma_1$ and $\gamma_2$
exchanged. In the case $m=1$, we have to consider a deformation of this Lie dg-algebra structure
which we call the Shoikhet $L_{\infty}$-structure. We just refer to~\cite{WillwacherShoikhet}
for the explicit definition of this structure. (Recall simply that an $L_{\infty}$-algebra
denotes the structure of a strongly homotopy Lie algebra.)

In the next theorem, we consider the Maurer--Cartan space $\DGMC_{\bullet}(L)$ associated to the Lie dg-algebra $L = \HGC_{m,n}$.
This simplicial set $\DGMC_{\bullet}(L)$ is defined by the sets
of flat $L$-valued PL connections on the simplices $\Delta^n$, $n\in\NN$.
To be more precise, in the definition of this object $\DGMC_{\bullet}(L)$, we generally assume that $L$ forms a complete Lie dg-algebra
with respect to a filtration $L = \DGF_1 L\supset\DGF_2 L\supset\cdots\supset\DGF_k L\supset\cdots$
such that $[\DGF_k L,\DGF_l L]\subset\DGF_{k+l} L$.
In the case $L = \HGC_{m n}$, we assume that $\DGF_k L = \DGF_k\DGH_{m n}$
consists of power series of graphs $\gamma\in\DGH_{m n}$
such that $e - v\geq k$, where $e$ denotes the number of edges and $v$ denotes the number of internal vertices
in $\gamma$.
Then we explicitly set:
\begin{equation}\label{eqn:mapping-spaces-homotopy:hairy-graph-complex:MaurerCartan-space}
\DGMC_n(L) = \bigl\{\omega\in(L\hat{\otimes}\DGOmega^*(\Delta^n))^1\,|\,\delta(\omega) + \frac{1}{2}[\omega,\omega] = 0\}\bigr\},
\end{equation}
for every simplicial dimension $n\in\NN$, where $(L\hat{\otimes}\DGOmega^*(\Delta^n))^1$ denotes the component of upper degree $1$
in the completed tensor product of the Lie dg-algebra $L$ with the Sullivan cochain dg-algebra
of PL forms $\DGOmega^*(\Delta^n)$. The face and degeneracy operators of this simplicial set
are inherited from the simplices.
This construction has a natural extension for $L_{\infty}$-algebras (see for instance~\cite{GetzlerMC}).
We now have the following main result:

\begin{thm}[{B. Fresse, V. Turchin, and T. Willwacher~\cite[Theorem 1]{FresseTurchinWillwacher}}]\label{thm:mapping-spaces-homotopy:hairy-graph-model}
For any $n\geq m\geq 2$, we have the relation:
\begin{equation*}
\Map_{\Top\Op}^h(\DOp_m,\DOp_n^{\QQ})\sim\DGMC_{\bullet}(\HGC_{m n}),
\end{equation*}
where $\HGC_{m n}$ is the hairy graph complex. This relation extends to the case $n>m=1$
when we equip $\HGC_{1 n}$ equipped with the Shoikhet $L_{\infty}$-structure.
\end{thm}

\begin{proof}[Proof (outline)]
The results of the previous section imply that we have the following weak-equivalences:
\begin{align}
\Map_{\Top\Op}^h(\DOp_m,\DOp_n^{\QQ})
& \sim\Map_{\dg^*\Hopf\Op_{*1}^c}^h(\DGR\DGOmega_{\sharp}^*(\DOp_n),\DGR\DGOmega_{\sharp}^*(\DOp_m))\\
& \sim\Map_{\dg^*\Hopf\Op_{*1}^c}^h(\PoisOp_n^c,\PoisOp_m^c),
\end{align}
where we use the Quillen adjunction between the functors $\DGG_{\bullet}(-)$ and $\DGOmega_{\sharp}^*(-)$ in the first equivalence
and the formality result of Theorem~\ref{thm:rational-homotopy:en-operads-formality}
in the second equivalence.

Thus, we can reduce our goal to the problem of computing the homotopy of the mapping space $\Map_{\dg^*\Hopf\Op_{*1}^c}^h(\PoisOp_n^c,\PoisOp_m^c)$
in the category of Hopf cochain dg-cooperads $\dg^*\Hopf\Op_{*1}^c$
in order to reach the objective of this theorem. We use several steps to reduce this computation
to the hairy graph complex.

\textit{Step 1}. To determine the mapping space $\Map_{\dg^*\Hopf\Op_{*1}^c}^h(\PoisOp_n^c,\PoisOp_m^c)$,
we have to pick a cofibrant resolution of the source object $\PoisOp_n^c$
and a fibrant resolution of the target object $\PoisOp_m^c$.
For this purpose, we consider the Chevalley--Eilenberg complex of the complete graded Drinfeld--Kohno Lie algebra operad $\DGC_{CE}^*(\hat{\palg}_n)$,
whose definition is explained in~\S\ref{rational-homotopy:graded-drinfeld-kohno-lie-algebra-operad},
and an analogue of the Boardman-Vogt $W$-construction for Hopf cochain dg-cooperads.
Recall simply that the standard Boardman-Vogt $W$-construction is a cofibrant resolution functor on the category of operads.
Thus, the cooperadic analogue of this construction, which we consider in our computation, is a fibrant resolution functor $\DGW^c(-)$
on the category of Hopf cochain dg-cooperads $\dg^*\Hopf\Op_{*1}^c$,
and we just take the image of the cooperad $\PoisOp_m^c$ under this functor $\DGW^c(-)$ to get a fibrant model of this object
in the category of Hopf cochain dg-cooperads
Then we eventually get the identity:
\begin{equation}
\Map_{\dg^*\Hopf\Op_{*1}^c}^h(\PoisOp_n^c,\PoisOp_m^c) = \Map_{\dg^*\Hopf\Op_{*1}^c}(\DGC_{CE}^*(\hat{\palg}_n),\DGW^c(\PoisOp_m^c)).
\end{equation}
We also have to fix a simplicial framing of the Boardman-Vogt $W$-construction
in order to determine our mapping space
as a morphism set of Hopf cochain dg-cooperads:
\begin{equation}
\Map_{\dg^*\Hopf\Op_{*1}^c}(\DGC_{CE}^*(\hat{\palg}_n),\DGW^c(\PoisOp_m^c)) = \Mor_{\dg^*\Hopf\Op_{*1}^c}(\DGC_{CE}^*(\hat{\palg}_n),\DGW^c(\PoisOp_m^c)^{\Delta^{\bullet}}).
\end{equation}
In~\cite[\S 6]{FresseTurchinWillwacher}, we define such a simplicial framing $\DGW^c(\PoisOp_m^c)^{\Delta^{\bullet}}$
by considering extra PL form factors $\DGOmega^*(\Delta^{\bullet})$
in the definition of our $W$-construction
for Hopf cochain dg-cooperads. We refer to this article for a detailed
account of this construction.

\textit{Step 2}. The $W$-construction $\DGW^c(\KOp)$ of a Hopf cochain dg-cooperad $\KOp$ is cofree as cooperad,
and similarly as regards the simplicial framing $\DGW^c(\KOp)^{\Delta^{\bullet}}$
which we consider in our construction.
To be explicit, we have an identity of the form $\DGW^c(\KOp)^{\Delta^{\bullet}}_{\flat} = \FreeOp^c(\QQ[-1]\otimes\mathring{\DGW}{}^c(\KOp)^{\Delta^{\bullet}})$
when we forget about commutative algebra structures and differentials,
where $\FreeOp^c(-)$ denotes the cofree cooperad functor.
Recall that the Chevalley--Eilenberg complex $\DGC_{CE}^*(\hat{\galg})$ of a complete Lie dg-algebra $\hat{\galg}$
also forms a symmetric algebra when we forget about differentials,
since we have by definition $\DGC_{CE}^*(\hat{\galg}) = (\Sym(\QQ[-1]\otimes\hat{\galg}^{\vee}),\partial)$.
Then, for our mapping space, we have a weak-equivalence:
\begin{equation}
\Map_{\dg\Hopf\Op_{*1}^c}(\DGC_{CE}^*(\hat{\palg}_n),\DGW^c(\PoisOp_m^c))
\sim\DGMC_{\bullet}(\BiDer(\DGC_{CE}^*(\hat{\palg}_n),\DGW^c(\PoisOp_m^c))),
\end{equation}
where $\BiDer(\DGC_{CE}^*(\hat{\palg}_n),\DGW^c(\PoisOp_m^c))$ consists of maps $\theta: \DGC_{CE}^*(\hat{\palg}_n)\rightarrow\DGW^c(\PoisOp_m^c)$
that are coderivations with respect to the cooperad composition coproducts
and derivations with respect to the commutative algebra structure
of our objects.
The idea is that a morphism $\phi_f: \DGC_{CE}^*(\hat{\palg}_n)\rightarrow\DGW^c(\PoisOp_m^c)^{\Delta^{\bullet}}$
is determined by a collection of equivariant maps $f: \hat{\palg}_n(r)^{\vee}\rightarrow\mathring{\DGW}{}^c(\PoisOp_m^c)^{\Delta^{\bullet}}(r)$, $r\in\NN$,
on the generating and cogenerating objects
of our Hopf cochain dg-cooperads $\DGC_{CE}^*(\hat{\palg}_n)$ and $\DGW^c(\PoisOp_m^c)^{\Delta^{\bullet}}$,
and we have a parallel result for the biderivations $\theta: \DGC_{CE}^*(\hat{\palg}_n)\rightarrow\DGW^c(\PoisOp_m^c)$.
To define the Maurer--Cartan space of our formula, we use that this object $\BiDer(\DGC_{CE}^*(\hat{\palg}_n),\DGW^c(\PoisOp_m^c))$
can be equipped with an $L_{\infty}$-algebra structure
which reflects the constraints that these maps $f: \hat{\palg}_n(r)^{\vee}\rightarrow\mathring{\DGW}{}^c(\KOp)^{\Delta^{\bullet}}(r)$, $r\in\NN$,
have to satisfy in order to ensure that the corresponding Hopf cochain dg-cooperad morphism $\phi_f$
preserves the differentials attached to our objects (we refer to~\cite[\S 3 and \S 6]{FresseTurchinWillwacher} for details).

\textit{Step 3}. Thus, we are now left to computing $\BiDer(\DGC_{CE}^*(\palg_n),\DGW^c(\PoisOp_m^c))$
as an $L_{\infty}$-algebra.

We have not been precise about the structure, underlying a cooperad, which we consider when we assert that the $W$-construction $\DGW^c(\PoisOp_m^c)$
forms a cofree object.
We actually consider the category of diagrams over the category $\Lambda_{>1}$
which has the ordinals $\underline{r} = \{1,\dots,r\}$ such that $r>1$
as objects and the injective maps between such ordinals $u: \{1,\dots,k\}\rightarrow\{1,\dots,l\}$
as morphisms.
Wa also assume that our diagrams are equipped with a coaugmentation over the constant diagram such that $\CstOp(r) = \QQ$,
for every $r>1$.
We can regard such a diagram as a collection $\MOp = \{\MOp(r),r>1\}$ whose terms $\MOp(r)$ are equipped with an action
of the symmetric groups $\Sigma_r$, $r>1$, and with coface operators $\partial^i: \MOp(r-1)\rightarrow\MOp(r)$
which correspond to the injective maps $\partial_i: \{1,\dots,r-1\}\rightarrow\{1,\dots,r\}$
such that $(\partial_i(1),\dots,\partial_i(i-1),\partial_i(i),\dots\partial_i(r-1)) = (1,\dots,i-1,i+1,\dots,r)$.
We can also determine the coaugmentation over the constant diagram by giving a collection of morphisms $\eta: \QQ\rightarrow\MOp(r)$,
whose image is left invariant under the symmetric structure and the action of the coface operators
on our object $\MOp$.

To a cooperad $\COp\in\dg^*\Op_{*1}^c$, we can associate the collection $U\COp = \{\COp(r),r>1\}$
equipped with the coface operators $\partial^i: \COp(r-1)\rightarrow\COp(r)$
given by the composition coproducts $\circ_i^*: \COp(r-1)\rightarrow\COp(r)\otimes\COp(0)$
associated to our object,
whereas the coaugmentations $\eta: \QQ\rightarrow\COp(r)$ represent the image of the factor $\QQ = \COp(0)$ into $\COp(r)$
which we may obtain by applying an $r-1$-fold composite of such composition coproducts
on the term $\COp(0)$ of our cooperad.
The cofree cooperad $\FreeOp^c(-)$ which we consider in our study forms a right adjoint
of this forgetful functor $U: \COp\mapsto U\COp$.
In what follows, we also deal with an extension of this category of coaugmented $\Lambda$-diagrams
in order to allow a term in arity one.
In all cases, we use the notation $\Hom_{\dg\Lambda\Mod}(\MOp,\NOp)$ for an enriched hom-object
associated to our category of coaugmented diagrams
in the category of dg-modules.
The correspondence between the biderivations $\theta: \DGC_{CE}^*(\hat{\palg}_n)\rightarrow\DGW^c(\PoisOp_m^c)$
and the collections of equivariant maps $f: \hat{\palg}_n(r)^{\vee}\rightarrow\mathring{\DGW}{}^c(\PoisOp_m^c)(r)$, $r\in\NN$,
which we use in the previous step
is equivalent to an isomorphism of dg-modules:
\begin{equation}
\BiDer(\DGC_{CE}^*(\palg_n),\DGW^c(\PoisOp_m^c))\simeq\Hom_{\dg\Lambda\Mod}(\hat{\palg}_n,\mathring{\DGW}{}^c(\PoisOp_m)),
\end{equation}
where we consider this enriched dg-hom object on the category of coaugmented $\Lambda$-diagrams $\Hom_{\dg\Lambda\Mod}(-,-)$
on the right hand side.

To go further, we use that the cogenerating object of our $W$-construction $\mathring{\DGW}{}^c(\KOp)$
is quasi-isomorphic to an operadic version of the classical cobar construction
of coalgebras. For the $m$-Poisson cooperad $\KOp = \PoisOp_m^c$,
we also have a Koszul duality result which asserts that this cobar construction $\DGB^c(\PoisOp_m^c)$
satisfies the quasi-isomorphism relation $\DGB^c(\PoisOp_m^c)\sim\LambdaOp^m\PoisOp_m$,
where $\LambdaOp$ denotes an operadic suspension operation.
Then, at the enriched dg-hom object level, we have a chain of quasi-isomorphisms
\begin{multline}
\Hom_{\dg\Lambda\Mod}(\hat{\palg}_n^{\vee},\mathring{\DGW}{}^c(\PoisOp_m^c))
\sim\Hom_{\dg\Lambda\Mod}(\hat{\palg}_n^{\vee},\DGB^c(\PoisOp_m^c))\\
\sim\Hom_{\dg\Lambda\Mod}(\hat{\palg}_n^{\vee},\LambdaOp^m\PoisOp_m),
\end{multline}
which is induced by this chain of weak-equivalences $\mathring{\DGW}{}^c(\PoisOp_m^c)\sim\DGB^c(\PoisOp_m^c)\sim\LambdaOp^m\PoisOp_m$
between our objects.

\textit{Step 4}. To get graphs, we use that the weak-equivalence relations $\GraphOp_n^c\sim\PoisOp_n^c\sim\DGC_{CE}^*(\hat{\palg}_n)$
in the category of Hopf dg-cooperads give rise a weak-equivalence relation $\ICG_n^c\sim\hat{\palg}_n^{\vee}$
in our category of coaugmented $\Lambda$-diagrams in dg-modules
when we pass to the generating objects of the symmetric algebras $\GraphOp_n^c(r) = (\Sym(\QQ[-1]\otimes\ICG_n^c(r)),\partial)$
and $\DGC_{CE}^*(\hat{\palg}_n(r)) = (\Sym(\QQ[-1]\otimes\hat{\palg}_n(r)^{\vee}),\partial)$.
This weak-equivalence gives rise to a quasi-isomorphism:
\begin{equation}
\Hom_{\dg\Lambda\Mod}(\hat{\palg}_n^{\vee},\LambdaOp^m\PoisOp_m)\sim\Hom_{\dg\Lambda\Mod}(\ICG_n^c,\LambdaOp^m\PoisOp_m)
\end{equation}
at the level of our dg-hom objects.
Let $\mu_r = \mu_r(x_1,\dots,x_r)$ denote the element such that $\mu_r(x_1,\dots,x_r) = x_1\cdots x_r$
in the $m$-Poisson operad $\PoisOp_m$. We transport this element to the suspension $\LambdaOp^m\PoisOp_m$
by using that this operad is just defined by a appropriate degree shift of the $m$-Poisson operad
in each arity and we consider the natural map
\begin{equation}
\Phi: \HGC_{m n}\rightarrow\Hom(\ICG_n^c,\LambdaOp^m\PoisOp_m)
\end{equation}
such that $\Phi(\gamma): \alpha\mapsto\langle\gamma,\alpha\rangle \mu_r$ for any graph $\gamma\in\HGC_{m n}$,
where $\langle-,-\rangle$ is the obvious pairing between the hairy graph and the internally connected graphs
with univalent external vertices in $\ICG_n^c$.
To be explicit, for a hairy graph $\gamma\in\HGC_{m n}$, we set $\langle\gamma,\alpha\rangle = \pm 1$
if $\alpha\in\ICG_n^c(r)$ is a graph with univalent external vertices of the same shape as $\gamma$,
and $\langle\gamma,\alpha\rangle = 0$ otherwise.

We check that the composite of this map with the mappings of the previous steps gives a quasi-isomorphism of $L_{\infty}$-algebras
\begin{equation}
\HGC_{m n}\sim\BiDer(\DGC_{CE}^*(\palg_n),\DGW^c(\PoisOp_m^c))
\end{equation}
and the conclusion of the theorem follows.
\end{proof}

The result of this theorem has the following corollary:

\begin{cor}\label{cor:mapping-spaces-homotopy:homotopy-groups}
For any $n\geq m\geq 2$ (and for $n>m=1$), we have the identity:
\begin{equation*}
\pi_*(\Map_{\Top\Op}^h(\DOp_m,\DOp_n^{\QQ}),\omega) = H_{*-1}(\HGC_{m n}^{\omega}),
\end{equation*}
for any $\omega\in\DGMC_0(\HGC_{m n})$, where $\HGC_{m n}^{\omega}$ is the complex $\HGC_{m n}$
equipped with the twisted differential $\delta_{\omega} = \delta + [\omega,-] + (\text{extra terms in the $L_{\infty}$-case})$.
\end{cor}

\begin{proof}[Explanations]
The identity of this statement follows from the result of Theorem~\ref{thm:mapping-spaces-homotopy:hairy-graph-model}
and from a general result about the homotopy groups of Maurer--Cartan spaces $\DGMC_{\bullet}(L)$
for which we refer to~\cite{Berglund}.
\end{proof}

Let us mention that a computation of the rational homotopy groups of the embedding spaces $\overline{\Emb}_c(\RR^m,\RR^n)$,
analogous to the result established in this corollary, is given in~\cite{AroneTurchinGraphs}
(see also~\cite{LambrechtsTurchin} for the case $m=1$ of these computations).
These previous computations are based on the interpretation
in terms of mapping spaces of operadic bimodules
of the Goodwillie--Weiss tower of the embedding spaces $\overline{\Emb}_c(\RR^m,\RR^n)$
(or of the equivalent interpretation of the Goodwillie--Weiss tower in terms of Sinha's cosimplicial model in the case $m=1$).
In~\cite{LambrechtsTurchinVolicHomology}, the formality of $E_n$-operads in chain complexes
is also used to get a description of the homology of the embedding spaces $\overline{\Emb}_c(\RR^1,\RR^n)$
in terms of a Hochschild cohomology theory for operads (we apply this Hochschild cohomology theory to the $n$-Poisson operad).
The graph operad model of the $n$-Poisson can also be used to deduce a graph complex model
of the homology of the embedding space $\overline{\Emb}_c(\RR^1,\RR^n)$
from this algebraic approach.

In fact, we can use the result of the above corollary and the equivalence between the embedding space $\overline{\Emb}_c(\RR^m,\RR^n)$
and the $m+1$-fold iterated loop space of the operadic mapping space $\Map_{\Top\Op}^h(\DOp_m,\DOp_n)$
given in Theorem~\ref{thm:background:embedding-operad-model}
to get applications of the result of Theorem~\ref{thm:mapping-spaces-homotopy:hairy-graph-model}
in the theory embedding spaces. For this purpose, we also use the following theorem:

\begin{thm}[{B. Fresse, V. Turchin, and T. Willwacher~\cite[Theorem 15]{FresseTurchinWillwacher}}]\label{cor:mapping-spaces-homotopy:rationalization}
In the case $n-m\geq 3$, the space $\Map_{\Top\Op}^h(\DOp_m,\DOp_n)$ is $n-m-1$ connected, and we moreover have the relation:
\begin{equation*}
\Map_{\Top\Op}^h(\DOp_m,\DOp_n)^{\QQ}\sim\Map_{\Top\Op}(\DOp_m,\DOp_n^{\QQ})
\end{equation*}
in the homotopy category of spaces.
\end{thm}

\begin{proof}
We refer to the cited reference~\cite[\S 10]{FresseTurchinWillwacher} for the proof of this statement.
Let us simply mention that an analogous result for spaces is established by Haefliger in~\cite{Haefliger}.
To establish the result of this theorem
we use the tower decomposition $\Map_{\Top\Op}^h(\DOp_m,\DOp_n) = \lim_k\Map_{\Top\Op^{\leq k}}^h(\DOp_m,\DOp_n)$,
where we again consider the truncated category of operads that are defined up to arity $k$.
The idea is to examine the fiberwise rational homotopy of the fibers of this tower in order to reduce the verification of this theorem
to the results obtained by Haefliger in the context of spaces. (We also refer to \cite{Goeppl} for a study of a model
of this tower in the category of dendroidal sets.)
\end{proof}

We examine the rational homotopy of the spaces of homotopy automorphisms to complete the result of this section.
We first explain the definition of the Kontsevich graph complexes $\GC_n$
which occur in this computation.

\subsubsection{The Kontsevich graph complex}\label{mapping-spaces-homotopy:kontsevich-graph-complex}
The definition of the complex $\GC_n$ is the same as the definition of the hairy graph complex $\HGC_{m n}$,
except that we now consider graphs without hairs,
as in the following examples:
\begin{equation}\label{mapping-spaces-homotopy:kontsevich-graph-complex:picture}
\begin{tikzpicture}[baseline=-.65ex, scale=.5]
\node[int] (c) at (0,0){};
\node[int] (v1) at (0:1) {};
\node[int] (v2) at (72:1) {};
\node[int] (v3) at (144:1) {};
\node[int] (v4) at (216:1) {};
\node[int] (v5) at (-72:1) {};
\draw (v1) edge (v2) edge (v5) (v3) edge (v2) edge (v4) (v4) edge (v5)
      (c) edge (v1) edge (v2) edge (v3) edge (v4) (c) edge (v5);
\end{tikzpicture},
\quad\quad\begin{tikzpicture}[baseline=-.65ex, scale=.5]
\node[int] (c) at (0.7,0){};
\node[int] (v1) at (0,-1) {};
\node[int] (v2) at (0,1) {};
\node[int] (v3) at (2.1,-1) {};
\node[int] (v4) at (2.1,1) {};
\node[int] (d) at (1.4,0) {};
\draw (v1) edge (v2) edge (v3)  edge (d) edge (c) (v2) edge (v4) edge (c) (v4) edge (d) edge (v3) (v3) edge (d) (c) edge (d);
\end{tikzpicture}.
\end{equation}
We determine the degree of a graph in $\GC_n$ by assuming that each vertex contributes by $\deg(\bullet) = n$
and each edge contributes by $\deg(\bullet\!\!-\!\!\bullet) = 1-n$
as in the case of hairy graphs.
We still assume that every vertex of a graph in $\GC_n$ is at least trivalent and we do not allow loops (edges with the same origin
and endpoint). The differential is defined by the blow-up of vertices again.

The space of homotopy automorphisms $\Aut_{\Top\Op}^h(\DOp_n^{\QQ})$ is the sum of the connected components
of the mapping spaces $\Map_{\Top\Op}^h(\DOp_n,\DOp_n)$
associated to the morphisms $\phi$ which are invertible in the homotopy category of operads.
Let $h: \Aut_{\Top\Op}^h(\DOp_n^{\QQ})\rightarrow\Aut_{\Hopf\Op}(\DGH_*(\DOp_n,\QQ))$
be the natural map which carries any such morphism
to the associated homology morphism.
For $n\geq 2$, we have a bijection $\Aut_{\Hopf\Op}(\DGH_*(\DOp_n,\QQ)) = \QQ^{\times}$
which is determined by taking the action of an automorphism $\phi\in\Aut_{\Hopf\Op}(\DGH_*(\DOp_n,\QQ))$
on the representative of the Poisson bracket operation $\lambda\in\PoisOp_n$
in the operad $\PoisOp_n = \DGH_*(\DOp_n)$. We get the following result:

\begin{thm}[{B. Fresse, V. Turchin, and T. Willwacher~\cite[Corollary 5]{FresseTurchinWillwacher}}]\label{thm:mapping-spaces-homotopy:automorphisms-kontsevich-graph-complex-model}
For each $\lambda\in\QQ^{\times}$, we have the identity:
\begin{equation*}
\pi_*(h^{-1}(\lambda)) = \DGH_*(\GC_n)\oplus\begin{cases} \QQ, & \text{if $*\equiv -n-1(4)$}, \\
0, & \text{otherwise}, \end{cases}
\end{equation*}
where $\GC_n$ denotes the Kontsevich graph complex.
\end{thm}

\begin{proof}[Explanation and references]
We deduce this statement from the result of Theorem~\ref{thm:mapping-spaces-homotopy:hairy-graph-model},
by using that the identity morphism is represented by the Maurer--Cartan element such that $\omega = |$
in the hairy graph complex $\HGC_{n n}$.
We just consider versions of the graph complexes $\GC_n^2$ and $\HGC_{m n}^2$
where bivalent vertices are allowed.
We have $\HGC_{m n}^2\sim\HGC_{m n}$ for any $n\geq m\geq 2$, whereas for the graph complex $\GC_n^2$, we have:
\begin{equation}
\DGH_*(\GC_n^2) = \DGH_*(\GC_n)\oplus\begin{cases} \QQ L_*, & \text{if $*\equiv -n-1(4)$}, \\
0, & \text{otherwise}, \end{cases}
\end{equation}
where $L_*$ denotes the homology classes of graphs
of the form:
\begin{equation}
L_* = \begin{tikzpicture}[baseline=-.65ex]
\node[int] (v1) at (0:1) {};
\node[int] (v2) at (72:1) {};
\node[int] (v3) at (144:1) {};
\node[int] (v4) at (216:1) {};
\node (v5) at (-72:1) {$\cdots$};
\draw (v1) edge (v2) edge (v5) (v3) edge (v2) edge (v4) (v4) edge (v5);
\end{tikzpicture}
\end{equation}
We easily see that the operation $[\omega,-]$ in the differential $\delta_{\omega} = \delta + [\omega,-]$
of the twisted complex $\HGC_{n n}^{2\omega}$
associated to the Maurer--Cartan element $\omega = |$
is given by the addition of a hair $|$ to any graph $\gamma\in\HGC_{n n}^2$.
We can then use a spectral sequence to check that we have a quasi-isomorphism $\QQ|\oplus\QQ[-1]\otimes\GC_n^2\xrightarrow{\sim}\HGC_{n n}^2$
where we consider the mapping $\gamma\mapsto\gamma-$
which associates a graph with one hair $\gamma-\in\HGC_{n n}^2$ to any graph $\gamma\in\GC_n$.
We refer to~\cite[Proposition 2.2.9]{FresseWillwacher} for the detailed line of arguments.
\end{proof}

In the case $n=2$, we have $\DGH_0(\GC_n) = \grt$ by a result of T. Willwacher,
where $\grt$ is the graded Grothendieck--Teichm\"uller Lie algebra (see~\cite{WillwacherGraphs}).
Therefore, in this case, the result of Theorem~\ref{thm:mapping-spaces-homotopy:automorphisms-kontsevich-graph-complex-model}
reflects the relation that we obtained in Theorem~\ref{thm:e2-operads:rational-homotopy-automorphisms}:
\begin{equation*}
\Aut_{\Top\Op}^h(\DOp_2^{\QQ})\sim\GT(\QQ)\ltimes\SO(2)^{\QQ},
\end{equation*}
where $\GT(\QQ)$ is the Grothendieck--Teichm\"uller group.

\section{Outlook}\label{sec:outlook}

Throughout this survey, we have focused on the study of the homotopy of $E_n$-operads themselves.
But one can use variants of the definition of an $E_n$-operad
to associate operadic right module structures
to any $n$-manifold $M$.
To be specific, the Fulton--MacPherson compactifications considered in~\S\ref{background:fulton-macpherson-operad}
have a natural generalization for the configuration spaces of manifolds $\FOp(M,r)$,
and when $M$ is a framed manifold, this construction returns a collection of spaces $\FMOp_M = \{\FMOp_M(r),r\in\NN\}$
which inherits the structure of a right module (in the operadic sense)
over the Fulton--MacPherson operad $\FMOp_n$.

This object is equivalent to constructions used by Ayala--Francis
in the definition
of the factorization homology of manifolds (see Ayala--Francis's paper, in this handbook volume,
for a survey of this subject).
In particular, one can use a relative composition product of the object $\FMOp_M$ over the operad $\FMOp_n$
to compute the factorization homology of any framed manifold $M$.
The methods used by Kontsevich to prove the formality of $E_n$-operads have been used by several authors to define models
of the rational homotopy type of this right module $\FMOp_M$,
and hence, to tackle the rational homotopy computations in factorization homology theory. To cite a few works on this subject,
let us mention that a graph complex model of the object $\FMOp_M$,
which extends the graph cooperad of~\S\ref{rational-homotopy:graph-cooperad}
when $M$ is a simply connected compact manifold without boundary,
is defined by Campos--Willwacher in~\cite{CamposWillwacher},
while an extension of Arnold's presentation is used by Idrissi in~\cite{Idrissi}
to get a small model of the object $\FMOp_M$. Idrissi's result provides a generalization
of Knudsen's description of the factorization homology for higher enveloping algebras
of Lie algebras~\cite{Knudsen}.
The paper~\cite{CamposIdrissiLambrechtsWillwacher} provides an extension of the constructions of~\cite{CamposWillwacher,Idrissi}
for manifolds with boundary, while the paper~\cite{CamposDucoulombierIdrissiWillwacher} addresses an extension
of the definition of these operadic module structures by using a framed version of the operads
of little discs.

In~\S\S\ref{sec:rational-homotopy}-\ref{sec:mapping-spaces-homotopy}, we entirely focus on the rational homotopy theory framework,
but we may wonder which information we may still retrieve by our methods
in positive characteristic.
For instance, partial formality results have been obtained by Cirici--Horel in~\cite{CiriciHorel}
when we take an arity-wise truncation of operads below the characteristic
of the coefficients.
In fact, the $E_n$-operads are not formal as symmetric operads in chain complexes in positive characteristic,
because their components are not formal as representations of the symmetric groups.
Nevertheless, we may wonder whether $E_n$-operads
are formal as non-symmetric operads, which is enough for the study of mapping spaces
over an $E_1$-operad.
The case $n>2$ of this question is still open, but Salvatore has proved in~\cite{SalvatoreNonFormality}
that $E_2$-operads are not formal as non-symmetric operads
over $\FF_2$.

\bibliographystyle{plain}
\bibliography{EnOperadHomotopy}

\end{document}